%% file: main.tex
\documentclass[10pt]{article}

\usepackage{url}
\usepackage{tikz-cd}
\tikzcdset{every diagram/.append style={row sep=huge, column sep=large}, every label/.append style={font=\normalsize}}
\usepackage[english]{babel}
\usepackage{latexsym,amsmath,amssymb}
\usepackage[
  colorlinks=true  % Remove the boxes
, linktocpage=true % Make page numbers (not section titles) links in ToC
, linkcolor=blue    % Colour for internal links
, citecolor=blue  % Colour for bibliographical citations
, urlcolor=blue % Colour for (external) urls{hyperref}
]{hyperref}
\usepackage[pdftex]{pict2e}
\usepackage{macros}
\usepackage{cleveref}

\newcommand{\FS}{\mathrm{Seq}}

\begin{document}

% \title{A universal notion of line segment}
\title{Euclidean interval objects \\ in categories with finite products}
% {A universal notion of real line}
% generalized abstract line segments
\author{Mart\'{\i}n H. Escard{\'o} \\[0.5ex]
  {\small School of Computer Science, University of Birmingham, UK} \\[1ex]
 Alex Simpson \\[0.5ex]
{\small Faculty of Mathematics and Physics (FMF), University of Ljubljana, Slovenia} \\
{\small Institute for Mathematics, Physics and Mechanics (IMFM), Ljubljana, Slovenia}
\\[1.2ex]
Dedicated to the memory of Phil Scott (1947 -- 2023)
}

\maketitle

\input{introduction}
\input{notions}

\input{set}
\input{top}
\input{topos}

\input{open}

\bibliographystyle{alphaurl}
\bibliography{references}

\newpage

%  {
%   \small
%   \tableofcontents
%   \clearpage
%  }

\tableofcontents

\end{document}

%% file: introduction.tex
\begin{abstract}
  Based on the intuitive notion of convexity, we formulate a universal
  property defining \emph{interval objects} in a category with finite
  products. Interval objects are structures corresponding to closed
  intervals of the real line, but their definition does not assume a
  pre-existing notion of real number.  The universal property
  characterises such structures up to isomorphism, supports the
  definition of functions between intervals, and provides a means of
  verifying identities between functions.  In the category of sets,
  the universal property characterises closed intervals of real
  numbers $[a,b]$ with endpoints $a < b$.  In the category of
  topological spaces, we obtain intervals with the Euclidean topology.
  We also prove that every elementary topos with natural numbers
  object contains an interval object; furthermore, we characterise
  interval objects as intervals $[a,b]$ of real numbers in the Cauchy
  completion of the rational numbers within the Dedekind reals.

%[[TO WRITE]]
%N  We propose a universal notion of line segment and investigate it in
%  various mathematical and foundational settings, including set
%  theory, general topology and topos theory.
%  We define a notion of interval object, formulated via a universal
%  property, that applies to any category with finite products. The
%  intention is to capture the intuitive geometric idea of closed and
%  bounded line segment. In the categories of sets and topological
%  spaces, the expected objects are identified. In an elementary topos,
%  there are two competing notions of real-number object, and hence of
%  interval, and ours give rise two a third that lies in between.

%  Our definition can be summarized thus. Out of the given category
%  with finite products, we construct a new category whose objects
%  capture a strong notion of convexity and whose morphisms capture a
%  strong notion of affine map. Then an interval object is defined to
%  be a free strongly convex object over two generators, which
%  correspond to its endpoints. Hence, when it exists, it is unique up
%  to isomorphism. In the category of sets, our strong notion of
%  convexity corresponds to what is known as superconvexity, together
%  with the cancellation property. But our formulation applies to any
%  category with finite products and hence does not presuppose an
%  interval object.
\end{abstract}

\section{Introduction}
\label{SecIntro}

We define and study a notion of \emph{interval object}, in any category
with finite products, by a universal property (Section~\ref{notions})\footnote{This
  paper is a full version of the extended
  abstract~\cite{escardo:simpson:interval:lics}.}.  In the category of
sets, interval objects are given by closed intervals of real
numbers. In topological spaces, they are closed intervals with the
Euclidean topology.  In other categories embodying principles of
intuitionistic, constructive or computable mathematics, the interval
objects are closed intervals within a notion of real number
appropriate to the setting.
%Similarly, in categories of computable functions, an interval
% object is a closed interval of computable reals.
Importantly, our notion
of interval object is defined without reference to any notion of
real number. Indeed, the definition %  of interval object
makes sense in any category with finite products (although, of course, an
interval object is not guaranteed to exist). Thus we achieve a
single general definition of interval object which applies
uniformly to characterise closed intervals of real numbers
across the diverse range of example categories we have mentioned.

%computable or
%constructive mathematics, one obtains repspectively %correspondingly
%closed intervals of computable or constructive

Our definition of interval object is based on an analysis
of geometric properties of Euclidean space. In this analysis, we take convexity
as a primitive concept, and we axiomatize
the properties of a few simple geometric constructions that can be performed on
an arbitrary bounded convex subset of Euclidean
space (of any dimension).
However, to best appreciate the development that follows,
the  notion of bounded convex set
should at first  be understood intuitively, not (yet) mathematically.
%Indeed,
%the reader should forget his or her mathematical sophistication,
%and instead consider the notion of bounded convex object
%as it might have been understood by a classical Greek geometer.
Thus we think of a bounded convex set as a bounded
geometric figure $A$ satisfying the property that,
between any two points $x,y$ in $A$, the unique line segment
between $x$ and $y$ also resides in $A$.
\newcommand{\pointp}{\circle*{16}}
\newcommand{\pointpp}{\circle*{8}}
\newcommand{\pointppp}{\circle*{4}}
\newcommand{\pointpppp}{\circle*{2}}
\begin{center}
\setlength{\unitlength}{0.006cm}
\begin{picture}(1200,700)
\polyline(0,0)(300,600)(1100,0)(0,0)%
\put(200,250){\pointp}
\put(700,100){\pointp}
\put(180,200){\text{$x$}}
\put(680,50){\text{$y$}}
\polyline(200,250)(700,100)
\end{picture}
\end{center}

The task we now undertake is to axiomatize properties of primitive
geometric constructions supported by such objects. Having done
this, we shall take these axiomatized properties as the
{defining properties} of a precise mathematical
notion, the notion of \emph{midpoint-convex body}.
This notion will then be used
% , in turn, to obtain associated notions
to define the notion of \emph{interval object}.
%and \emph{real line object}.
%The latter provides  the general notion of real number structure
%adumbrated above.

The first geometric construction we consider is the most basic
of all, the bisection of a line.
Given, any two points $x,y$ inside an object $A$,
we bisect the line between them to obtain the midpoint $m(x,y)$.
\begin{center}
\setlength{\unitlength}{0.006cm}
\begin{picture}(1200,700)
\polyline(0,0)(300,600)(1100,0)(0,0)%
\put(200,250){\pointp}
\put(700,100){\pointp}
\put(450,175){\pointp}
\put(180,200){\text{$x$}}
\put(270,105){\text{$m(x,y)$}}
\put(680,50){\text{$y$}}
\polyline(200,250)(700,100)
\end{picture}
\end{center}
By convexity, the midpoint $m(x,y)$ lies in $A$. Thus we have
a midpoint operation $m \colon A \times A \to A$.
This operation satisfies three equations:
\begin{align*}
& \textbf{(Idempotency)} & m(x,x) & =x \, , \\
& \textbf{(Commutativity)} & m(x,y)& =m(y,x) \, , \\
& \textbf{(Transposition)} & m(m(x,y),m(z,w)) & = m(m(x,z),m(y,w)) \, .
\end{align*}
The first two
are immediately clear. The third is illustrated by the
picture below.
\begin{center}
\setlength{\unitlength}{0.006cm}
\begin{picture}(1000,700)
\polyline(150,50)(50,500)(1000,600)(700,50)(150,50)

\put(150,50){\pointp}         % x
\put(50,500){\pointp}         % y
\put(1000,600){\pointp}        % w
\put(700,50){\pointp}         % z

\put(100,275){\pointp}        % m(x,y)
\put(525,550){\pointp}        % m(y,w)
\put(850,325){\pointp}        % m(z,w)
\put(425,50){\pointp}         % m(x,z)

\put(475,300){\pointp}        % midpoint

\polyline(100,275)(850,325)
\polyline(425,50)(525,550)

\put(100,20){\text{$x$}}
\put(10,520){\text{$y$}}
\put(1015,615){\text{$w$}}
\put(730,20){\text{$z$}}

\put(-110,260){\text{$m(x,y)$}}
\put(325,580){\text{$m(y,w)$}}
\put(875,310){\text{$m(z,w)$}}
\put(375,-10){\text{$m(x,z)$}}

\end{picture}
\end{center}
% [[It would be very nice to find a verification of this equality in Euclid's Elements!!!]]

Using the midpoint operation,
given any point $x \in A$,
we can consider the derived
%
% I THINK IT READS BETTER NOT TO SAY MIDPOINT-CONVEX BODY HERE
convex set
$m(x,A)$ consisting
of all points $\{m(x,y) \mid y \in A\}$. This is illustrated
by the picture below, where $A$ is the outer triangle, and $m(x,A)$ is
the inner triangle.
\begin{center}
\setlength{\unitlength}{0.006cm}
\begin{picture}(1200,700)
\polyline(0,0)(300,600)(1100,0)(0,0)%

\put(750,150){\pointp}
\put(700,120){\text{$x$}}

\polyline(375,75)(525,375)(925,75)(375,75)
\end{picture}
\end{center}
The operation $y \mapsto m(x,y)$ gives a mapping from $A$ to
$m(x,A)$. By the (assumed) homogeneity of space, this mapping is injective.
Thus the midpoint operation satisfies the following (non-equational)
property.
\begin{align} \label{cancellation}
& \textbf{(Cancellation)}  & &
\text{$m(x,y)=m(x,z)$ implies $y=z$.}
\end{align}

Our final construction is more involved, but nonetheless geometrically
intuitive. Suppose we are given any sequence $(x_i)$ of points in $A$,
for example as in the following picture:
\begin{center}
\setlength{\unitlength}{0.006cm}
\begin{picture}(1200,700)
\polyline(0,0)(300,600)(1100,0)(0,0)%

\put(320,420){\pointp}
\put(200,100){\pointp}
\put(800,100){\pointp}

\put(300,370){\text{$x_{3n}$}}
\put(180,50){\text{$x_{3n+1}$}}
\put(780,50){\text{$x_{3n+2}$}}

\end{picture}
\end{center}
Then we can successively construct a nested sequence
of convex sets
\[\begin{array}{llll}
A_0 & = \; A & & \\
A_1 & = \; m(x_0,A) & = \; m(x'_0,A_0) & \text{where $x'_0 = x_0$,} \\
A_2 & = \; m(x_0,m(x_1,A)) & = \; m(x'_1,A_1) & \text{where $x'_1 = m(x_0,x_1)$,} \\
A_3 & = \; m(x_0,m(x_1,m(x_2,A))) & = \; m(x'_2,A_2) & \text{where $x'_2 = m(x_0,m(x_1,x_2))$,} \\
\dots
\end{array}
\]
Note that we always have $x'_i \in A_i$. The above
corresponds to the iterative construction of the following picture:
\begin{center}
\setlength{\unitlength}{0.006cm}
\begin{picture}(1200,700)
\polyline(0,0)(300,600)(1100,0)(0,0)%

%\put(320,420){\pointp}

\polyline(160,210)(310,510)(710,210)(160,210)%

%(320,420)
%\put(260,260){\pointpp}
%(560,260)

\polyline(210,235)(285,385)(485,235)(210,235)

%(290,340)
%(260,260)
%\put(410,260){\pointppp}

\polyline(310,247.5)(347.5,322.5)(447.5,247.5)(310,247.5)

%\put(350,300){\pointppp}
%(335,260)
%(410,260)

\polyline(330,273.75)(348.75,311.25)(398.75,273.75)(330,273.75)

%(350,300)
%\put(342.5,280){\pointppp}
%(380,280)

\polyline(336.25,276.88)(345.63,295.63)(370.63,276.88)(336.25,276.88)

%(346.25,290)
%(342.5,280){\pointppp}
%\put(361.13,280)

%\polyline(348.69,278.44)(353.38,287.82)(365.88,278.44)(348.69,278.44)

\put(355,283){\pointpp}

\end{picture}
\end{center}
Here each $A_{i+1}$ is obtained by first relocating
$x_i$ at its corresponding position $x'_i$ within $A_i$, and then
constructing $m(x'_i,A_i)$. As we are assuming that $A$ is bounded,
the sequence $(A_i)$ shrinks to a point, irrespective of
whether the sequence $(x_i)$ is convergent or not.
Thus the above construction determines an infinitary
\emph{iterated midpoint} operation $M\colon A^\omega \to A$.

Intuitively, the
iterated midpoint operation
is the unique solution to
\[
M(x_0,x_1,x_2, \dots) \; = \; m(x_0,\, m(x_1, \, m(x_2, \, \dots))) \, .
\]
Of course, the right-hand side does not make formal sense.
Nevertheless, the idea behind the equation is captured
by two  properties,
which characterize how
the iterated midpoint operation is derived from the binary operation.
\begin{align*}
& \textbf{(Unfolding)}
&  & M (x_0, x_1, x_2, \dots) = m(x_0, M(x_1, x_2, \dots)) \\
& \textbf{(Canonicity)}
& &
\text{If $y_0 = m(x_0,y_1)$, $y_1 = m(x_1,y_2)$, $y_2 = m(x_2,y_3)$, \dots} \\
& & &
\phantom{blah}
\text{then $y_0 = M (x_0,x_1,x_2, \dots)$}
\end{align*}
We shall collectively refer to these two axioms as the
\emph{iteration axioms}.

Having identified this structure on the intuitive concept of bounded
convex set, we now turn the tables and take the above structure as
\emph{defining} an algebraic notion of convex object. Specifically, in
the category of sets, a \emph{midpoint-convex body} (or \emph{m-convex
  body} for short) is a set $A$, together with operations
$m\colon A \times A \to A$ and $M \colon A^\omega \to A$ satisfying
the idempotency, commutativity, transposition, cancellation and
iteration axioms above.  In the technical development, we shall give
various reformulations of this notion, we shall show that each of $m$
and $M$ determine the other, and we shall show that m-convex bodies in
the category of sets are closely related to the \emph{superconvex}
sets of K\"onig~\cite{koenig:superconvex}.
% https://www.sciencedirect.com/science/article/abs/pii/S0304020808719509

How does all this connect with obtaining a characterization of the
interval? Importantly, in contrast to standard algebraic notions of
convexity, the real numbers play no part in the definition of m-convex
body. The idea is now very straightforward. We can use the above
abstract notion of convexity to \emph{define} the closed interval
$[a,b]$ as the freely-generated m-convex body on two generators, the
endpoints $a$ and $b$.
%Once we have obtained the closed interval, it is straightforward
%to use it to determine the real line, see [[???]] for details.

The definition of the interval as a free structure has
very intuitive content. As maps between m-convex bodies,
we consider functions $f\colon A \to B$ that are homomorphisms
with respect to the operations $m$ and $M$. In fact,
any homomorphism with respect to one of the operations is automatically
a homomorphism with respect to the other, see Proposition~\ref{PropSetIt} and equation~(\ref{m-from-M}) in Section~\ref{notions}.
One can think of such maps as {affine} functions (Proposition~\ref{prop-affine}).
The interval is then formally defined as follows. An \emph{interval set} is given by an m-convex body $(I,m,M)$,
together with two points $a,b \in I$, that
satisfies the
property: for any m-convex body $X$ and points $x_a,x_b \in X$,
there exists a unique homomorphism $f\colon I \to X$ with $f(a)=x_a$ and
$f(b)=x_b$. This definition determines the interval up to (homomorphic)
isomorphism. The definition is validated by the theorem (which appears as Corollary~\ref{ThmSet} in the paper)
that, for any $a < b \in \mathbb{R}$,
the structure
$([a,b], \oplus, \bigoplus)$ is an interval set,
where
\begin{align*}
x \oplus y \; & = \; \frac{x+y}{2}, &
\bigoplus (x_0,x_1,x_2, \dots) \; & = \;
\sum_{i \geq 0} 2^{-(i+1)}x_i \, .
\end{align*}
Furthermore, one can generalise the above procedure
to generate other geometric objects. For example, we shall show
that the free m-convex body over the generators
$\{0,1,\dots, n\}$ is the $n$-dimensional simplex,
see Corollary~\ref{n-simplex}.

From the discussion thus far, it is unclear
why the above approach to defining the interval
gives a general
method applicable in many mathematical settings.
This hinges upon the following fact. It
is possible to formulate all the notions used above
in an arbitrary category with finite products.
Such a formulation is given in Section~\ref{notions},
where we % successively
define the notions
of \emph{midpoint-convex body} and  \emph{interval object}
%, and \emph{real line object},
in any such category.
In the motivating development above, we merely presented these definitions
in the special case of one particular category, the category of sets.

In Section~\ref{SecIntInTop},  we consider
the category of topological spaces, and
prove that the interval object is the structure $([a,b], \oplus, \bigoplus)$, where the interval is endowed
with the Euclidean topology (Theorem~\ref{ThmTop}).
% and that the associated real line object is indeed the Euclidean line.
Then in Section~\ref{SecIntInTopos}, we consider the case of an arbitrary elementary topos with natural numbers object.
Any such topos possesses an interval object, with this being given by
the structure $([a,b], \oplus, \bigoplus)$, where $[a,b]$ is an interval of real numbers within the Cauchy completion of the Cauchy reals within the Dedekind reals (Theorem~\ref{ThmTopos}). Thankfully, this somewhat convoluted description simplifies in many example toposes of interest. For example, in any boolean topos or in any topos validating countable choice, the Cauchy and Dedekind reals coincide, and interval objects are given by closed intervals within them.

The characterisation of interval objects in elementary toposes provides a versatile tool for
identifying interval objects in other categories of interest, since many categories naturally embed in toposes, as discussed in Section~\ref{conclusion}.

This paper is an expanded journal version of a paper originally presented at LICS 2001~\cite{escardo:simpson:interval:lics}. Compared with the conference version, this paper contains full proofs, and also new material in Section~\ref{SecIntInSet} on general iterative midpoint objects in the category of sets in the absence of the cancellativity axiom (Theorem~\ref{ThmSuperconvex}). One part of the conference paper, the material on \emph{primitive interval functions} is omitted from the present paper. That material was further developed, within the  context of G\"odel's system~T, in
a follow-up conference paper at TLCA~2014~\cite{escardo:simpson:tlca2014}.  This and other developments stemming from the original LICS paper are discussed in more detail in Section~\ref{conclusion}.

\paragraph{Organisation.}
(\ref{notions})~Convex bodies and interval objects in a category with finite products.
(\ref{SecItCan})~Consequences of the iteration and cancellation axioms.
(\ref{SecIntInSet})~Iterative midpoints in the category of sets.
(\ref{SecIntInTop})~Interval objects in the category of topological spaces.
(\ref{SecIntInTopos})~Interval objects in elementary toposes with natural numbers object.
(\ref{conclusion})~Discussion, questions and open problems.

%% file: notions.tex
\section{Interval objects} \label{notions}

% In this section, we define the notions introduced
% in the introduction in the general setting of a category with
% finite products.
% In doing so, we try to fulfil two conflicting aims.
% On the one hand, we take care to give an
% accurate mathematical treatment using the proper category-theoretic
% notions.
% On the other, we simultaneously try to ensure that the technical
% development can be followed by a reader who is not well versed in
% category theory. Such aims are, of course, hard to reconcile.
% To help the non-categorical reader, we write the important
% finitions and proofs using standard algebraic notation, easily
% interpretable in the familiar category of sets. However, we do
% not refrain from making technical category-theoretic remarks when
% appropriate. A non-categorical reader should simply
% skip such remarks, as they do not affect the main line of development.

We define an interval object in a category $\CC$ with finite products
to be an initial two-pointed, iterative and cancellative midpoint
algebra in $\CC$.  Midpoint algebras in $\CC$ and the cancellation
axiom are defined in Section~\ref{midpoint:objects}, generalising the
discussion of the introduction, and the crucial iteration axiom is
defined in Section~\ref{iteration}, again motivated by the discussion
of the introduction. The notion of interval object is defined and
studied in Section~\ref{interval:objects:section}, following the
development of cancellative iterative midpoint objects, called
midpoint-convex bodies, in Section~\ref{m:convex}. In particular, in
Section~\ref{interval:objects:section} we define negation and
multiplication on an interval object, and we also prove their expected
properties, using the universal property that defines interval objects,
where for multiplication we need a stronger notion of a
\emph{parametrised} interval object, also defined in that section.

% To address readers who are not comfortable with
% category theory, we use
% ordinary algebraic notation, easily
% interpretable in the category of sets.
% Such readers
% should simply skip the occasional more technical
% category-theoretic remarks, which
% do not affect the main line of development.
% Throughout this section let $\CC$ be a category with
% specified finite products.

\subsection{Preliminaries}
\label{SubSecPrelim}

We briefly review the notation and terminology we use for
categories. % We illustrate each definition by giving an explicit
% description in the special case of the category $\Set$ of sets.

We write $\One$ for the terminal object in $\CC$ and
$X \xleftarrow{\pi_1} X \times Y \xrightarrow{\pi_2} Y$ for the binary product
structure.  We say that an object $X$ is \emph{exponentiable} if the
functor $(-) \times X$ has a right adjoint $(-)^X$.  The object $Y^X$
can be thought of as a function space internalising the hom-set
$\CC(X,Y)$.  The category $\CC$ is said to be \emph{cartesian closed}
if every object is exponentiable.

% In $\Set$, the finite product structure
% is given by $\One = \{\emptyset\}$, a singleton set,
% and by
% $X \times Y = \{(x,y) \mid x \in X,\,  y \in Y\}$ with
% the projections $\pi_1(x,y) = x$ and $\pi_2(x,y) = y$.
% The category $\Set$ is cartesian closed with
% $Y^X$ defined as the set of all functions from $X$ to $Y$.

Lawvere's notion of natural numbers object, see e.g.~\cite[Chapter I.9]{lambek:scott}, already mentioned in
Section~\ref{SecIntro}, abstracts the characterising properties of
the set of natural numbers.
\begin{defn}
\label{DefNno}
A \emph{natural numbers object (nno)} in $\CC$
is an object $N$ together with
two maps $0 \colon \One \longrightarrow N$ and
$\Succ \colon N \longrightarrow N$ such that,
for any object $X$ and maps $\One \xrightarrow{a} X \xleftarrow{f} X$,
there exists a unique map $g \colon N \longrightarrow X$ satisfying
\begin{align*}
g(0) & = a, \\
g(n+1) & = f(g(n)).
\end{align*}
\end{defn}
% where, in the usual category-theoretic style, the variables are to be understood as generalised elements.
% In the category $\Set$, the nno is simply $N = \mathbb{N}$, the set
% of natural numbers. Here,  as the notation suggests, $0$ and $\Succ$ are given by
% $0 \in \mathbb{N}$ (we abuse notation by identifying maps from $\One$ to $X$ with
% elements of $X$) and
% the successor function
% $\Succ \colon \mathbb{N} \to \mathbb{N}$ respectively.

We have written the equations above using ordinary algebraic notation.
Indeed, we shall use such notation freely throughout the paper.
% as if we were working in the category of sets.
In order to understand such notation in the general setting of an
arbitrary category with finite products, the variables must be
interpreted as \emph{generalised elements}, i.e.\ as maps targeted at
the object in question but sourced at an arbitrary object~$Z$.  For
example, the equation $g(n+1) = f(g(n))$ above, states that,
for all objects $Z$ and generalised elements
$n \colon Z \longrightarrow N$, it holds that
$g \circ (\Succ) \circ n = f \circ g \circ n$. In this case, as in many others, the condition simplifies to
an unquantified equation between maps, namely, $g \circ (\Succ) = f \circ g$.

% expressed by a diagram:
% \[
% \begin{tikzcd}
% P \times N \arrow[r, "{(\pi_1, \, g)}"] \arrow[d, "{\Id_P \times \: (\Succ)}"'] & P \times X \arrow[d, "{f}"] \\
% P \times N \arrow[r, "{g}"'] & X.
% \end{tikzcd}
% \]

% We shall generally give definition and proofs using such algebraic
% notation freely. The
% notation is deliberately chosen to be
% accessible to readers who are uncomfortable with category theory.
% Indeed, all definitions and proofs can be read as if we are working in the category of
% sets. On the other hand, a
% categorically inclined
% reader will have no trouble in
% translating everything into \emph{bona fide} categorical terms.
% To further assist such readers,
% we take care to perform
% our proofs in such a way that they are directly
% transferable to the categorical setting.

In the presence of finite products alone, the notion
of nno is not sufficient to implement
definition by primitive recursion in $\CC$. Indeed, even
addition need not be definable. This is rectified by using
the following parametrised notion of nno instead~\cite{MR1059247}.
\begin{defn}
A \emph{parametrised nno}
is given by an object $N$ together with
two maps $0 \colon \One \longrightarrow N$ and
${\Succ} \colon N \longrightarrow N$ such that,
for any objects $X, P$ and maps $P \xrightarrow{a} X \xleftarrow{f} P \times X$,
there exists a unique map $g \colon P \times N \longrightarrow X$ satisfying
\begin{align*}
g(p,0) & = a(p), \\
g(p,n+1) & = f(p,g(p,n)).
\end{align*}
\end{defn}
Here $P$ is to be thought of as an object of parameters.

It is easily shown that any parametrised nno is an nno,
and that if a parametrised nno exists then any nno is
parametrised. It is not in general the
case that any nno is parametrised, though the following
sufficient condition for this is well known.
\begin{prop}
\label{PropCccNno}
If $\CC$ is cartesian closed then any nno is parametrised.
\end{prop}

\subsection{Midpoint objects} \label{midpoint:objects}

We begin the reformulation of the notions from Section~\ref{SecIntro}
by defining a notion of \emph{midpoint object} based on
the laws of the binary midpoint operation.
\begin{defn}
  A \emph{midpoint object} in $\CC$ is a pair $(A,m)$, where
  $A$ is an object of $\CC$ and $m \colon A \times A \longrightarrow A$
  is a map
  satisfying the
  idempotency, commutativity and transposition equations of
  Section~\ref{SecIntro}.
  A midpoint object is said to be \emph{cancellative} if it also satisfies
  the cancellation law~(\ref{cancellation}) of Section~\ref{SecIntro}.
  If $(A,m)$ and $(B,m')$ are midpoint objects then
  a \emph{homomorphism} from $(A,m)$ to $(B,m')$ is a map
  $h \colon A  \longrightarrow B$ satisfying $h(m(x,y)) = m(h(x),h(y))$.
\end{defn}
% The above definitions make use of ordinary algebraic notation.
% In order to understand such notation in an arbitrary category with finite
% products, the variables must be interpreted as ``generalised elements'',
% i.e.\ as maps targeted at $A$ and sourced at any object.
% For example, the homomorphism equation states that, for
% all generalised elements $x,y\colon Z \longrightarrow A$ (where $Z$ is any
% object), $m \circ h \circ ( x,y) =
% m' \circ ( f \circ x, f \circ y )$. In this case,
% the condition simplifies to the (unquantified) equation
% $h \circ m = m' \circ (h \times h)$.
Observe that the transposition law is equivalent to saying that the
midpoint operation $m:A \times A \longrightarrow A$ is itself a homomorphism
with respect to the induced midpoint operation on $A \times A$.

Henceforth, we often use the same notation for a midpoint object and
its underlying object $A$, in which case we
conventionally
use the letter~$m$ for the associated midpoint operation.
% In order to avoid ambiguity, we avoid the word ``morphism''.
% Instead,
We use \emph{map} to refer to arrows in the underlying category
$\CC$, and \emph{homomorphism} for homomorphisms of midpoint objects.
We also adopt the convention of referring to midpoint objects in
categories of particular interest by replacing the word ``object''
with the name of the objects of the category. Thus a \emph{midpoint set}
is simply a midpoint object in the category of sets,
and a homomorphism of midpoint sets is a function preserving
the midpoint operation. Similarly,
a \emph{midpoint (topological) space} is a midpoint object in the
category of topological spaces (thus $m$ is required to
be continuous), and a homomorphism between midpoint spaces is
a continuous function preserving $m$.
Similar naming conventions will be applied to all other notions
of algebraic structure we consider.

Examples of midpoint sets abound. Every convex subset of a vector
space over $\mathbb{R}$ (or $\mathbb{Q}$) is a midpoint set with
$m(x,y)= (x+y)/2$.  Similarly, every convex subset of a topological
vector space is a midpoint space. Moreover, one can obtain further
examples by taking arbitrary subsets (subspaces) of midpoint sets
(spaces) that are closed under the operation $m(x,y)$. Thus, for
example, the (dyadic-)rational points in any convex subspace of
Euclidean space form a midpoint space. These examples are all
cancellative. Semilattices provide examples of midpoint sets that are not
cancellative.

As discussed in Section~\ref{SecIntro}, we shall define an interval
object as a free midpoint-convex body on two generators.  To motivate
the need for considering a notion of convexity, let us consider the
analogous free objects in the categories of midpoint sets and
spaces. It is easy to show that the free midpoint set over two
generators~$0$ and~$1$ is the
set~$D = \{\frac{m}{2^n} \! \in [0,1] \mid m,n \in \mathbb{N}\}$ of
dyadic rationals in the closed unit interval endowed with the midpoint
structure $m(x,y)=(x+y)/2$.  Explicitly, this says that for any
midpoint set~$A$ and any $a,b \in A$ there is a unique midpoint
homomorphism $h:D \to A$ with $h(0)=a$ and $h(1)=b$.
%
% Notice that $D$ is cancellative.
%
% (AKS) I understand why this was here. But it's not that important a remark,
% and it breaks the flow.
%
In the category of topological spaces,
the free midpoint
space over $\{0,1\}$ is the set~$D$ with the
\emph{discrete} topology.
Here the discrete topology arises because, when all
operators are finitary,
free topological algebras
over discrete spaces
always have discrete topologies.

By considering the example of topological spaces, one sees that the
notion of midpoint space falls short of characterising the interval
$[0,1]$ on two counts. First, the points are wrong --- the object $D$
contains only dyadic rational points. Second, the topology
is wrong --- it is the discrete rather than Euclidean topology.
The iteration property solves both problems at once.

\subsection{Iterative midpoint objects} \label{iteration}

One formulation of
iteration was given in Section~\ref{SecIntro}.
This formulation involves an infinitary operation $M$,
which is not directly interpretable in an arbitrary
category with finite products.
We circumvent this issue by giving here a
concise reformulation of iteration, phrased
in terms of finite products alone.
Moreover, whenever $\CC$ has enough structure to interpret
infinitary operations,  the new
formulation of iteration is equivalent to that
of Section~\ref{SecIntro}.

\begin{defn}
  A midpoint object~$(A,m)$ is said to be \emph{iterative} if for every
  map $c \colon X \longrightarrow A \times X$ there is a
  unique map $u:X \longrightarrow A$ making the diagram below commute.
\begin{equation}
\label{DiagIter}
\begin{tikzcd}[column sep=huge]
  A \times X \arrow[r, "{\Id \times u}"] & A \times A \arrow[d, "{m}"] \\
  X \arrow[u, "{c}"] \arrow[r, "{u}"'] & A,
\end{tikzcd}
\end{equation}
\end{defn}
Using category-thoretic terminology, the above definition
says that
$A$ is iterative if the
$(A \times (-))$-algebra $m \colon A \times A \longrightarrow A$ is
final with respect to coalgebra-to-algebra
homomorphisms from abitrary $(A \times (-))$-coalgebras.

To understand the above definition intuitively, one should
consider the map $c \colon X \longrightarrow A \times X$ as exhibiting
$X$ as an \emph{object of sequences over $A$}. Indeed,
any such map has the form $c = (h,t)$ for some  $h \colon X \longrightarrow A$ and
$t \colon X \longrightarrow X$. Think of  $h$ as the map returning the head
$a_0$ of a sequence $a_0a_1a_2\dots$, and of
$t$ as  the map returning the tail $a_1a_2a_3\dots$.
The definition above simply says that $A$ is iterative if
for every such object of sequences, $c = (h,t)\colon X \longrightarrow A \times X$,
there is a unique map $u \colon X \longrightarrow A$ satisfying $u(x) = m(h(x),u(t(x)))$.
% A connection with the iteration axioms of Section~\ref{SecIntro}
% should now be emerging. Intuitively,
The map $u$ implements
the operation $x \mapsto M(x)$ on sequences $x$ in $X$.
% , as defined by $c \colon X \longrightarrow A \times X$.
The equation $u(x) = m(h(x),u(t(x)))$
corresponds to the unfolding property of $M$. The uniqueness of $u$ corresponds
to the canonicity property of $M$.

To establish a precise correspondence between the definition
of iterativity above and the iteration axioms of Section~\ref{SecIntro},
we need to assume additional structure on $\CC$ sufficient for
interpreting the notion of infinitary operation. Here, the appropriate
structure is an \emph{exponentiable parametrised nno},
see \S\ref{SubSecPrelim}.
An \emph{infinitary operation} $M$ on an object $A$
is then given by a map $M \colon A^N \longrightarrow A$.
% Why and where do we use that the nno is parametrised.

In what follows, we shall use the notational convenience of working
with generalised elements of $A^N$ as if they were simply
sequences $(x_i)$ of generalised elements of $A$. Indeed,
if $M$ is an infinitary operation, then we shall
often write $M(x_0,x_1,x_2, \dots)$ or (more honestly)
$M_{\!i} \, x_i$ for
$M((x_i))$. Using the latter  notation, we can formulate the
iteration axioms of Section~\ref{SecIntro} so that they make
sense whenever $\CC$ has an exponentiable parametrised nno.
%% Remove repetition
% \begin{prop} \label{unfolding:canonicity}
%   In a category with an exponentiable parameterised nno $N$, a
%   midpoint object $(A,m)$ is iterative if and only if there exists
%   $M : A^N \to \Nat$ such that, for all sequences $(x_i)$,
% \begin{align*}
% & \textbf{(Unfolding)}
% &  & M_i \, x_i  \,  =  \, m(x_0, M_i \, x_{i+1}), \\
% & \textbf{(Canonicity)}
% & &
% \text{if $y_i = m(x_i,y_{i+1})$ for all $i$, then $y_0 = M_i \, x_i$.}
% \end{align*}
% \end{prop}

\pagebreak[3]
\begin{prop}
\label{PropSetIt} \label{unfolding:canonicity}
Suppose that $\CC$ has an exponentiable parametrised nno $N$.
\begin{enumerate}
\item \label{iter:1}
A midpoint object $A$ is iterative if and only if
there exists an infinitary operation $M :  A^N \to A$ on $A$
satisfying the following properties:
\begin{align*}
& \textbf{(Unfolding)}
&  & M_i \, x_i  \,  =  \, m(x_0, M_i \, x_{i+1}), \\
& \textbf{(Canonicity)}
& &
\text{if $y_i = m(x_i,y_{i+1})$ for all $i$, then $y_0 = M_i \, x_i$.}
\end{align*}
Moreover if $A$ is iterative then there is a unique $M$
satisfying unfolding.

\item \label{iter:2} If $A$ and $A'$ are iterative then any midpoint
  homomorphism $f: A \longrightarrow A'$ is also a homomorphism with respect to
  the unique associated infinitary operations; i.e.\ for all sequences $(x_i)$,
  \[
  f(M_i \, x_i) = M'_i \, f(x_i).
  \]
\end{enumerate}
\end{prop}
\begin{proof}
  For statement~\ref{iter:1}, suppose that $A$ is iterative.
  We derive $M \colon A^N \longrightarrow A$.
  Define $h \colon A^N \longrightarrow A$ by $h((x_i)) = x_0$,
  and $t \colon A^N \longrightarrow A^N$ by $t((x_i)) = (x_{i+1})$.
  This determines $(h,t)\colon A^N \longrightarrow A \times A^N$.
  So, by iterativity, there is a unique
  $M \colon A^N \longrightarrow A$ such that $M((x_i)) = m(h((x_i)),M(t((x_i))))$,
  i.e.\ such that $M_i\, x_i \, = \, m(x_0,\, M_i\, x_{i+1})$.
  Thus there is a unique infinitary operation $M$ satisfying unfolding.

  To show that this unique~$M$ also satisfies canonicity, let
  $(x_i)$ and $(y_i)$ be sequences satisfying $(y_i) =
  (m(x_i,y_{i+1}))$.  Consider the map $c \colon
  N \longrightarrow A \times N$ defined by $c(n) = (x_n,n+1)$.  Then  $n \mapsto y_n\colon
  N \longrightarrow A$ and $n \mapsto M_i\, x_{i+n}: N \longrightarrow A$, if taken as $u$,
  both make diagram~(\ref{DiagIter}) commute. Thus these two maps are
  equal, i.e.\ $y_n = M_i\, x_{i+n}$. Thus, in particular, $y_0 = M_i\, x_i$
  as required.

  Conversely, suppose that there exists $M: A^N \longrightarrow A$ satisfying
  unfolding and canonicity. To show iterativity, take any
  $c = (h,t)\colon X \longrightarrow A \times X$. We must show that there is a unique
  map $u\colon X \longrightarrow A$ such that $u(x) = m(h(x), u(t(x)))$.  Such a
  map is defined by
  \begin{gather*}
  u(x) =  M_i \, h(t^i(x)) = M(h(x), h(t(x)), h(t(t(x)), \dots)).
  \end{gather*}
  (This is defined formally by using primitive recursion
  to first define $(x,n) \mapsto t^n(x) :  X \times N \longrightarrow X$,
  from which one obtains $(x,n) \mapsto h(t^n(x)) \colon X \times N \longrightarrow A$,
  whence $x \mapsto (h(t^i(x))) \colon X \longrightarrow A^N$, from which
  $u$ is immediately definable.)
  The map $u$ satisfies $u(x) = m(h(x), u(t(x)))$ by the unfolding property
  of $M$.  For uniqueness of~$u$, suppose there exists $u': X \to A$ such
  that $u'(x) = m(h(x), u'(t(x)))$.  Thus, by the
  uniqueness property of the nno, $(u'(t^i(x))) = (m(h(t^i(x)),
  u'(t^{i+1}(x))))$.
  So, by canonicity, $u'(x) = (u'(t^0(x))) = M_i \, h(t^i(x)) = u(x)$, as required.

  For statement~\ref{iter:2}, if $f: A \longrightarrow A'$ is a homomorphism
  then, by unfolding,
  \begin{gather*}
  (f(M_{j}\, x_{i+j}))\:  = \:  (f(m(x_i,\, M_{j} \, x_{i+j+1} ))) \:
     =  \: (m'(f(x_i), \, f(M_{j} \,  x_{i+j+1}))).
  \end{gather*}
  So, by canonicity, we have $f(M_j \, x_j) = M'_j \, f(x_j)$,
  as required.
\end{proof}
Note that the proof above does not
depend on the midpoint axioms, we use only the properties
of an iterative binary operation.

\begin{rem}
  Proposition~\ref{PropSetIt} can be reformulated as follows. Construct a category whose objects are triples $(A,m,M)$, where $(A,m)$ is a midpoint object and $M$ satisfies the equations in statement~(\ref{iter:1}), and whose morphisms are midpoint homomorphisms preserving $M$. Then the evidently faithful forgetful functor into the category of midpoint objects is full, its image on objects consists of exactly the iterative midpoint objects, and  the resulting functor onto the category of iterative midpoint objects is an isomorphism of categories.
\end{rem}

We end this subsection with the remark that
there are two different ways of making sense of
infinitary operations in even greater generality than above.
The first is to drop the assumption that
$N$ is exponentiable.  At this level of generality, an infinitary operation is
given by a family of functions
$M_X \colon \CC(X \times N, A) \to \CC(X,A)$
natural in $X$. In the special case that the object $N$ is
exponentiable, this notion coincides with the one given above.
It is possible to formulate the iteration axioms using this
more general notion of infinitary operation, and to establish the counterpart of
Proposition~\ref{PropSetIt}.
% However, the extra  generality complicates the proof.

A second way of generalising the notion of infinitary operation is
to drop the requirement that $\CC$ have an nno, and require instead
that each functor $X \times (-)$ have a final coalgebra $\FS(X)$.
With this assumption, one formulates an infinitary operation on $A$ as
a map $M \colon \FS(A) \longrightarrow A$. Again, in the special case that
$\CC$ has a parametrised
natural numbers object $N$, the new definition   coincides with the one
given, as then $N$ is exponentiable with $A^N$ given by $\FS(A)$.
At this level of generality, one can again formulate the iteration axioms.
However, it appears that the counterpart of Proposition~\ref{PropSetIt} does
\emph{not} hold, as, in the absence of the object~$N$, the unfolding and canonicity axioms
seem strictly stronger than iterativity.

\subsection{Midpoint-convex bodies} \label{m:convex}

% TODO. Probably rename "convex body" to something else.
\begin{defn}
  A \emph{midpoint-convex body}, or \emph{m-convex body} for short, is
  a cancellative iterative midpoint object.
\end{defn}
By Proposition~\ref{PropSetIt}, in the category of sets this
definition coincides with the definition of midpoint-convex body given in
Section~\ref{SecIntro}.

Many examples of m-convex bodies will be presented in Section~\ref{SecApplySuperconvex}.
Here we just give the following non-example, showing that the cancellation property of
m-convex bodies is not implied by the midpoint and iteration axioms.
\begin{ex}
Consider the midpoint set $A$ defined by:
\begin{eqnarray*}
A & = & \{(x,y) \in [0,1] \times [0,1] \mid
\mbox{$x = 1$ or $y = 1$}\} \\
m((x,y),\, (x',y'))  & = &
\left\{ \begin{array}{ll}
  (1,\, y \oplus y') & \mbox{if $x = x' = 1$} \\
  (x \oplus x',\, 1) & \mbox{otherwise}
    \end{array}\right.
\end{eqnarray*}
This is iterative, as is
easily verified by defining $M$ and
checking the unfolding and canonicity properties.
The failure of cancellation is shown by
$m((1,x),(z,1)) = m((1,y),(z,1))$ whenever
$z \neq 1$.
\end{ex}

The cancellation property of m-convex bodies will serve many technical
purposes in the sequel.  One of them will be established in
Proposition~\ref{PropSetApprox}, where cancellation is proved
equivalent to a useful approximation property of the infinitary
operation $M$.  A further consequence of cancellation is that it
allows iterativity to be obtained via an equationally axiomatized
iterated midpoint operation $M$ (Proposition~\ref{PropSetEqns} below).

% The notion of iterativity, as characterised
% in Proposition~\ref{PropSetIt}, shows that
% the infinitary $M$ is determined by the binary midpoint operation $m$.
% We now proceed in the other direction, deriving $m$ from an equational
% axiomatization of $M$.

\begin{defn}
\label{def:supermidpoint}
 Suppose that $\CC$ has an exponentiable parametrised nno $N$.
  A \emph{supermidpoint object} is an object $A$ together with an infinitary
  operation $M:A^N \longrightarrow A$ subject to the following equations:
  \begin{align*}
  & \textbf{(Idempotency)} & M(x,x,x,\dots)  & = x, \\
  & \textbf{(Commutativity)} & M(x,y,y,y,\dots) & = M(y,x,x,x,\dots), \\
  & \textbf{(Transposition)} & M_i \,  M_j \, x_{ij} & = M_j \, M_i \, x_{ij}, \\
  & \textbf{(Unfolding)}  & M_i \, x_i & = M(x_0, \, M_i \, x_{i+1}, \, M_i \, x_{i+1},
    M_i \, x_{i+1}, \, \dots).
  \end{align*}
A \emph{supermidpoint homomorphism} between supermidpoint objects $A$ and $B$,
is a map $h \colon A \longrightarrow B$ preserving the infinitary operation.
\end{defn}
Given a supermidpoint object $A$, the \emph{associated
binary operation} is defined by
  \begin{equation}
  m(x,y) = M(x,y,y,y,\dots). \label{m-from-M}
  \end{equation}
It is clear that the idempotency, commutativity and unfolding laws for the
infinitary operation coincide with those given earlier,
and that the transposition law for the infinitary case
specializes to that for the binary case.
Thus the binary operation associated to the supermidpoint
operation is a midpoint operation. However, there is no reason for $M$ to
satisfy the canonicity axiom. For a counterexample, take any non-trivial
complete lattice and let $M$ be the countable supremum operation.

\begin{defn}
We say that a supermidpoint object $(A,M)$ is \emph{cancellative} if
the induced binary $m$ satisfies cancellation.
\end{defn}
It turns out that any cancellative supermidpoint object \emph{is} iterative, and hence
an m-convex body. % This is the second half of the result below.
\begin{prop} Suppose that $\CC$ has an exponentiable parametrised nno $N$.
\label{PropSetEqns}
\begin{enumerate}
\item \label{iter:eqn} Any (cancellative) iterative midpoint object is a
  (cancellative) supermidpoint object under the associated infinitary operation.
\item \label{eqn:iter} Any cancellative supermidpoint object is an
  m-convex body % cancellative iterative midpoint object
  under the associated binary operation.
\end{enumerate}
\end{prop}
\begin{proof}
% TODO. Maybe give quick indications of the routine verification.
We omit the routine verification of \ref{iter:eqn}.
For statement \ref{eqn:iter}, by the remarks above, it is
sufficient to establish the canonicity axiom for $M$.
Suppose then that $(x_i) =   (m(y_i,x_{i+1}))$.
Using unfolding and transposition, we obtain,
\begin{gather*}
     m(x_0, \, M_i \, x_{i+1})
          =  M_i \, x_i
          =  M_i \, m(y_i,x_{i+1})
          =  m(M_i \, y_i, \, M_i \, x_{i+1}).
  \end{gather*}
Hence, by Cancellation, $x_0 = M_i \, y_i$, as required.
\end{proof}
Using the fact that homomorphisms with respect to the $m$ and $M$
operations coincide (Proposition~\ref{PropSetIt}(\ref{iter:2})), the
following conclusion is obtained.
\begin{corollary}
If $\CC$ has an exponentiable parametrised nno $N$ then
 the category of m-convex bodies with midpoint homomorphisms is
  isomorphic to the category of cancellative supermidpoint objects with
  supermidpoint homomorphisms.
\end{corollary}
In Section~\ref{sets} below, we shall see another
similar isomorphism between the category of
convex bodies and the
category of cancellative superconvex sets with superaffine maps
in the sense of K\"onig~\cite{koenig:superconvex}.

One benefit of having simple abstract definitions of
convex body is that it is easy to prove that such objects
are preserved by various categorical constructions and
functors. We now discuss some examples of such properties.
% The reader who is not interested in such categorical properties may
% safely skip to the next subsection.

\newcommand{\Mid}{\operatorname{Mid}}
\newcommand{\Ultra}{\operatorname{CB}}

We write $\Mid(\CC)$ and $\Ultra(\CC)$ for the categories of midpoint
objects and of m-convex bodies over $\CC$, both with midpoint
homomorphisms.
\begin{prop} \label{limit:creation}
  The forgetful functor $U: \Ultra(\CC) \to \CC$ creates limits.
\end{prop}
\begin{proof}
Consider any diagram $D$ in  $\Ultra(\CC)$ with objects $\{(A_i, m_i)\}_{i \in I}$, and suppose that
$U$ maps the diagram to a diagram $UD$ in $\CC$ with limit cone  $\{A \xrightarrow{l_i}  A_i\}_{i \in I}$.
It is standard that there exists a unique map $m : A \times A \longrightarrow A$ such that
 every $l_i$ is a homomorphism from $(A,m)$ to $(A_i,m_i)$. % namely the unique $m$ such that,
%for every $i \in I$,
%the diagram below commutes.
\[
\begin{tikzcd}
A \times A \arrow[r, "{m}"] \arrow[d, "{l_i \times l_i}"'] & A \arrow[d, "{l_i}"] \\
A_i \times A_i \arrow[r, "{m_i}"'] & A_i \, .
\end{tikzcd}
\]
Moreover, the
 structure $(A, m)$ inherits all equational and Horn properties satisfied by every $A_i$, and is therefore
a cancellative midpoint object. It remains to show that $(A,m)$ is iterative.

Suppose we have $(h,t) \colon X \longrightarrow A \times X$ in $\CC$. Using the iterativity of $(A_i, m_i)$, let
$u_i : X \longrightarrow A_i$ be the unique map such that the top-left square below commutes.
\[
\begin{tikzcd}
X \arrow[r, "{(l_i \circ h,\,t)}"] \arrow[d, "{u_i}"'] & A_i \times X \arrow[r, "{f \times \Id}"] \arrow[d, "{\Id \times u_i}"] & A_j \times X \arrow[d, "{\Id \times u_i}"] \\
A_i \arrow[d, "{f}"'] & A_i \times A_i \arrow[l, "{m_i}"] \arrow[d, "{f \times f}"] & A_j \times A_i \arrow[d, "{\Id \times f}"] \\
A_j & A_j \times A_j \arrow[l, "{m_j}"] & A_j \times A_j \arrow[l, "{\Id}"]
\end{tikzcd}
\]
For any map $f \colon A_i \to A_j$ in $D$, the diagram shows that
$f \circ u_i$ enjoys the characterising property of $u_j$, so $u_j = f \circ u_i$.
That is,
$\{X \xrightarrow{u_i} A_i\}_{i \in I}$ is a cone for $U\!D$.

Let $X \xrightarrow{u} A$ be the unique cone morphism from
$\{X \xrightarrow{u_i} A_i\}_{i \in I}$
to  $\{A \xrightarrow{l_i}  A_i\}_{i \in I}$.
The diagram below,
in which the bottom composite equals $u_i$,
shows that $m \circ (\Id \times u) \circ (h,t)$ is also a cone morphism.
\[
\begin{tikzcd}[column sep=huge]
X \arrow[r, "{(h,t)}"] \arrow[d, "{\Id}"'] & A \times X \arrow[r, "{\Id \times u}"] \arrow[d, "{l_i \times \Id}"] & A \times A \arrow[r, "{m}"] \arrow[d, "{l_i \times l_i}"] & A \arrow[d, "{l_i}"] \\
X \arrow[r, "{(l_i \circ h,\, t)}"'] & A_i \times X \arrow[r, "{\Id \times u_i}"'] & A_i \times A_i \arrow[r, "{m_i}"'] & A_i
\end{tikzcd}
\]
It follows that  $u = m \circ (\Id \times u) \circ (h,t)$.

It remains to show that any $u$ satisfying $u = m \circ (\Id \times u) \circ (h,t)$ is a cone morphism;
that is, $l_i \circ u = u_i$, for all $i \in I$. The diagram below, in which the top composite equals $u$, shows that
$l_i \circ u = m_i \circ (\Id \times (l_i \circ u)) \circ (l_i \circ h,\, t)$.
\[
\begin{tikzcd}[column sep=huge]
X \arrow[r, "{(h,t)}"] \arrow[d, "{\Id}"'] & A \times X \arrow[r, "{\Id \times u}"] \arrow[d, "{l_i \times \Id}"] & A \times A \arrow[r, "{m}"] \arrow[d, "{l_i \times l_i}"] & A \arrow[d, "{l_i}"] \\
X \arrow[r, "{(l_i \circ h,\, t)}"'] & A_i \times X \arrow[r, "{\Id \times (l_i \circ u)}"'] & A_i \times A_i \arrow[r, "{m_i}"'] & A_i
\end{tikzcd}
\]
That is, $l_i \circ u$ enjoys the characterising property of $u_i$, so indeed $l_i \circ u = u_i$.
\end{proof}

As a particular case of Proposition~\ref{limit:creation}, if $(A,m)$
and $(A',m')$ are m-convex bodies then so is
\[
(A \times A', \; (A \times A') \times (A \times A') \xrightarrow{\cong} (A
\times A) \times (A' \times A') \xrightarrow{m \times m'} A \times A').
\]
Midpoint-convex bodies are also closed under internal powers.
\begin{prop}
\label{PropPowers}
If $(A,m)$ is a m-convex object then so is
\[
(A^B, \; A^B \times A^B \xrightarrow{\cong} (A \times A)^B \xrightarrow{m^B} A^B)
\]
for any exponentiable object $B$.
\end{prop}
%\begin{proof}
%Since powers preserve equational and Horn properties,
%it holds that $(A^B, m')$, where $m'$ is the structure map in the statement of the proposition, is a cancellative midpoint object.
%It remains to verify the iteration axiom.
%
%Consider any map $X \xrightarrow{(h,t)} A^B \times X$. We need to prove that there is a unique map
%$u \colon X \longrightarrow A^B$ making the diagram below commute.
%\[
%\begin{tikzcd}
%A^B \times  X \arrow[r, "{\Id_{A^B} \times u}"] & A^B \times A^B \arrow[d, "{m'}"] \\
%X \arrow[u, "{(h, \, t )}"] \arrow[r, "{u}"'] & A^B
%\end{tikzcd}
%\]
%
%Let $\widetilde{h} \colon X \times B \longrightarrow A$ be the exponential transpose of $h$.
%Since $(A,m)$ is iterative, there is a unique map $\widetilde{u} \colon X \times B \longrightarrow A$ making the diagram below commute.
%\[
%\begin{tikzcd}
%A \times \left( X \times B\right) \arrow[r, "{\Id_A \times \widetilde{u}}"] & A \times A \arrow[d, "{m}"] \\
%X \times B \arrow[u, "{(\widetilde{h}, \, t \times \Id_B )}"] \arrow[r, "{\widetilde{u}}"'] & A
%\end{tikzcd}
%\]
%
%It is easy to verify that a map $u \colon X \longrightarrow A^B$ makes the top diagram commute if and only if its exponential transpose
%$\widetilde{u} \colon X \times B \longrightarrow A$ makes the bottom diagram commute. Therefore, there is indeed a unique such $u$.
%\end{proof}
\noindent
Proposition~\ref{PropPowers} is, in fact, a special case of Proposition~\ref{PropAdj}\eqref{PropAdj:1} below, instantiated to the adjunction $(-) \times B \dashv (-)^B$. We accordingly postpone the proof to the more general Proposition~\ref{PropAdj}.
%We do not give a direct proof, because the proposition is a special {PropPowers}

We next establish conditions under which
m-convex bodies are preserved by functors.  Suppose $\CD$ is a
category with finite products, and that $F: \CC \to \CD$ is a functor
preserving finite
products.  Then there is a functor $\Ext{F} : \Mid(\CC) \to \Mid(\CD)$
whose action on objects is
\begin{eqnarray}
\label{struture:mprime}
\Ext{F}(A,m) & = &
(FA,\; FA \times FA \xrightarrow{\cong} F(A \times A) \xrightarrow{Fm} FA)
\end{eqnarray}
and whose action on homomorphisms is inherited from $F$.

\begin{prop}
\label{PropAdj}
Given an adjunction  $F \dashv G : \CD \to \CC$ then:
\begin{enumerate}
\item \label{PropAdj:1} $\Ext{G}$ cuts down to a functor $\Ext{G}:
  \Ultra(\CD) \to \Ultra(\CC)$.

\item \label{PropAdj:2} If $G$ also has a right adjoint $H$ then
    $\Ext{H}: \Ultra(\CC) \to \Ultra(\CD)$ is right adjoint to
    $\Ext{G}: \Ultra(\CD) \to \Ultra(\CC)$.
% Thus, in particular, $F: \CC     \to \CD$ preserves interval objects.
\end{enumerate}
\end{prop}
\begin{proof}
For statement~\eqref{PropAdj:1}, suppose $(A,m)$ is m-convex in $\CD$. We  need to show
$\Ext{G}(A,m)$, which is a midpoint object, satisfies  cancellation and iterativity. We verify the latter, leaving the leave the former  to the reader.

Consider any map $X \xrightarrow{(h,t)} GA \times X$ in $\CC$. We need to prove that there is a unique map
$u \colon X \longrightarrow GA$ making the diagram below commute
\[
\begin{tikzcd}
GA \times  X \arrow[r, "{\Id \times u}"] & GA \times GA \arrow[d, "{m'}"] \\
X \arrow[u, "{(h, \, t )}"] \arrow[r, "{u}"'] & GA\, ,
\end{tikzcd}
\]
where $m'$ is the structure map defined in~\eqref{struture:mprime}.

Let $\widetilde{h} \colon FX  \longrightarrow A$ in $\CD$ be  the transpose of $h$ across the adjunction.
Since $(A,m)$ is iterative, there is a unique map $\widetilde{u} \colon FX  \longrightarrow A$ making the diagram below commute.
\[
\begin{tikzcd}
A \times FX \arrow[r, "{\Id \times \widetilde{u}}"] & A \times A \arrow[d, "{m}"] \\
FX \arrow[u, "{(\widetilde{h}, \, Ft)}"] \arrow[r, "{\widetilde{u}}"'] & A
\end{tikzcd}
\]

It is easy to verify that a map $u \colon X \longrightarrow GA$ makes the top diagram commute if and only if its transpose
$\widetilde{u} \colon FX \longrightarrow A$ makes the bottom diagram commute. Therefore, there is indeed a unique such $u$.

For statement~\eqref{PropAdj:2}, we have, by statement~\eqref{PropAdj:1}, functors $\Ext{G}:
  \Ultra(\CD) \to \Ultra(\CC)$ and $\Ext{H}:
  \Ultra(\CC) \to \Ultra(\CD)$. Write
  $\eta_Y : Y \longrightarrow HGY$ and $\varepsilon_X: GHX \longrightarrow X$ for components of the unit and counit of the
  adjunction $G \dashv H$. We show that, for midpoint objects $(Y,m')$ in $\Mid(\CD)$ and
  $(X,m)$ in $\Mid(\CC)$, it holds that 
  \begin{equation}
  \label{unit-and-counit}
  (Y,m') \xrightarrow{\eta_Y} \Ext{H}\Ext{G} (Y,m') \qquad \Ext{G} \Ext{H} (X,m) \xrightarrow{\varepsilon_X} (X,m)
  \end{equation}
  are maps in $\Mid(\CD)$ and $\Mid(\CC)$ respectively. In particular, if $(Y,m')$ lies in $\Ultra(\CD)$ and
  $(X,m)$ in $\Ultra(\CC)$, then so do the above maps. This establishes the claimed adjunction
  $\Ext{G} \dashv \Ext{H}$, with \eqref{unit-and-counit} providing its unit and counit, since all equations required by an adjunction  are then immediate consequences of
  the corresponding equations for $G \dashv H$.
  
It remains to show that the maps~\eqref{unit-and-counit} lie in the categories claimed, i.e., that they are midpoint homomorphisms. This is established by the commuting diagrams below, which are in $\CD$ and $\CC$. 
\[
\begin{tikzcd}
Y \times Y \arrow[r, "{\eta_Y \times \eta_Y}"] \arrow[d, "{\Id}"'] & HGY \times HGY \arrow[d, "{\cong}"] \\
Y \times Y \arrow[r, "{\eta_{Y \times Y}}"] \arrow[d, "{m'}"'] & HG(Y \times Y) \arrow[d, "{HGm'}"] \\
Y \arrow[r, "{\eta_Y}"'] & HGY
\end{tikzcd}
\qquad
\begin{tikzcd}
GHX \times GHX \arrow[r, "{\varepsilon_X \times \varepsilon_X}"] \arrow[d, "{\cong}"'] & X \times X \arrow[d, "{\Id}"] \\
GH(X \times X) \arrow[r, "{\varepsilon_{X \times X}}"] \arrow[d, "{GHm}"'] & X \times X \arrow[d, "{m}"] \\
GHX \arrow[r, "{\varepsilon_X}"'] & X
\end{tikzcd}
\]
The bottom squares commute by the naturality of the unit and counit. The commutativity of the top square in the left-hand diagram is established by showing that the top map is equal to the composite going around the other three sides of the square via the inverse to the canonical isomorphism, which is easily shown by considering each component separately. For the top square of the right-hand diagram, one shows in a similar way that the bottom map is equal to the composite going around the square with the isomorphism inverted. 
\end{proof}

\subsection{Interval objects} \label{interval:objects:section}

We define the notion of interval object by specifying its universal property.
\begin{defn}\label{DefIO}
  An \emph{interval object} is an initial m-convex body with two
  global points, i.e., a structure $(I,\oplus, a,b)$ where
  $(I,\oplus)$ is an m-convex body, and $a: \One \longrightarrow I$ and
  $b: \One \longrightarrow I$, such that, for any other m-convex body with two
  global points $(A, \oplus, x, y)$, there is a unique midpoint
  homomorphism $h:I \longrightarrow A$ with $h(a) = x$ and $h(b) = y$.
\end{defn}
We often denote such an interval by the interval notation $[a,b]$ and we call the global points $a,b$ the \emph{endpoints} of $[a,b]$.
Our usual choices of notation for the endpoints are $a=0$ and $b=1$, or
$a=-1$ and $b=1$, as convenient.
Examples of interval objects will be presented in Sections~\ref{sets}-\ref{topos}.

By idempotency, every global point $x : \One \longrightarrow A$ of a midpoint algebra $A$ is a homomorphism from the terminal midpoint algebra. Because of this,
Definition~\ref{DefIO} defines the interval object as the coproduct $\One + \One$ calculated in
the category of m-convex bodies.
Alternatively, using the notion of free generation defined below,
an interval object is the
m-convex body freely generated
by the coproduct $\One + \One$ in $\CC$. Note, however,
that we have formulated the definition of
interval object in such a way that it makes sense without
assuming that a coproduct $\One + \One$ exists in $\CC$.
\begin{defn}
  An \emph{m-convex body freely generated} by an object $X$ is a
  m-convex body $FX$ together with a map $\eta:X \longrightarrow FX$, called
  the \emph{insertion of generators}, such that for every m-convex
  body $A$ and map $f:X \longrightarrow A$ there is a unique
  homomorphism $\bar{f}:FX \longrightarrow A$ such that
  $f = \bar{f} \circ \eta$.
\end{defn}
% Various interesting examples of free objects determined by
% other generating objects
% are discussed in Section~\todo{find the section} below.
% In this subsection, we develop the basic properties of
% interval objects in a category with finite products.

Of course, one easily observes that any two interval objects are isomorphic
to each other by isomorphisms that are midpoint homomorphisms.

We now derive the existence of a negation operation and its basic laws.
\begin{prop}[Negation]
  If $[-1,1]$ is an interval object
  then there is a unique map $x \mapsto -x:[-1,1] \longrightarrow \text{$[-1,1]$}$
  such that
  \begin{gather}
   -(1)  =  -1, \label{neg:1} \\
   -(-1)=1, \label{neg:2}\\
   -(x \oplus y)=(-x) \oplus (-y).\label{neg:3}
  \end{gather}
  Moreover,
  \begin{gather}
  -(-x)=x, \label{neg:4}\\
  -0=0, \label{neg:5}\\
  x \oplus -x = 0,\label{neg:6}
  \end{gather}
  where we define
  \begin{gather}
  0 = -1 \oplus 1.
  \end{gather}
\end{prop}
\begin{proof}
  The existence and uniqueness of $-$ follow directly from the definition
  of interval object, because equation~(\ref{neg:3}) says
  that $-$ is an endohomorphism.  For (\ref{neg:4}), the map
  $h(x)=-(-x)$ is easily seen to satisfy the equations $h(-1)=-1$ and
  $h(1)=1$. Being a composition of homomorphisms, $h$ itself is a
  homomorphism. Hence, by the uniqueness part of the defining
  property of $[-1,1]$, $h$ must be the identity.  For (\ref{neg:5}), $-0=-(-1 \oplus
  1)=1 \oplus -1 = -1 \oplus 1 = 0$.  For (\ref{neg:6}), recalling that
  the midpoint operation is a homomorphism, we conclude that so is the
  map $h(x)= x \oplus -x$.  Since $h(-1)=h(1)=0$, the universal
  property shows that $h$ must be the constant function with
  value~$0$.
\end{proof}

The universal property of an interval object is not strong enough
to allow us to define functions of two or more variables such as
multiplication. The situation is exactly analogous to that for
natural numbers objects, discussed in  \S\ref{SubSecPrelim}.
As in that case, the solution is to define a stronger
\emph{parametrised} notion of interval object.

\begin{defn}
  A \emph{parametrised interval object} is an m-convex body~$[a,b]$
  together with two global points $a$ and $b$ such that for every
  object~$P$ %of ``parameters''
  and every m-convex body~$A$ with maps %``parametric'' points
  $x_a, x_b: P \to A$ there is a unique map $h:P \times [a,b] \longrightarrow A$
  such that the equations below hold.
  \begin{align*}
  h(p,a) & =x_a(p), \\
  h(p,b) & =x_b(p), \\
  h(p,x \oplus y) & = m(h(p,x),h(p,y)).
 \end{align*}
  %  for all
  %  generalised elements~$p$ of~$P$ and $x$ and $y$ of $[a,b]$.
\end{defn}
The last equation says that $h$ is a homomorphism in its second
argument.  As in the non-parametric case, we use the operator
symbol~$\oplus$ to denote the midpoint operation of a parametrised interval object.
Easily, every parametrised interval object is an interval object.
Hence if both objects exist then they are isomorphic (via
midpoint homomorphisms).
Thus, a category with binary products may or may not have an interval
object, and if it does, the interval object may or may not have the
property of being parametrised.
The result below is analogous to Proposition~\ref{PropCccNno}.
\begin{prop}
\label{PropParam}
If $\CC$ is cartesian closed then any interval object is parametrised.
\end{prop}
The proof, which is routine given Proposition~\ref{PropPowers}, is omitted.

\begin{prop}[Multiplication]
  If $I=[-1,1]$ is a parametrised interval object
  then there is a unique map $(x,y) \mapsto x \cdot
  y:I \times I \longrightarrow I$ such that
  \begin{gather}
    x \cdot -1 = -x, \label{mul:1} \\
    x \cdot  1  = x, \label{mul:2} \\
    x \cdot (y \oplus z)  = x \cdot y \oplus x \cdot z. \label{mul:3} \\
      x \cdot 0 = 0, \label{mul:4} \\
      x \cdot y = y \cdot x, \label{mul:6} \\
     (x \cdot y) \cdot z = x \cdot (y \cdot z). \label{mul:7}
  \end{gather}
\end{prop}
\begin{proof}
  Existence and uniqueness follow from the universal property by
  considering the negation and identity functions as the parametric
  points.  All equations are proved using the same method.  We show
  the commutativity of multiplication, which is the most interesting.
  By the uniqueness condition of the free property of an interval
  object, it suffices to show that the defining equations are
  satisfied when the arguments are swapped, i.e.\ that
\begin{gather*}
-1 \cdot  y   =  -y, \qquad
1 \cdot  y   =  y, \qquad
(x \oplus y) \cdot z   =  (x \cdot z) \oplus (y \cdot z).
\end{gather*}
For the first equation, we have that $-1 \cdot -1 = -(-1)$ and $-1 \cdot
1 = -1$ and $-1 \cdot (z \oplus z') = (-1 \cdot z) \oplus (-1 \cdot
z')$, all by the definition of multiplication. Hence the map $z
\mapsto -1 \cdot z : I \to I$ is a homomorphism satisfying the
defining properties of ${-}: I \to I$.  Thus indeed $-1 \cdot z = -z$.
The second equation is established similarly.  To show the third
equation, which says that multiplication is a left-homomorphism, we
show that the maps $(x,y,z) \mapsto (x \oplus y) \cdot z$ and $(x,y,z)
\mapsto (x \cdot z) \oplus (y \cdot z)$ from $I \times I \times I$ to
$I$ are both right-homomorphisms agreeing on $-1$ and $1$, and thus
equal by the parametrised initiality of~$I$.  The map $(x,y,z)
\mapsto (x \oplus y) \cdot z$ is trivially a right-homomorphism,
because multiplication is.  For $(x,y,z) \mapsto (x \cdot z) \oplus (y
\cdot z)$, we have that
\begin{eqnarray*}
(x \cdot (z \oplus z')) \oplus (y \cdot (z \oplus z'))
  & = & ((x \cdot z) \oplus (x \cdot z'))
 \oplus ((y \cdot z) \oplus (y \cdot z')) \\
  & = & ((x \cdot z) \oplus (y \cdot z))
 \oplus ((x \cdot z') \oplus (y \cdot z')),
\end{eqnarray*}
where the first equation holds because $\cdot$ is a right-homomorphism
and the second by transposition.  For agreement on $-1$, we have that
\begin{gather*}
(x \cdot -1) \oplus (y \cdot -1)
 =  {-x} \oplus {-y}
 =  - (x \oplus y)
 =  (x \oplus y) \cdot -1
\end{gather*}
where the first and last equations hold by definition of
multiplication and the second by definition of negation.
%\[
%\begin{array}{lcll}
%(x \cdot -1) \oplus (y \cdot -1)  & = & {-x} \oplus {-y}
%& \mbox{(definition of multiplication)} \\
%& = & - (x \oplus y) & \mbox{(definition of negation)} \\
%& = & (x \oplus y) \cdot -1 & \mbox{(definition of multiplication).}
%\end{array}
%\]
And agreement on $1$ is similar.
\end{proof}
Note that the proofs of the above propositions do not make any use of
iterativity.

The above results  are just some first steps in using
the universal property of a parametrised interval object
to define and prove equalities between functions on the interval.
For another example,
given any midpoint-convex body $X$ and a parametrised interval
object $[0,1]$, we can define a binary convex combination function
$c:[0,1] \times X \times X \to X$ as the unique map which is a
homomorphism in its first argument and satisfies $c(0,x_0,x_1)=x_0$ and
$c(1,x_0,x_1)=x_1$. The intuition is that
$c(\lambda,x_0,x_1)=(1-\lambda) x_0 + \lambda x_1$.
This operation, in the case that $X$ is itself the interval, plays
a significant role in defining the \emph{primitive interval functions} in~\cite{escardo:simpson:interval:lics},
and in an extension of G\"odel's system T with a primitive type for the interval of real numbers~\cite{escardo:simpson:tlca2014}.

%The equational laws of binary convex combinations can be verified if the base

%\todo[inline]{In progress}

%This operation
%satisfies two laws (to be included) which capture the idea of
%convexity. We define a convex object to be \emph{any} object $X$
%endowed with a map $c:[0,1] \times X \times X \to X$ subject to these
%two laws. It is easy to see that the interval object (with its
%convexity structure obtained by the above method) is the convex set
%freely generated by two global points, where the homomorphisms of
%convex sets are the affine maps, that this, the ones that commute with
%the convexity structure. See the introduction.

We end this subsection with a brief summary of how
Proposition~\ref{PropAdj} relates to interval objects.
We first observe that any functor $F : \CC \to \CD$ with both left and
right adjoint preserves interval objects. By
Proposition~\ref{PropAdj}(\ref{PropAdj:2}), the functor
$\Ext{F}: \Ultra(\CC) \to \Ultra(\CD)$ has a right adjoint.
Easily $\Ext{F}$ preserves $\One$, so, as it also
preserves coproducts,
it preserves interval objects.
%Put more simply, $F : \CC \to \CD$ preserves interval objects.

Further, an immediate consequence of Proposition~\ref{PropAdj}(\ref{PropAdj:1})
is that  if $\CC$ is
a full reflective subcategory of~$\CD$ and if $\CD$ has an interval
object $I$, which is also an object of~$\CC$, then
$I$ is also an interval object in~$\CC$.

%% file: set.tex
\section{Consequences of iteration and cancellation}
\label{SecItCan}

In this section, we establish a couple of technical consequences
of the iteration and cancellation properties.

The first result says that,
in an iterative midpoint object, nested applications
of the infinitary $M$ can be reduced to a single
application over a sequence of terms,
each  expressed solely in terms
of the binary $m$.
This property is conveniently formulated
using $(n+1)$-ary operations, $m_n$,  derived
from the midpoint operation by
\begin{align*}
m_0(x) \: & = \: x ,\\
m_{n}(x_0, \dots, x_{n}) \: & = \: m(x_0, m_{n-1} (x_1, \dots,
x_{n})) && \text{$n \geq 1$}.
\end{align*}
Thus $m_1$ is just $m$ itself.

As in Section~\ref{notions}, we state and prove the
proposition below using the convenient language
of set theory. However,
with
a little care,
both the statement and proof generalize to
apply to iterative midpoint objects in any
category with finite products and parametrised nno.

\begin{prop}[Flattening]
\label{PropFlatten}
In any iterative midpoint object,
\begin{align}
\label{equation:flatten}
M_i(M_j \,x_{ij}) \: = \:
M_l \: m(\, &
   m_{(l+1)}(x_{0 \, (l+1)}, % x_{1 \, (l+1)} ,
           \dots , x_{l\, (l+1)}, x_{ll}), \,
\\ &
   m_{(l+1)}(x_{(l+1) \, 0}, % x_{(l+1) \, 1}
            \dots , x_{(l+1) \, l}, x_{ll}) \: ).
% M_l \left ( 2^{-(l+1)}.x_{ii
% \sum_{i = 0}^{2l} \, 2^{-(2l+2)} x_{i \, (2l-i)} \: + \:
%     \sum_{i = 0}^{2l + 1} \, 2^{-(2l+3)} x_{i \, (2l+1-i)} \right )
\end{align}
\end{prop}

\begin{proof}
Define
\begin{align*}
z_l \: = \:
  m( \:
   &  m_l(M_{i \geq (l+1)}\, x_{0i}, \: \dots \: ,
                \, M_{i \geq (l+1)}\, x_{\, (l-1) \, i}, \,
         M_{i \geq l} M_{j \geq l}\, x_{ij}), \\
   &  m_l(M_{i \geq (l+1)} \, x_{i0}, \: \dots \: , \,
         M_{i \geq (l+1)} \, x_{\, i \, (l-1)}, \,
         M_{i \geq l} M_{j \geq l} \,x_{ij}) \,
   ).
\end{align*}
Thus $z_0 = m(M_iM_j x_{ij}, M_iM_j x_{ij}) = M_iM_j x_{ij}$.
We show below that, for all $l \geq 0$,
\begin{equation}
\label{EqnFlatstep}
\begin{split}
z_l \: = \: m(\,
  & m(\, m_{(l+1)}(x_{0 \, (l+1)}, % x_{1 \, (l+1)} ,
           \dots , x_{l\, (l+1)}, x_{ll}), \,
         m_{(l+1)}(x_{(l+1) \, 0}, % x_{(l+1) \, 1}
            \dots , x_{(l+1) \, l}, x_{ll})), \\
  & z_{l+1} \, ),
\end{split}
\end{equation}
whence the desired equality,
\[
z_0 \: = \:
M_l \, m(\,
   m_{(l+1)}(x_{0 \, (l+1)}, % x_{1 \, (l+1)} ,
           \dots , x_{l\, (l+1)}, x_{ll}), \,
   m_{(l+1)}(x_{(l+1) \, 0}, % x_{(l+1) \, 1}
            \dots , x_{(l+1) \, l}, x_{ll}) \: ),
\]
follows immediately from iterativity.

It remains to show that~(\ref{EqnFlatstep}) holds.
First, observe the straightforward:
\begin{equation}
\label{EqnStraight}
\begin{split}
M_{i \geq l} M_{j \geq l} \,x_{ij} \: = \:
m(\,&   m(\, m(x_{\,l\,(l+1)},\, x_{\,(l+1) \, l}), \, x_{ll} \,) ,  \\
&   m( \, m(M_{i \geq (l+2)} x_{li},\, M_{i \geq (l+2)} x_{il}), \,
      M_{i \geq (l+1)} M_{j \geq (l+ 1)} \,x_{ij} \,  ) \, ).
\end{split}
\end{equation}
Now, equation~(\ref{EqnFlatstep}) is shown as follows, using~(\ref{EqnStraight})
to obtain the second equality.
\begin{align*}
z_l \:  & = \:
  m( \, m_l(M_{i \geq (l+1)}\, x_{0i}, \: \dots \: ,
                \, M_{i \geq (l+1)}\, x_{\, (l-1) \, i}, \,
         M_{i \geq l} M_{j \geq l}\, x_{ij}), \\
& \phantom{= \: m( \, } \,
        m_l(M_{i \geq (l+1)} \, x_{i0}, \: \dots \: , \,
                 M_{i \geq (l+1)} \, x_{\, i \, (l-1)}, \,
                 M_{i \geq l} M_{j \geq l} \,x_{ij}) \, )
\\
\displaybreak[0]
\\
& = \:
  m( \, m_l( m(x_{\, 0 \,(l+1)}, \, M_{i \geq (l+2)}\, x_{0i}), \:
        \dots \: , \,
    m (x_{\,(l-1) \, (l+1)}, \, M_{i \geq (l+2)}\, x_{\, (l-1) \, i}), \\
& \phantom{ = \: m( \, m_l( } \,
    m(\, m( \, m(x_{\,l\,(l+1)},\, x_{\,(l+1) \, l}),\, x_{ll} \,) , \\
&\phantom{ = \: m(\, m_l ( \, m( } \,
   m ( \, m(M_{i \geq (l+2)} x_{li},\, M_{i \geq (l+2)} x_{il}), \,
      M_{i \geq (l+1)} M_{j \geq (l+ 1)} \,x_{ij} \, ) \, ) \, ), \\
& \phantom{ = \: m ( \, } \,
   m_l(m(x_{\,(l+1) \, 0}, \, M_{i \geq (l+2)}\, x_{i0}), \:
         \dots \: ,
           m (x_{\, (l+1)\,(l-1)}, \, M_{i \geq (l+2)}\, x_{\, i \, (l-1) }), \\
& \phantom{ = \: m( \, m_l( } \,
    m(\, m( \, m(x_{\,l\,(l+1)},\, x_{\,(l+1) \, l}), \, x_{ll} \, ) , \\
&\phantom{ = \: m(\, m_l ( \, m( } \,
     m (\, m(M_{i \geq (l+2)} x_{li},\, M_{i \geq (l+2)} x_{il}), \,
        M_{i \geq (l+1)} M_{j \geq (l+ 1)} \,x_{ij}  \, ) \, ) \, ) \, )
\\
\displaybreak[0]
\\
& = \: m(\,  m( \,m_l(\, x_{\, 0 \,(l+1)},\: \dots \: , \, x_{\, (l-1) \, (l+1)}, \,
        m(\, m(x_{\,l\,(l+1)},\, x_{\,(l+1) \, l}), \, x_{ll} \,) \, ), \\
& \phantom{= \: m(\,  m( \,} \,
   m_l( \, M_{i \geq (l+2)}\, x_{0i}, \: \dots \:
        ,\, M_{i \geq (l+2)}\, x_{\, (l-1) \, i}, \, \\
& \phantom{= \: m(\,  m( \, m_l( \,} \,
   m (\, m(M_{i \geq (l+2)} x_{li},\, M_{i \geq (l+2)} x_{il}), \,
         M_{i \geq (l+1)} M_{j \geq (l+ 1)} \,x_{ij} \, ) \, ) \, ), \\
& \phantom{= \: m(\,} \,
m( \, m_l(\, x_{\,(l+1)\, 0 },\: \dots \: , \,
          x_{\, (l+1)\, (l-1) }, \,
        m(\, m(x_{\,l\,(l+1)},\, x_{\,(l+1) \, l}),\, x_{ll} \, )\,), \\
& \phantom{= \: m(\,  m( \,} \,
   m_l( \, M_{i \geq (l+2)}\, x_{i0}, \: \dots \:
        ,\, M_{i \geq (l+2)}\, x_{\, i\, (l-1) }, \, \\
& \phantom{= \: m(\,  m( \, m_l( \,} \,
   m (\, m(M_{i \geq (l+2)} x_{li},\, M_{i \geq (l+2)} x_{il}), \,
      M_{i \geq (l+1)} M_{j \geq (l+ 1)} \,x_{ij} \, ) \, ) \, ) \, )
\\
\displaybreak[0]
\\
& = \: m(\,  m( \,m_l(\, x_{\, 0 \,(l+1)},\: \dots \: , \, x_{\, (l-1) \, (l+1)}, \,
        m(m(x_{\,l\,(l+1)},\, x_{\,(l+1) \, l}), \, x_{ll}) \, ), \\
& \phantom{= \: m(\,  m( \,} \,
m_l(\, x_{\,(l+1)\, 0 },\: \dots \: , \,
          x_{\, (l+1)\, (l-1) }, \,
        m(m(x_{\,l\,(l+1)},\, x_{\,(l+1) \, l}), \, x_{ll}) \, ) \, ), \\
& \phantom{= \: m(\,} \,
m(\, m_l( \, M_{i \geq (l+2)}\, x_{0i}, \: \dots \:
        ,\, M_{i \geq (l+2)}\, x_{\, (l-1) \, i}, \, \\
& \phantom{= \: m(\,  m( \, m_l( \,} \,
   m (\, m(M_{i \geq (l+2)} x_{li},\, M_{i \geq (l+2)} x_{il}), \,
      M_{i \geq (l+1)} M_{j \geq (l+ 1)} \,x_{ij} \, ) \, ), \\
& \phantom{= \: m(\,  m( \,} \,
   m_l( \, M_{i \geq (l+2)}\, x_{i0}, \: \dots \:
        ,\, M_{i \geq (l+2)}\, x_{\, i\, (l-1) }, \, \\
& \phantom{= \: m(\,  m( \, m_l( \,} \,
   m (\, m(M_{i \geq (l+2)} x_{li},\, M_{i \geq (l+2)} x_{il}), \,
      M_{i \geq (l+1)} M_{j \geq (l+ 1)} \,x_{ij} \, ) \, ) \, ) \, )
\\
\displaybreak[0]
\\
& = \: m(\,  m( \,m_l(\, x_{\, 0 \,(l+1)},\: \dots \: , \, x_{\, (l-1) \, (l+1)}, \,
        m(x_{\,l\,(l+1)},\, x_{ll}) \,), \\
& \phantom{= \: m(\,  m( \,} \,
m_l(\, x_{\,(l+1)\, 0 },\: \dots \: , \,
          x_{\, (l+1)\, (l-1) }, \,
        m(x_{\,(l+1) \, l}, \, x_{ll}) \, ) \, ), \\
& \phantom{= \: m(\,} \,
m(\, m_l( \, M_{i \geq (l+2)}\, x_{0i}, \: \dots \:
        ,\, M_{i \geq (l+2)}\, x_{\, (l-1) \, i}, \,
\\ & \phantom{= \: m(\,  m( \, m_l( \,} \,
   m (M_{i \geq (l+2)} x_{li},\, M_{i \geq (l+1)} M_{j \geq (l+ 1)} \,x_{ij}) \,), \\
& \phantom{= \: m(\,  m( \,} \,
   m_l( \, M_{i \geq (l+2)}\, x_{i0}, \: \dots \:
        ,\, M_{i \geq (l+2)}\, x_{\, i\, (l-1) }, \,
\\ & \phantom{= \: m(\,  m( \, m_l( \,} \,
   m (M_{i \geq (l+2)} x_{il}, \, M_{i \geq (l+1)} M_{j \geq (l+ 1)} \,x_{ij}) \, ) \, ) \, )
\\
\displaybreak[0]
\\
& = \: m(\,  m( \,m_{l+1}(\, x_{\, 0 \,(l+1)},\: \dots \: ,
% \, x_{\, (l-1) \, (l+1)}, \,
        x_{\,l\,(l+1)},\, x_{ll} \,), \,
%\\ & \phantom{= \: m(\,  m( \,} \,
m_{l+1}(\, x_{\,(l+1)\, 0 },\: \dots \: , \,
% x_{\, (l+1)\, (l-1) }, \,
          x_{\,(l+1) \, l}, \, x_{ll}) \, ), \\
& \phantom{= \: m(\,} \,
m(\, m_{l+1}( \, M_{i \geq (l+2)}\, x_{0i}, \: \dots \: , \,
%             M_{i \geq (l+2)}\, x_{\, (l-1) \, i}, \,
% \\ & \phantom{= \: m(\,  m( \, m_l( \,} \,
   M_{i \geq (l+2)} x_{li},\, M_{i \geq (l+1)} M_{j \geq (l+ 1)} \,x_{ij}), \\
& \phantom{= \: m(\,  m( \,} \,
   m_{l+1}( \, M_{i \geq (l+2)}\, x_{i0}, \: \dots \: , \,
%        ,\, M_{i \geq (l+2)}\, x_{\, i\, (l-1) }, \,
%\\ & \phantom{= \: m(\,  m( \, m_l( \,} \,
   M_{i \geq (l+2)} x_{il}, \, M_{i \geq (l+1)} M_{j \geq (l+ 1)} \,x_{ij}) \, ) \, )
\\
\displaybreak[0]
\\
& = \: m(\,  m( \,m_{l+1}(\, x_{\, 0 \,(l+1)},\: \dots \: ,
% \, x_{\, (l-1) \, (l+1)}, \,
        x_{\,l\,(l+1)},\, x_{ll} \,), \,
%\\ & \phantom{= \: m(\,  m( \,} \,
m_{l+1}(\, x_{\,(l+1)\, 0 },\: \dots \: , \,
% x_{\, (l+1)\, (l-1) }, \,
          x_{\,(l+1) \, l}, \, x_{ll}) \, ), \\
& \phantom{= \: m(\,} \,
z_{l+1} ).
\end{align*}
\end{proof}
\noindent
We do not know if the equation established by the proposition
is more generally true for arbitrary supermidpoint objects.

The second result in this section  illuminates the
power of the cancellation axiom for supermidpoint objects in the special case of the category of sets---cancellation is
equivalent to
an important and useful approximation property.

\begin{prop}[Approximation]
\label{PropSetApprox}
For any supermidpoint set  $X$, the following are equivalent.
\begin{enumerate}
\item \label{PropSetApprox:1} $X$ is cancellative.

\item \label{PropSetApprox:2} $X$ satisfies the following
``approximation'' property.

For all sequences $(x_i)$, $(y_i)$, $(z_i)$, $(w_i)$
of points of $X$, if,
for all $n \geq 0$,
\begin{eqnarray*}
m_n(x_0, \dots, x_{n-1}, z_n) & = & m_n(y_0, \dots, y_{n-1}, w_n)
\end{eqnarray*}
then $M(x_0,x_1,x_2, \dots) = M(y_0,y_1,y_2, \dots)$.
\end{enumerate}
\end{prop}
% We do not know if
% Proposition~\ref{PropSetApprox} holds in weaker categorical
% settings.

\begin{proof}
To prove (\ref{PropSetApprox:2}) implies (\ref{PropSetApprox:1}),
assume $M: X^\omega \to X$ satisfies approximation.
To show that $(X,m)$ is cancellative, assume $m(x,z) = m(y,z)$.
We prove by induction that, for all $n \geq 0$,
\begin{eqnarray}
\label{EqnShowCanc}
m_n(\overbrace{x, \dots, x}^{n \;\mathit{times}}, z)  & = &
m_n(\overbrace{y, \dots ,y}^{n \;\mathit{times}}, z)
\end{eqnarray}
The $n = 0$ and $n = 1$ cases
% (saying $z = z$ and $m(x,z) = m(y,z)$ respectively)
are
immediate. For $n \geq 2$, assume the equality holds for
lower $n$.
Write $w =
m_{n-2}(x, \dots,x,z)$. By the induction hypothesis,
$w = m_{n-2}(y, \dots,y,z)$
and $m(x,w) = m_{n-1}(x, \dots,x,z) = m_{n-1}(y, \dots,y,z) =
m(y,w)$. So:
% (using the idempotency and commutativity
%axioms of Definition~\ref{DefMidAlg} without mention):
\[
\begin{array}{lcll}
m_n(x,\dots,x,z) & = & m(x,m(x,m_{n-2}(x, \dots,x,z)))  \\
& = & m(x,m(x,w))  \\
& = & m(m(x,x),m(x,w)) &  \\
& = & m(m(x,x),m(y,w)) & \text{induction hypothesis} \\
& = & m(m(x,y),m(x,w)) & %\text{(transposition)}
                                                \\
& = & m(m(x,y),m(y,w)) & \text{induction hypothesis} \\
& = & m(m(y,y),m(x,w)) & %\text{(transposition)}
                                                \\
& = & m(m(y,y),m(y,w)) & \text{induction hypothesis} \\
& = & m(y,m(y,w)) & \\
& = & m(y,m(y,m_{n-2}(y, \dots,y,z))) \\
& = & m_n(y,\dots,y,z).
\end{array}
\]
Thus (\ref{EqnShowCanc}) holds. Hence, by approximation,
$M(x,x,x,\dots) = M(y,y,y,\dots)$, i.e. $x=y$.
This proves cancellation.

For the converse implication, assume
$(X,m)$ is cancellative and hence iterative.
We first observe that the following
``one-sided approximation'' property holds:
if, for all $n \geq 0$, it holds %there exists $w_n$ such
that
$x  =  m_n(y_0, \dots, y_{n-1}, w_n)$
then $x = M(y_0,y_1,y_2, \dots)$.
Indeed, if the premise is satisfied then, by
repeated applications of cancellation,
$w_0 = m(y_0, w_1)$, $w_1 = m(y_1, w_2)$, etc. Hence
$x = M(y_0,y_1,y_2, \dots)$, by iterativity.

Next, we prove the following equality. For all $(x_i)$, $(y_i)$
and $n \geq 0$:
\begin{equation}
\label{EqnSkew}
\begin{split}
M_i(m(x_i,y_i)) \: = \:
m(\, & m_{n+1}(x_0, \dots, x_n, y_n), \\
     & M(y_0, \dots, y_{n-1}, m(x_{n+1},y_{n+1}), m(x_{n+2},y_{n+2}), \dots) \, ).
\end{split}
\end{equation}
The proof is by induction on $n$. When $n = 0$, we have:
\begin{align*}
\text{r.h.s.} \: & = \:
m(m(x_0,y_0), M(m(x_1,y_1), m(x_2,y_2), \dots))  \\
& = \:  M(m(x_0,y_0),m(x_1,y_1), m(x_2,y_2), \dots).
\end{align*}
For $n > 0$, we have:
\begin{align*}
\text{r.h.s.} \: & \:  =
m( \, m(x_0, m_{n}(x_1, \dots, x_n, y_n)),\, \\
& \phantom{\: = m(\, } \,
m(y_0, M(y_1, \dots, y_{n-1}, m(x_{n+1},y_{n+1}), m(x_{n+2},y_{n+2}), \dots)) \, ) \\
& = \:
m(\, m(x_0,y_0),\,  \\
& \phantom{\: = m(\, } \,
 m(\, m_{n}(x_1, \dots, x_n, y_n),\, \\
& \phantom{\: = m(\, m(\,  } \,
                  M(y_1, \dots, y_{n-1},
                   m(x_{n+1},y_{n+1}), m(x_{n+2},y_{n+2}), \dots)\, ) \, ) \\
& = \:
m(\, m(x_0,y_0),\, M(m(x_1,y_1), m(x_2,y_2), \dots) \,)
  & & \text{ind.\ hyp.} \\
& = \:  M(m(x_0,y_0),m(x_1,y_1), m(x_2,y_2), \dots).
\end{align*}

Finally, to prove approximation, suppose that,
for all $n \geq 0$,
\begin{eqnarray}
\label{EqnAppAss}
m_n(x_0, \dots, x_{n-1}, z_n)  & = &
 m_n(y_0, \dots, y_{n-1}, w_n).
\end{eqnarray}
We show below that
\begin{equation}
\label{EqnApproxZ}
m(M_i\, x_i,  \, M_i\, z_i ) \: = \:
m(M_i \, y_i, \, M_i \, z_i),
\end{equation}
from which the desired equation,
$M_i x_i = M_i y_i$, follows
by cancellation.

To verify (\ref{EqnApproxZ}), we have:
\begin{align*}
\text{l.h.s.}  \:
&  = \:
  M_i(m(x_0,z_0), m(x_1,z_1), m(x_2,z_2), \dots)  \\
& = \:
m(\, m_{n+1}(x_0, \dots, x_n, z_n),\, \\
& \phantom{= \: m(\, } \,
M(z_0, \dots, z_{n-1}, m(x_{n+1},z_{n+1}), m(x_{n+2},z_{n+2}), \dots)\,)
& &  \text{by (\ref{EqnSkew})} \\
& = \:
m(\, m_{n+1}(y_0, \dots, y_n, w_n),\, \\
& \phantom{= \: m(\, } \,
M(z_0, \dots, z_{n-1}, m(x_{n+1},z_{n+1}), m(x_{n+2},z_{n+2}), \dots) \,)
& & \text{by (\ref{EqnAppAss})} \\
& =  \:
M( \, m(y_0,z_0), \dots, m(y_{n-1},z_{n-1}), m(y_n, w_n), \\
& \phantom{= \: M(\, } \,
m(x_{n+1},z_{n+1}), m(x_{n+2},z_{n+2}), \dots)
& & \text{by (\ref{EqnSkew})} \\
& = \: m_n(m(y_0,z_0), \dots, m(y_{n-1},z_{n-1}), w'_n),
\end{align*}
where $w'_n = M(m(y_n, w_n),
m(x_{n+1},z_{n+1}), m(x_{n+2},z_{n+2}), \dots)$. But the above is
true for any $n \geq 0$. Hence, by one-sided approximation,
we can continue:
\begin{align*}
\text{l.h.s.} \: &
= \:  M_i(m(y_0,z_0), m(y_1,z_1), m(y_2,z_2), \dots)  \\
& =  \: m(M(y_0,y_1,y_2, \dots), M_i(z_0,z_1,z_2, \dots)),
\end{align*}
as required.
\end{proof}
\noindent
In contrast to all previous results, although
the statement of the above proposition can be formulated in any
category with finite products and parametrised
natural-numbers object, its proof does not
go through at this level of generality. In particular,
the ``repeated applications of cancellation'', used to establish
one-sided approximation, cannot be performed in
an arbitrary such category.
Nevertheless, the proof above
can be directly formalized within the internal logic of a
locally-cartesian-closed category,
%(which is, essentially,
%Martin-L{\"o}f type theory, see~\cite{Seely_1984,hofmann:lccc:typetheory},
and hence in any elementary topos. Indeed
Proposition~\ref{PropSetApprox} will be fundamental to the
development in Section~\ref{topos}, where the fact that it holds
in any topos is crucial.

\section{Iterative midpoints in the category of sets}
\label{SecIntInSet} \label{sets}

In this section we study iterative midpoint objects in the category
of sets, establishing a connection with so-called superconvexity,
and proving that the standard interval of real numbers is indeed
the interval object in the category.

\subsection{Superconvex sets}

\newcommand{\superind}{\operatorname{w}}
\newcommand{\conab}{$\operatorname{\mathsf{K}}$}
\newcommand{\interp}{\operatorname{val}}

Superconvexity, in the sense of Rod\'e~\cite{rode:superconvex} (see
also K\"onig \cite{koenig:superconvex}), is an algebraically formulated notion
of closure under countable convex combinations.

A \emph{superconvex weight function}, over a set $X$,
is a pair $(I, \lambda)$ where $I$ is an arbitrary
index set, and $\lambda \colon
I \to [0,1]$ satisfies $\sum_{i \in I} \lambda_i = 1$ (which implies
that its  \emph{support} $\{i \in I \mid \lambda_i \neq 0\}$
is countable).
A \emph{convex weight function} is a superconvex weight function
with finite support. A \emph{dyadic convex weight function}
is a convex weight function in which each $\lambda_i$ is a
dyadic rational. We shall simultaneously develop
notions of superconvex set, convex set and dyadic convex set.
To do this conveniently, we use {\conab} to stand for
any one of the terms: superconvex, convex or dyadic convex.

A \emph{{\conab} set} is given by a set $X$ together
with a \emph{{\conab} combination} operation
$\supersym^X$ that maps each triple
$(I, \lambda, x)$, where $(I, \lambda)$ is a
{\conab} weight function and $x \colon I \to X$ is a function,
to an element $\supersym^X(I,\lambda, x) \in X$,
written $\lsupercom{X}{i}{I}{\lambda_i}{x_i}$.
The operation $\supersym^X$ must satisfy:
\begin{gather}
\label{EqnSuperA}
\dlsupercom{X}{i}{\One}{1}{x} = x, \\
\label{EqnSuperB}
\dlsupercom{X}{i}{I}{\lambda_i}
         {\left(\dlsupercom{X}{j}{J}{\mu_{ij}}{x_j}\right)} =
       \dlsupercom{X}{j}{J}{\left(\sum_{i \in I} \lambda_i \mu_{ij}\right)}{x_j}.
\end{gather}
It is easily seen that,
for each of the three cases for {\conab}, in the right-hand side of
equation (\ref{EqnSuperB}), $\supersym^X$  is indeed applied to
a {\conab} weight function.
In the case of convex combinations, it will
be convenient to write
$\lambda_1 \cdot x_1 + \dots + \lambda_k \cdot x_k$
for the value $\supersym^X(\{1, \dots, k\}, \lambda, x)$.

Lemma~\ref{LemKoenig} below shows that the equations of
{\conab} convexity,~(\ref{EqnSuperA}) and~(\ref{EqnSuperB}), are
complete
in a natural sense. In its formulation, the  \emph{indicator}
of a {\conab} combination $(I, \lambda , x)$ is the
function $\superind(I, \lambda, x): X \to [0,1]$
defined by
\[
\superind(I, \lambda, x)(y) \: = \:
\sum_{i \in I} \, \lambda_i  \delta_{yx_i}
\]
where $\delta_{yx_i}$ is the Kronecker delta.
\begin{lemma}(Completeness)
\label{LemKoenig}
For any two {\conab} combinations $(I, \lambda, x)$ and
$(J, \mu, y)$ over a {\conab} set $X$, if
$\superind(I, \lambda, x) = \superind(J, \mu, y)$ then
$\supercom{i}{I}{\lambda_i}{x_i} = \supercom{j}{J}{\mu_j}{y_j}$.
\end{lemma}
For a proof see~\cite[Lemma~1.1]{koenig:superconvex}.
The reader is also referred to~\cite{koenig:superconvex} for a discussion
of examples and applications of superconvex sets, although
we shall also review a range of examples
in~{\S}\ref{SecApplySuperconvex}.

A function $f \colon X \to Y$, between superconvex sets, is
said to be \emph{superaffine} if it preserves superconvex
combinations, i.e.
\[
f\left(\dlsupercom{X}{i}{I}{\lambda_i}{x_i}\right) \: = \:
\dlsupercom{Y}{i}{I}{\lambda_i}{f(x_i)}.
\]
Similarly, a function between (dyadic) convex sets is said to be
\emph{(dyadic) affine} if it preserves (dyadic) convex combinations.

A \emph{proper binary {\conab} combination} is one of the
form $\lambda \cdot x + \mu \cdot y$ where the weights $\lambda,\mu$ are taken from
the open interval $(0,1)$. A {\conab} set is said to be \emph{cancellative} if,
whenever
$\lambda \cdot x + \mu \cdot z = \lambda \cdot y + \mu \cdot z$ for a proper
binary combination, then
$x = y$.
\begin{lemma}
\label{LemSuperCan}
A {\conab} set $X$ is cancellative if and only if,
for all $x,y,z \in X$, it holds that \linebreak[3]
$\:\frac{1}{2} \cdot x + \frac{1}{2} \cdot z = \frac{1}{2} \cdot y + \frac{1}{2} \cdot z$
implies $x = y$
\end{lemma}
For a proof, see~{\S}3 of~\cite{koenig:superconvex}.

A superconvex set is said to be \emph{iterative} if,
for every sequence of proper binary combinations
satisfying $x_i = \lambda_i \cdot y_i + \mu_i \cdot x_{i+1}$, for which
the weights $\mu_i$ satisfy the condition $\lim_{n \to \infty} \left(\prod_{i < n} \mu_i \right)= 0$,
the equality
\begin{equation}
\label{equation:iterative-superconvex}
x_0 \: = \:
\supercom{i}{\SetNat}{\lambda_i \! \left(\prod_{j < i} \mu_j\right)}{y_i}
\end{equation}
holds. In this definition, the limit condition is necessary for the right-hand side~\eqref{equation:iterative-superconvex} to be a superconvex combination.
% (Proposition~\ref{PropSetEqns}(\ref{eqn:iter})),
Any cancellative
superconvex set is automatically iterative, but not vice versa; a fact which follows from
the analogous statement for midpoint objects (Proposition~\ref{PropSetEqns}(\ref{eqn:iter})) together with the lemma below.
\begin{lemma}
\label{LemSuperIt}
A superconvex set is iterative
if and only if, for all sequences $(x_i)_i$ and $(y_i)_i$ satisfying
\begin{equation}
  \label{LemSuperIt:assumption}
  x_i \: = \: \frac{1}{2} \cdot y_i + \frac{1}{2} \cdot x_{i+1},
\end{equation}
  it holds
that $x_0 = \supercom{i}{\SetNat}{2^{-(i+1)}}{ y_i}$.
\end{lemma}
\begin{proof}
For the non-trivial right-to-left implication, assume the right-hand side of the if-and-only-if holds and
consider any sequence of proper binary convex combinations
satisfying
\begin{equation}
\label{EqnEasySoFar}
x_i \: = \: \lambda_i \, \cdot \, y_i  \, +  \, \mu_i \, \cdot \, x_{i+1},
\end{equation}
for which the condition $\lim_{n \to \infty} \left(\prod_{i < n} \mu_i \right)= 0$ applies.
We need to prove equation~\eqref{equation:iterative-superconvex}.

Let $n_0 \geq 0$, be the smallest number for which
$\prod _{i \leq n_0} \mu_i \leq \frac{1}{2}$, which exists by the limit assumption.
Using~\eqref{EqnEasySoFar} for $i = 0,\dots, n_0$ in combination, expand $x_0$ as a convex combination involving $y_0, \dots, y_{n_0}$ and $x_{n_0+1}$:
\begin{align}
\nonumber
x_0 ~ & = ~ \sum_{j \leq n_0} \left(\lambda_j \prod_{i < j} \mu_i   \right)  \cdot \, y_j + \left(\prod_{i \leq n_0} \mu_i\right) \cdot x_{n_0 +1} \\
\nonumber
& = ~ \sum_{j < n_0} \left(\lambda_j \prod_{i < j} \mu_i   \right)  \cdot \, y_j  ~ +~  (\gamma_1 + \delta_1)  \cdot \, y_{n_0} ~ +~
\left(\prod_{i \leq n_0} \mu_i \right) \cdot x_{n_0 +1} \, ,
\\
\intertext{where $\delta_1 :=  \frac{1}{2} - \prod_{i \leq n_0} \mu_i$ and $\gamma_1 := \left(\lambda_{n_0} \prod_{i < n_0} \mu_i\right) - \delta_1$,}
\label{equation:XZWO}
& = ~  \frac{1}{2} \cdot w_0 +  \frac{1}{2} \cdot z_1 \, ,
\end{align}
where
\begin{align*}
w_0 ~ & := ~ \sum_{j < n_0} \left(2 \lambda_j \prod_{i < j} \mu_i\right) \cdot \, y_j ~ + ~  2\gamma_1 \cdot y_{n_0} \\
\intertext{and}
z_1 ~ & := ~ 2\delta_1 \cdot y_{n_0} ~ + ~  \left(2\prod_{i \leq n_0} \mu_i \right)\cdot x_{n_0 +1}\, .
\end{align*}
%Similarly, taking $n_1 \geq n_0$ to be the smallest number for which
%$\prod _{i \leq n_1} \mu_i \leq \frac{1}{4}$, we can expand $z_1$ as
%\begin{align}
%\nonumber
%z_1 ~ & = ~ 2\delta_1 \cdot y_{n_0} ~ + ~
%\sum_{j = n_0+1}^{n_1} \left(2 \lambda_j \prod_{i < j} \mu_i\right)   \, \cdot \, y_j ~ + ~
%  \left(2\prod_{i \leq n_1} \mu_i\right)  \cdot x_{n_1 +1} \, .
%\end{align}
%In the case $n_1 = n_0$ (equivalently $\delta_1 \geq \frac{1}{4}$), we have
%\begin{align}
%\label{equation:ZOZTa}
%z_1 ~ & = ~ \frac{1}{2} \cdot y_{n_0} ~ + ~  \frac{1}{2} \cdot z_2 \, ,
%\end{align}
%where
%\[
%z_2 ~ = ~ \left(1 - 4\delta_1 \right ) \cdot y_{n_0} +  4\prod_{i \leq n_0} \mu_i \cdot x_{n_0 +1} \, .
%\]
%Otherwise, if $n_1 > n_0$ (equivalently $\delta_1 < \frac{1}{4}$),we have,
%\begin{align}
%\nonumber
%z_1 ~ & = ~ 2\delta_1 \cdot y_{n_0} ~ + ~
%\sum_{j = n_0+1}^{n_1 -1} \left(2 \lambda_j \prod_{i < j} \mu_i\right)   \, \cdot \, y_j ~ + ~
%(\gamma_2 + \delta_2) \cdot y_{n_1} ~ + ~
%  \left(2\prod_{i \leq n_1} \mu_i\right)  \cdot x_{n_1 +1} \, ,
%\\
%\intertext{where $\delta_2 := \frac{1}{2} - 2\prod_{i \leq n_1} \mu_i$ and $\gamma_2 := \left(2 \lambda_j \prod_{i < n_1}
%\mu_i\right) - \delta_2$,}
%\label{equation:ZOZTb}
%& = ~ ~ \frac{1}{2} \cdot  \left( 4\delta_1 \cdot y_{n_0} ~ + ~
%\sum_{j = n_0+1}^{n_1 -1} \left(4 \lambda_j \prod_{i < j} \mu_i\right)   \, \cdot \, y_j ~ + ~
%2\gamma_2  \cdot y_{n_1}\right) ~ + ~ \frac{1}{2} \cdot z_2
%\end{align}
%where
%\[
%z_2~ = ~ 2\delta_2 \cdot y_{n_1} ~ + ~
%  \left(4\prod_{i \leq n_1} \mu_i\right)  \cdot x_{n_1 +1} \, .
%\]

Continuing in the same vein, for each $i \geq 1$ in turn, define $n_i$ to be the least number such that
$\prod _{j \leq n_i} \mu_j \leq 2^{-i-1}$, and expand $z_i$ as
\[
z_i ~ = ~ \frac{1}{2} \cdot w_i  \; + \;  \frac{1}{2} \cdot z_{i+1},
\]
where each $w_i$ is given by  a convex combination of $y_{n_{i-1}}, \dots, y_{n_i}$.
With the inclusion of \eqref{equation:XZWO},
% and \eqref{equation:ZOZTb} (or \eqref{equation:ZOZTa}),
we obtain a system of equations that, by the main assumption, has a unique solution
\begin{equation}
\label{equation:superinterative:final}
x_0 ~ =  ~
\supercom{i}{\SetNat}{2^{-(i+1)}}{w_i}.
\end{equation}
By the derivation of
each $w_i$ as a convex combination of $y_{n_{i-1}}, \dots, y_{n_i}$, it holds that
the weight of $y_i$ in the superconvex combination on the right-hand side of~\eqref{equation:superinterative:final} is
$\lambda_i \prod_{j < i} \mu_j$. So by the superconvex instance of completeness (Lemma~\ref{LemKoenig}),
equation~\eqref{equation:iterative-superconvex} indeed holds.
\end{proof}

One consequence of the lemma above is an analogue of
Proposition~\ref{PropSetEqns}(\ref{eqn:iter}) for superconvex sets.
\begin{prop}
\label{prop:css-iterative}
Every cancellative superconvex set is iterative.
\end{prop}
\begin{proof}
If $X$ is a cancellative superconvex set, then
\[
M(x_0,x_1,x_2,\dots) ~ := ~ \sum_{i = 0}^\infty 2^{-(i+1)} \cdot x_i \, ,
\]
defines a cancellative supermidpoint structure on $X$.  By Proposition~\ref{PropSetEqns}(\ref{eqn:iter})
the corresponding  midpoint structure is iterative.
Hence, by Lemma~\ref{LemSuperIt}, the superconvex structure is also iterative.
\end{proof}

The main theorem of this section states that the
iterative midpoint sets
are exactly the iterative superconvex sets, and the same holds for the cancellative case.
Moreover, the associated categories, of midpoint homomorphisms
and superaffine maps respectively,
are isomorphic. In particular the category of m-convex bodies
is isomorphic to the category of cancellative
superconvex sets.
\pagebreak[3]
\begin{theorem}
\label{ThmSuperconvex}
\leavevmode
\begin{enumerate}
\item \label{ThmSuperconvex:1} If $X$ is an iterative superconvex set then $(X,m)$
is an iterative midpoint set, where
$m(x,y) = \frac{1}{2} \cdot x + \frac{1}{2} \cdot y$.
Moreover, if $X$ is cancellative then so is $(X,m)$.

\item \label{ThmSuperconvex:2} If $(X,m)$ is an iterative midpoint
set then there exists
a unique superconvex structure $\supersym$ on $X$ such that
$\frac{1}{2} \cdot x + \frac{1}{2} \cdot y = m(x,y)$. This
superconvex structure is iterative.
Moreover, if $(X,m)$ is cancellative then so is $X$.

\item \label{ThmSuperconvex:3} Let $(X,m)$ and $(X',m')$ be iterative midpoint
sets with their unique superconvex structures, then a function
$f \colon X \to X'$ is a midpoint homomorphism
if and only if it is superaffine.

\end{enumerate}
\end{theorem}
The proof of the theorem occupies
{\S}{\S}\ref{SubSecDyadic}--\ref{SubSecProofA}, after which applications
of it are discussed in Section~\ref{SecApplySuperconvex}.

\subsection{Dyadic convex sets and midpoints}
\label{SubSecDyadic}

We first establish an analogue of Theorem~\ref{ThmSuperconvex},
which expresses a coincidence between dyadic convex sets
and midpoint sets.
\begin{prop}
\label{PropDyadic}
\leavevmode
\begin{enumerate}
\item \label{PropDyadic:1} If $X$ is a dyadic convex set then $(X,m)$
is a midpoint set, where
$m(x,y) = \frac{1}{2} \cdot x + \frac{1}{2} \cdot y$.

\item \label{PropDyadic:2} If $(X,m)$ is a midpoint set then there exists
a unique dyadic convex structure $\supersym$ on $X$ such that
$\frac{1}{2} \cdot x + \frac{1}{2} \cdot y = m(x,y)$.

\item \label{propdyadic:3} \label{PropDyadic:3}
Let $(X,m)$ and $(X',m')$ be midpoint sets with
their unique dyadic convex structures, then a function
$f \colon X \to X'$ is a midpoint homomorphism
if and only if it
is dyadic affine.

\end{enumerate}
\end{prop}
We  give only an outline of the proof of this proposition, since the details are routine to fill in, and similar results appear in the literature~\cite{stone_postulates_1949,staton:icalp:2018}. For future reference, notice that the proof outlined below is constructive.

Statement~\ref{PropDyadic:1} is straightforward. One only needs to verify the
midpoint equations.
For statement~\ref{PropDyadic:2}, we first
show how to represent any dyadic combination,
\[q_1 \cdot x_1 + \dots + q_k \cdot x_k,\] using the midpoint
operation $m$. Let $t \geq 0$ be the smallest integer such that every
$q_l$ (where $1 \leq l \leq k$) can be expressed as an integer
fraction with denominator $2^t$. Then the dyadic convex combination
can be equivalently expressed as
\[\sum_{d = 1}^{2^t} \, 2^{-t} \,  \cdot  \, y_d \]
with each $y_d \in \{x_1, \dots, x_k\}$.
Such a combination is easily represented in terms of the binary midpoint operation $m$,
using an expression that is a full binary tree of height $t$.
Next, one proves, by induction on the height of trees, that
if two full term trees assign the same total weight to
each element $x_n$ then the terms are
provably equal using the midpoint equations.
This fact suffices to
verify that  equations~(\ref{EqnSuperA}) and~(\ref{EqnSuperB})
hold between the interpretations of dyadic convex combinations.
Finally, statement~\ref{PropDyadic:3} follows directly from the mutual
interpretations of dyadic convex combinations and midpoints.

Proposition~\ref{PropDyadic} allows one to use the convenient notation
of dyadic convex combinations instead of manipulating cumbersome
midpoint expressions.  We shall make frequent use of such notation in
our proof of Theorem~\ref{ThmSuperconvex}.

\subsection{Term algebras}
\label{SubSecTerms}

\newcommand{\binop}{\mathsf{o}}
\newcommand{\infop}{\mathsf{O}}
\newcommand{\algset}{\mathcal{T}}
\newcommand{\algsetf}{\mathcal{T}^{f}}

In the outline proof of Proposition~\ref{PropDyadic},
we made reference to term trees over the binary operation $m$.
To establish properties of iterative midpoints, we
again need to consider terms, but allowing also the infinitary
operation $M$. We introduce such terms formally.

For an index set $I$, we define
the free term algebra over a single binary operator
and a single infinitary operator. This
is the smallest set $\algset_{I}$ satisfying:
\begin{enumerate}
\item $I \subseteq \algset_{I}$,
\item if $s,t \in \algset_I$ then the pair $(s,t)$ is in
       $\algset_I$, and
\item if $(t_l)_l$ is a sequence of elements of $\algset_I$
      then the sequence $(t_l)_l$ is itself an
      element of $\algset_I$.
\end{enumerate}
Equivalently, $\algset_I$ is the set of well-founded
trees with the following
properties: each internal node is either binary branching or
$\omega$-branching, and each leaf is labelled with an element of $I$.
We write $\algsetf_I$ for the smallest subset of $\algset_I$ closed
under 1 and 2 above (i.e.\ the set of finite binary trees
with $I$-labelled leaves).

Given a function $\sigma \colon I \to \algset_J$, the
substitution function $[\sigma] \colon \algset_I \to \algset_J$
is defined by
\begin{align*}
[\sigma]i & = \sigma(i)\\
[\sigma](s,t)  & = ([\sigma]s,[\sigma]t)  \\
[\sigma](t_l)_l  & = ([\sigma]t_l)_l.
\end{align*}
If $X$ is a supermidpoint set,
with operations $M$ and
$m$, then the interpretation function
$\interp \colon \algset_I \times X^I \to X$
%relative to a function $x \colon I \to X$,
is defined by
\begin{align*}
\interp(i,x) & = x_i \\
\interp((s,t),x)  & = m(\interp(s,x), \interp(t,x))  \\
\interp((t_l)_l,x)  & = M_l \, \interp(t_l, x).
\end{align*}
\begin{lemma}
\label{LemSemSub}
For any $t \in \algset_I$, $\;\sigma \colon I \to \algset_J$,
and $y \colon J \to X$,
%it holds that $
\[
\interp([\sigma]t,y) =
\interp (t, \, i \mapsto \interp(\sigma(i), y)).\]
\end{lemma}
\begin{proof}
A straightforward induction on $t$.
\end{proof}

Analogously to the earlier definition for {\conab} combinations,
for $t \in \algset_I$, define
the \emph{weight}
$\superind(t) \colon I \to [0,1]$ by
\begin{align*}
\superind(i)(j) & = \delta_{ij}   \\
\superind((s,t))(j) & = \frac{1}{2} \superind(s)(j) +
                      \frac{1}{2} \superind(t)(j) \\
\superind((t_l)_l)(j)  & =
\sum_{l =0}^\infty \, 2^{-(l+1)} \superind(t_l)(j).
\end{align*}
Note that $\superind(t)$ is zero on all but countably many
elements of $I$ and also that $\sum_{i \in I} \, \superind(t)(i)\, = \, 1$.
\begin{lemma}
\label{LemIndSub}
For any $t \in \algset_I$ and
$\sigma \colon I \to \algset_J$ it holds that
\[
\superind([\sigma]t)(j) \: = \:
\sum_{i \in I} \, \superind(t)(i) \cdot \superind(\sigma(i))(j).
\]
\end{lemma}
\begin{proof}
A straightforward induction on $t$.
\end{proof}

The main goal of the section
is to prove Proposition~\ref{PropPhew} below, which
is an analogue
of Lemma~\ref{LemKoenig} for terms $t \in \algset$ interpreted
in iterative midpoint sets.
Accordingly, for the remainder of this section, we
assume that $X$ is an
iterative midpoint set.

As a first step, we show that every tree is semantically equal to one in which the infinitary operation is used once only, as the top-level constructor.
By a \emph{normal form} we mean
a sequence $(u_i)_i$ with each $u_i \in \algsetf_I$. The following generalises Proposition~\ref{PropFlatten} to trees with arbitrary nestings of the infinitary operation.
\begin{lemma}
\label{LemNormalize}
For any $s \in \algset_I$ there exists a normal form $t$
such that $\superind(s) = \superind(t)$ and,
for all $x \colon I \to X$, it holds that
$\interp(s,x) = \interp(t,x)$.
\end{lemma}
\begin{proof}
By induction on $s$.

If $s = i \in I$, then the required normal form $t$ is the constantly $s$ sequence.

If $s = (s_1,s_2)$ then let $t_1 =
(u_l)_l$ and $t_2 = (v_l)_l$ be the normal forms of
$s_1$ and $s_2$ respectively. Define
$t = ((u_l,v_l))_l$. Easily $\superind(s) = \superind(t)$.
That $\interp(s,x) = \interp(t,x)$ follows immediately from the supermidpoint
identity $m(M_l \, y_l, \, M_l \, z_l)\, = \, M_l \,m(y_l,z_l)$.

If $s = (s_m)_m$ then let $t_m = (u_{mn})_n$ be the normal form of $s_m$.
Following Proposition~\ref{PropFlatten},
define $t = (v_l)_l$ where, for each $l$,
\[
v_l  \: = \: (\, ( u_{0\, (l+1)}, \,
                   (u_{1\, (l+1)}, \dots (u_{l\, (l+1)},\, u_{ll}))), \:
                 ( u_{(l+1) \, 0}, \,
                   (u_{(l+1) \, 1}, \dots (u_{(l+1) \, l},\, u_{ll}))) \, ).
\]
Again, one easily calculates that $\superind(s) = \superind(t)$.
That  $\interp(s,x) = \interp(t,x)$ follows immediately from
Proposition~\ref{PropFlatten}.
\end{proof}

\begin{lemma}
\label{LemNormalComplete}
If $s$ and $t$ are normal forms with
$\superind(s) = \superind(t)$ then,
for all $x \colon I \to X$, it holds that
$\interp(s,x) = \interp(t,x)$.
\end{lemma}
\begin{proof}
We have $s = (u_l)_l$ and $t = (v_l)_l$ with each
$u_l$ and $v_l$ in $\algsetf_I$.
Then
\begin{equation}
\label{EqnExpands}
\begin{split}
\interp(s, x)
 & = M_l\, \interp(u_l, x) \\
 & = M_l \, (q_{l1} \cdot x_{i_{l1}} + \dots + q_{l {k_l}} \cdot x_{i_{lk_l}}),
\end{split}
\end{equation}
using the dyadic affine structure of $X$ to interpret each $u_l$.
Similarly,
\begin{align*}
\interp(t, x)
 & = M_l\, \interp(v_l, x) \\
 & = M_l \, (q'_{l1} \cdot x_{j_{l1}} + \dots + q'_{l {m_l}} \cdot x_{j_{lm_l}}).
\end{align*}
Below, we shall construct a sequence $(z_n)_n$  satisfying $z_0 = \interp(s, x)$
and, for every $n \geq 0$,
\begin{equation}
\label{EqnComplicated}
z_n \: = \:
m ( \, q'_{n 1} \cdot  x_{j_{ n1}} + \dots +
            q'_{n m_n}  \cdot  x_{j_{ n m_{\!n}}},  \, z_{n+1} \, )
\end{equation}
Whence, using iterativity for the second equality,
\[
\interp(s,x)
 = z_0
 = M_l \, (q'_{l1} \cdot x_{j_{l1}} + \dots + q'_{l {m_l}} \cdot x_{j_{lm_l}})  = \interp(t,x).
\]

It remains to find the $z_n$ and
verify equation~(\ref{EqnComplicated}). In doing so, we shall
ensure that each $z_n$ is an iterated midpoint of dyadic combinations,
\begin{equation}
\label{EqnZForm}
z_n \: = \: M_l \,
(q_{nl1} \cdot x_{i_{nl1}} + \dots + q_{nl k_{nl}} \cdot x_{i_{nlk_{\!nl}}})
\end{equation}
(where $q_{abc}$, $i_{abc}$ and $k_{ab}$ are distinguished from $q_{ab}$, $i_{ab}$ and
$k_a$ by the number of indices), for which the dyadic weights satisfy the equality, for all $i \in I$,
\begin{equation}
\label{EqnToMaintain}
\sum_{l = 0}^\infty \sum_{d = 1}^{k_{nl}} \, 2^{-(l+1)} \,  q_{nld}
\, \delta_{\, i \, i_{nld}} \: = \:
\sum_{l = 0}^\infty \sum_{d = 1}^{m_{n+l}} \, 2^{-(l+1)} \, q'_{\,(n+l)\, d} \,
\delta_{\, i \, j_{\,(n+l) \, d}} \, .
\end{equation}

For $z_0$, we have that $z_0 = \interp(s, x)$ is,
by~(\ref{EqnExpands}), explicitly of the
form~(\ref{EqnZForm}) by setting $k_{0l} := k_l$,
$\; i_{0 l d} := i_{ld}$ and
$q_{0ld} := q_{ld}$. Moreover,~(\ref{EqnToMaintain}) holds
because
\begin{align*}
\sum_{l = 0}^\infty \sum_{d = 1}^{k_{0l}}\, 2^{-(l+1)} \, q_{0ld}
\, \delta_{\, i \, i_{0ld}}
& = \: \sum_{l = 0}^\infty \sum_{d = 1}^{k_l} \,2^{-(l+1)}\,
        q_{ld}\, \delta_{\,i \, i_{ld}} \\
& = \:\superind(s)(i)  \\
& =  \: \superind(t)(i) \\
& = \: \sum_{l = 0}^\infty \sum_{d = 1}^{m_l} \, 2^{-(l+1)} \,
        q'_{ld} \, \delta_{\,i \, j_{ld}}.
\end{align*}

Next, given $z_n$ satisfying~\eqref{EqnZForm} and~\eqref{EqnToMaintain}, we construct $z_{n+1}$.
By considering just the $l = 0$ summand on the right-hand side of~\eqref{EqnToMaintain}, there exists $l' \geq 0$ such that, for
all $i \in I$,
\begin{equation}
\label{EqnFindlprime}
\sum_{l=0}^{l'}
\sum_{d = 1}^{k_{nl}} \, 2^{-(l+1)} \,  q_{nld}
\, \delta_{\, i \, i_{nld}} \: \geq \:
\sum_{d = 1}^{m_{n}} \, 2^{-1} \, q'_{n d} \,
\delta_{\, i \, j_{nd}} \, .
\end{equation}
Expanding the right-hand side of~(\ref{EqnZForm}),
we have:
\begin{equation}
\label{EqnLumpy}
\begin{split}
z_n \: = \:
& 2^{-1}\, \cdot \,
(q_{n01} \cdot x_{i_{n01}} + \dots + q_{n0 k_{n0}} \cdot x_{i_{n0k_{\!n0}}})
\, + \\
& 2^{-2} \, \cdot \,
(q_{n11} \cdot x_{i_{n11}} + \dots + q_{n1 k_{n1}} \cdot x_{i_{n1k_{\!n1}}})
\, + \\
& \dots \\
& 2^{-(l'+1)} \, \cdot \,
(q_{nl'1} \cdot x_{i_{nl'1}} + \dots + q_{nl' k_{nl'}} \cdot x_{i_{nl'k_{\!nl'}}})
\,
+ \\
& 2^{-(l'+1)} \, \cdot \,
  (M_{l > l'} \, \,
(q_{nl1} \cdot x_{i_{nl1}} + \dots + q_{nl k_{nl}} \cdot x_{i_{nlk_{\!nl}}})
).
\end{split}
\end{equation}
Let $t \geq 0$ be the smallest integer such that:
\begin{align}
\label{IneqA}
2^{-t} & \: \leq  \: 2^{-(l+1)} \, q_{nld} & & \text{for all $l \leq l'$ and
                                                   $1 \leq d \leq k_{nl}$, and} \\
\label{IneqB}
2^{-t} & \: \leq  \: 2^{-(l+1)} \, q'_{nd} & & \text{for all
                                                   $1 \leq d \leq m_n$.}
\end{align}
By~(\ref{IneqA}),
equation ~(\ref{EqnLumpy}) can be reformulated as:
\[
z_n \: = \:
\sum_{d = 1}^{{h_{l'}}} \, 2^{-t} \, \cdot \, x_{i'_d} \, + \,
{2^{-(l'+1)}} \,  \cdot \,
  (M_{l > l'} \, \,
(q_{nl1} \cdot x_{i_{nl1}} + \dots + q_{nl k_{nl}} \cdot x_{i_{nlk_{\!nl}}})\, ) \, ,
\]
for suitable $i'_d$, where, for any $l \geq 0$, we define
$h_{l} := 2^{t}(1 - 2^{-(l+1)})$ (the  case for general $l$ is used below).
But, by~(\ref{EqnFindlprime}) and~(\ref{IneqB}), this can, without
loss of generality, be reformulated as
\begin{equation}
\label{EqnUgly}
\begin{split}
z_n \:
= \: & 2^{-1}\,  \cdot\,
       (q'_{n1} \cdot x_{j_{n1}} + \dots + q'_{n {m_n}} \cdot x_{j_{nm_{\!n}}}) \, + \\
&
2^{-2}\, \cdot  \,
\left({\sum}_{d = h_0 + 1}^{h_1} \, 2^{2-t} \, \cdot \, x_{i'_d}\right)  \, + \\
&   \dots  \\
& 2^{-(l'+1)}\, \cdot  \,
\left({\sum}_{d = h_{l'-1}\! + \! 1}^{h_{l'}} \, 2^{l'+1-t}
\,  \cdot \, x_{i'_d}\right)  \, + \\
&
{2^{-(l'+1)}} \, \cdot \,
  (M_{l > l'} \, \,
(q_{nl1} \cdot x_{i_{nl1}} + \dots + q_{nl k_{nl}} \cdot x_{i_{nlk_{\!nl}}})\, ) .
\end{split}
\end{equation}
Thus we define
\[
z_{n+1} \: = \:
M_{l \geq 1} \, \left(
\begin{array}{ll}
{\sum}_{d = h_{l}\! + \! 1}^{h_{l+1}} \, 2^{l+1 -t} \, \cdot \, x_{i'_d}
& \text{if $l \leq l'$} \\
q_{nl1} \cdot x_{i_{nl1}} + \dots + q_{nl k_{nl}} \cdot x_{i_{nlk_{\!nl}}}
& \text{otherwise}
\end{array}
\right)\, .
\]
This explicitly exhibits $z_{n+1}$ in the form required by~(\ref{EqnZForm}).
Equation~(\ref{EqnComplicated}) is an immediate consequence of~(\ref{EqnUgly})
and the equality~(\ref{EqnToMaintain}) for $n+1$ follows
from~(\ref{EqnToMaintain}) for $n$ and equation~(\ref{EqnUgly}).
\end{proof}

\begin{prop}
\label{PropPhew}
If $X$ is an iterative midpoint set and
$\superind(s) = \superind(t)$
%, where $s,t \in \algset$,
then,
for all $x \colon I \to X$, it holds that
$\interp(s,x) = \interp(t,x)$.
\end{prop}
\begin{proof}
Immediate from Lemmas~\ref{LemNormalize}
and~\ref{LemNormalComplete}.
\end{proof}

\subsection{Proof of Theorem~\ref{ThmSuperconvex}}
\label{SubSecProofA}

Statement~\ref{ThmSuperconvex:1}  of the theorem is obvious: every superconvex set is \emph{a fortiori} a midpoint set, and the
notions of cancellativity and  iterativity  for midpoint and superconvex sets coincide by
Lemmas~\ref{LemSuperCan} and~\ref{LemSuperIt} respectively. In fact, only the easy directions of these lemmas are needed to obtain statement \ref{ThmSuperconvex:1}.

For statement \ref{ThmSuperconvex:2}, we need to show that every iterative midpoint set
carries a unique iterative superconvex structure extending the midpoint structure.
Suppose that $(X,m)$ is an
iterative midpoint set.
We write $M$ for its supermidpoint operation, and we shall
also use the dyadic convex structure on $X$ given in
{\S}\ref{SubSecDyadic}.
We need to define an operation
$\supersym$ on superconvex weight functions,
and to show that it satisfies the equalities demanded
of superconvex combinations.

For later convenience, the next lemma is stated simultaneously for
superconvex, convex and dyadic convex weight functions.
\begin{lemma}
\label{LemFindNormal}
For any {\conab} weight function $(I, \lambda)$,
one can find a dyadic convex weight
function $(I, \rho)$, and a
{\conab} weight function $(I, \mu)$ satisfying, for
all $i \in I$,
\[
\lambda_i \: = \: \frac{1}{2} \rho_i + \frac{1}{2} \mu_i.
\]
% Moreover $\mu$ is uniquely determined by $\rho$.
\end{lemma}
\begin{proof}
We define two finite
sequences $\rho^0, \dots, \rho^{k}: I \to [0,1]$, of functions valued in the dyadic rationals, and
$\mu^0, \dots, \mu^{k} \colon
I \to [0,2]$,
where $k \geq 1$, such that:
each $\rho^l$ has finite support;
for all $i \in I$,
$\: \frac{1}{2}\rho^l_i + \frac{1}{2}\mu^{l}_i = \lambda_i$;
and $\sum_{i \in I} \rho^l_i \leq 1$
(hence $\sum_{i \in I} \mu^l_i \geq 1$).
Define $\mu^0_i := 2\lambda_i$, and
$\rho^0_i = 0$. If $\sum_{i \in I} \rho^l_i = 1$ then define $k := l$ and stop the construction of the sequences.
Otherwise if $\sum_{i \in I} \rho^l_i < 1$ then,
to define $\rho^{l+1}$ and $\mu^{l+1}$ from
$\rho^l$ and $\mu^l$, let $t\geq 0$ be the smallest integer
for which there exists $i \in I$ with $2^{-t} \leq \mu^l_i$, and let
$i_l$ be one such $i$. Now define
\begin{align*}
\rho^{l+1}_j & \:  = \:
\left\{ \begin{array}{ll}
        \rho^l_j + 2^{-t} & \text{if $j = i_l$,} \\[2pt]
        \rho^l_j          & \text{otherwise,}
\end{array} \right.
\\ \\
\mu^{l+1}_j & \:  = \:
\left\{ \begin{array}{ll}
        \mu^l_j - 2^{-t} & \text{if $j = i_l$,} \\[2pt]
        \mu^l_j          & \text{otherwise.}
\end{array} \right.
\end{align*}
It is not obvious from the above definition that $\rho^{l+1}_j \leq 1$.
But, in fact, a stronger property holds: we are guaranteed to eventually
arrive at $k$ with $\sum_{i \in I} \rho^k_i = 1$. Hence, for all $l \leq k$, we have
$\sum_{i \in I} \rho^l_i \leq 1$, hence $\rho^{l+1}_j \leq 1$.

To see that there exists $k$ with
$\sum_{i \in I} \rho^k_i = 1$, suppose otherwise.
It cannot be the case that one ever arrives at $l$ with
$\sum_{i \in I} \rho^l_i > 1$, because the successive $\rho^l$ are produced by
adding a (non-strictly) decreasing sequence of weights of the form
$2^{-t}$, and if such a sequence sums to greater than $1$ then a finite
prefix of it sums to $1$. Thus the assumed non-existence of $k$
must be because one obtains infinite sequences $\rho^l$, $\mu^l$
for which $\sum_{i \in I} \rho^l_i < 1$ always holds.
Define $\rho^{\infty}_i = \lim_{l = 0}^\infty \rho^l_i$ and
$\mu^{\infty}_i = \lim_{l = 0}^\infty \mu^l_i$ (both are monotone
bounded sequences).
Clearly, $\sum_{i \in I} \rho^\infty_i \leq 1$.
Because each $\rho^{l+1}$ is defined by
taking the largest possible $2^{-t}$ from $\mu^l$, we have
that $\mu^\infty_i = 0$ for all $i$.
So $\sum_{i \in I} \frac{1}{2}\rho^\infty_i + \frac{1}{2}\mu^\infty_i
\leq \frac{1}{2}$.
However, as
$\sum_{i \in I} \frac{1}{2}\rho^l_i + \frac{1}{2}\mu^l_i = 1$,
it must hold that
$\sum_{i \in I} \frac{1}{2} \rho^\infty_i + \frac{1}{2}\mu^\infty_i = 1$,
a contradiction.

We have shown that there does indeed exist $k$ as claimed.
The $\rho$ and $\mu$ required by the lemma
are defined by $\rho = \rho^k$
and $\mu = \mu^k$. If the original $\lambda$ is a {\conab} weight
function then so is $\mu$.
\end{proof}

Let $(I, \lambda)$ be a superconvex weight function. We
define
an associated term, in the sense of~{\S}\ref{SubSecTerms},
$\: t_{(I, \lambda)} = (u_l)_l \in \algset_I$,
as follows.
First, using Lemma~\ref{LemFindNormal}, we construct a sequence
$(I, \lambda^l)_l$ of superconvex weight functions, and
a sequence $(I, \rho^l)_l$ of dyadic convex weight functions
satisfying, for all $i \in I$,
\begin{align}
\label{EqnLamZero}
\lambda^0_i \:  & = \: \lambda_i\, , \\
\label{EqnLamL}
\lambda^l_i \: & = \: \frac{1}{2}\rho^l_i + \frac{1}{2}\lambda^{l+1}_i \, .
\end{align}
Thus, for all $i \in I$,
\begin{equation}
\label{EqnInfExpandWeight}
\lambda_i \: = \: \sum_{l = 0}^\infty 2^{-(l+1)} \, \rho^l_i.
\end{equation}
Finally, define $u_l$ to be any element of $\algsetf_I$ satisfying
\begin{equation}
\label{EqnGluck}
\superind(u_l) = \rho^l.
\end{equation}
Such an element
exists by Proposition~\ref{PropDyadic}.
%, which says that every dyadic weight
% function can be represented using midpoints.
Observe that, by its definition,
$t_{(I, \lambda)}$ is  a normal form
in the sense of~{\S}\ref{SubSecTerms}.

\begin{lemma}
\label{LemIndCalc}
For any superconvex weight function $(I, \lambda)$,
we have  $\superind(t_{(I,\lambda)}) = \lambda$.
\end{lemma}
\begin{proof}
We have, for any $i \in I$,
\begin{align*}
\superind(t_{(I,\lambda)})(i) &
\: = \: \superind((u_l)_l)(i) \\
& \: = \: \sum_{l = 0}^\infty 2^{-(l+1)} \superind(u_l)(i) \\
& \: = \: \sum_{l = 0}^\infty 2^{-(l+1)} \rho^l_i
& & \text{by (\ref{EqnGluck})} \\
& \: = \: \lambda_i
& & \text{by (\ref{EqnInfExpandWeight}).}
\end{align*}
\end{proof}

The required superconvex combination operation, $\supersym$,
can now be defined by
\begin{equation}
\label{equation:define-superconvex}
\supercom{i}{I}{\lambda_i}{x_i} \: = \:
 \interp(t_{(I,\lambda)},x).
\end{equation}
\begin{lemma}
$\supersym$ is an iterative superconvex structure on $X$ satisfying
the identity $\frac{1}{2} \cdot x + \frac{1}{2} \cdot y = m(x,y)$.
Moreover, if $(X,m)$ is cancellative then so is
$\supersym$.
\end{lemma}

\begin{proof}
To show that $\supersym$ gives superconvex structure,
equation~(\ref{EqnSuperA}) holds easily so we verify
equation~(\ref{EqnSuperB}).
For the left-hand side we have:
\begin{align*}
\supercom{i}{I}{\lambda_i}
         {\left(\supercom{j}{J}{\mu_{ij}}{x_j}\right)}
& = \interp(t_{(I,\lambda)}, \, i \mapsto \interp(t_{(J, \mu_i)}, x)), \\
& = \interp([i \mapsto t_{(J, \mu_i)}]\,t_{(I,\lambda)}, x),
& \text{by Lemma~\ref{LemSemSub}}.
\end{align*}
For the right-hand side:
\begin{align*}
\supercom{j}{J}{\left(\sum_{i \in I} \lambda_i \mu_{ij}\right)}{x_j}
& = \interp\left(t_{\left(J, \, j \mapsto \sum_{i \in I} \lambda_i \mu_{ij}\right)}, x\right).
\end{align*}
By Proposition~\ref{PropPhew}, to show the two are equal, it
suffices to show the equality:
\[
\superind([i \mapsto t_{(J, \mu_i)}]\, t_{(I,\lambda)}) \: = \:
\superind\left(t_{\left(J, \, j \mapsto \sum_{i \in I} \lambda_i \mu_{ij}\right)}\right).
\]
For this, we have:
\begin{align*}
\superind([i \mapsto t_{(J, \mu_i)}]\, t_{(I,\lambda)})(j)
& = \sum_{i \in I} \, \superind(t_{(I,\lambda)})(i) \cdot
          \superind(t_{(J, \mu_i)})(j),
& \text{by Lemma~\ref{LemIndSub},} \\
& = \sum_{i \in I}  \lambda_i \mu_{ij},
& \text{by Lemma~\ref{LemIndCalc},} \\
& = \superind\left(t_{\left(J, \, j \mapsto \sum_{i \in I} \lambda_i \mu_{ij}\right)}\right)(j),
& \text{by Lemma~\ref{LemIndCalc}.}
\end{align*}

To show that
$\frac{1}{2} \cdot x_0 + \frac{1}{2} \cdot x_1 = m(x_0,x_1)$, we first observe that
$\frac{1}{2} \cdot x_0 + \frac{1}{2} \cdot x_1 =
\interp\left(t_{\left(\{0,1\}, \, i \mapsto \frac{1}{2}\right)}, x\right)$, by definition of
$\supersym$,
and that $m(x_0,x_1) = \interp((0,1), x)$,
for the term $(0,1) \in \algsetf_{\{0,1\}}$, by definition
of $\interp$. Thus, by Proposition~\ref{PropPhew},
it suffices to show that
$\superind\left(t_{\left(\{0,1\}, \, i \mapsto \frac{1}{2}\right)}\right) =
 \superind((0,1))$. But this is immediate from
Lemma~\ref{LemIndCalc}.

The claimed iteration and cancellation
properties of $\supersym$ now follow immediately from
Lemmas~\ref{LemSuperIt} and~\ref{LemSuperCan} respectively.
\end{proof}

Next we address the uniqueness part of
Theorem~\ref{ThmSuperconvex}(\ref{ThmSuperconvex:2}). In fact, we
simultaneously establish the uniqueness of the associated convex
structure on $X$.

\begin{lemma}
\label{LemUniqueConab}
If $X$ is an
an iterative midpoint set
then there is a unique {\conab} structure $\supersym$ on
$X$ satisfying
$\frac{1}{2} \cdot x_0 + \frac{1}{2} \cdot x_1 = m(x_0,x_1)$.
\end{lemma}

\begin{proof}
Existence has already been established. For uniqueness,
suppose that $\supersym$ and $\supersym'$ are
{\conab} structures satisfying
$\frac{1}{2} \cdot x_0 + \frac{1}{2} \cdot x_1 = m(x_0,x_1) =
\frac{1}{2} \cdot x_0 +' \frac{1}{2} \cdot x_1$.
Consider any two {\conab} combinations of the form
$\supercom{i}{I}{\lambda_i}{x_i}$
and
$\lsupercom{\prime}{i}{I}{\lambda_i}{x_i}$.
Use Lemma~\ref{LemFindNormal} to construct
a sequence $(I, \lambda^l)$ of {\conab}
weight functions and a sequence
$(I, \rho^l)$ of dyadic convex weight functions
satisfying equations~(\ref{EqnLamZero}) and~(\ref{EqnLamL}).
Then, using Lemma~\ref{LemKoenig} to obtain the first equality,
\begin{align*}
\supercom{i}{I}{\lambda^l_i}{x_i}
\:
&  = \:
\frac{1}{2} \,  \cdot \,
\supercom{i}{I}{\rho^l_i}{x_i} \: +' \:
\frac{1}{2} \, \cdot \,
\supercom{i}{I}{\lambda^{l+1}_i}{x_i} \\
& = \: m\left(\supercom{i}{I}{\rho^l_i}{x_i} \, , \,
         \supercom{i}{I}{\lambda^{l+1}_i}{x_i}\right).
\end{align*}
Thus, by iteration for $m$, and because $\lambda = \lambda^0$,
we have
\begin{equation}
\label{EqnReduceSuperconvex}
\supercom{i}{I}{\lambda_i}{x_i}
 \:
   = \:
M_l \,\left(\,  \supercom{i}{I}{\rho^l_i}{x_i} \, \right) \, .
\end{equation}
By the same argument for $\supersym'$,
\[
\dlsupercom{\prime}{i}{I}{\lambda_i}{x_i}
 \:
   = \:
M_l \,\left(\,  \dlsupercom{\prime}{i}{I}{\rho^l_i}{x_i} \, \right) \, .
\]
But, because
$\frac{1}{2} \cdot x_0 + \frac{1}{2} \cdot x_1 = m(x_0,x_1) =
\frac{1}{2} \cdot x_0 +' \frac{1}{2} \cdot x_1$, it follows from
Proposition~\ref{PropDyadic} that, for every $l \geq 0$,
\[ \supercom{i}{I}{\rho^l_i}{x_i}
\: = \:
\dlsupercom{\prime}{i}{I}{\rho^l_i}{x_i}\, .
\]
Thus indeed
$\supercom{i}{I}{\lambda_i}{x_i}
= \lsupercom{\prime}{i}{I}{\lambda_i}{x_i} $.
\end{proof}

It remains to prove
statement~(\ref{ThmSuperconvex:3}) of Theorem~\ref{ThmSuperconvex}.
\begin{prop} \label{prop-affine}
Let $X$ and $Y$ be iterative midpoint sets, considered together with
their unique associated (super)convex structures. Then,
for any function $f \colon X \to Y$,
the following are equivalent:
\begin{enumerate}
\item $f$ is a midpoint homomorphism,
\item $f$ is affine,
\item  $f$ is superaffine.
\end{enumerate}
\end{prop}
\begin{proof}
The implications $3 \Longrightarrow 2 \Longrightarrow 1$ are
trivial, so we prove $1 \Longrightarrow 3$. To this end,
consider any superconvex combination
$\lsupercom{X}{i}{I}{\lambda_i}{x_i}$ in $X$.
As in the proof of Lemma~\ref{LemUniqueConab},
equation~(\ref{EqnReduceSuperconvex}),
we expand this as
\[
\dlsupercom{X}{i}{I}{\lambda_i}{x_i}
 \:
   = \:
M^X_l \,\left(\dlsupercom{X}{i}{I}{\rho^l_i}{x_i} \right),
\]
where each $\lsupercom{X}{i}{I}{\rho^l_i}{x_i}$ is a dyadic convex
combination. But then,
\begin{align*}
f\left(\dlsupercom{X}{i}{I}{\lambda_i}{x_i} \right)
\: & = \: f\left(\, M^X_l \left( \dlsupercom{X}{i}{I}{\rho^l_i}{x_i} \right) \right)
\\
& = \: M^Y_l \, f\!\left( \dlsupercom{X}{i}{I}{\rho^l_i}{x_i}  \right)
& & \text{by Proposition~\ref{PropSetIt}(\ref{iter:1})} \\
& = \: M^Y_l  \left(  \dlsupercom{Y}{i}{I}{\rho^l_i}{f(x_i)}  \right)
& & \text{by Proposition~\ref{PropDyadic}(\ref{propdyadic:3})} \\
& = \: \dlsupercom{Y}{i}{I}{\lambda_i}{f(x_i)} \, ,
\end{align*}
where the last equality is again
equation~(\ref{EqnReduceSuperconvex}) from the proof of
Lemma~\ref{LemUniqueConab}.
\end{proof}

Since every cancellative superconvex set is iterative (Proposition~\ref{prop:css-iterative}),
a special case of Proposition~\ref{prop-affine} is that every midpoint homomorphism between
cancellative superconvex sets is superaffine. In contrast, it is not in general the case that
 midpoint homomorphisms between cancellative convex sets are automatically affine. For example,
 if $\mathbb{R}$ is considered as a vector space over $\mathbb{Q}$ then, it is a well-known consequence of the axiom of choice (using which one obtains a Hamel basis for $\mathbb{R}$) that there exist $\mathbb{Q}$-linear functions that are not continuous. So \emph{a fortiori} there exist
 midpoint homomorphisms that are not affine.

\subsection{Applications of Theorem~\ref{ThmSuperconvex}}
\label{SecApplySuperconvex}

\newcommand{\sreal}{\mathbb{R}}
\newcommand{\qreal}{(\mathbb{R}/{\!\sim})}
\newcommand{\dcl}[1]{\overline{#1}}

It is immediate from Theorem~\ref{ThmSuperconvex} that the forgetful functor from the category of iterative superconvex sets (with superaffine maps as morphisms) to the category of iterative midpoint objects (with midpoint homomorphisms as morphisms), defined by taking $m(x,y) := \frac{1}{2} \cdot x + \frac{1}{2} \cdot y$ as the midpoint structure, is an isomorphism of categories. Furthermore, this cuts down to an isomorphism of categories between the cancellative objects on either side.
We now show that, for every set $X$, there exists a free iterative superconvex set $FX$ generated by $X$, which, moreover, is cancellative. Furthermore, we give an explicit description of $FX$. By the isomorphism of categories noted above, $FX$
is also the free iterative midpoint set over a set $X$.

\begin{theorem}
The free iterative superconvex set over a set $X$ is the set
\[
FX \: = \:
\{ w \colon X \to [0,1] \mid
\text{$(X, w)$ is a superconvex weight function} \} \, ,
\]
with the superconvex combination operation defined pointwise
\begin{equation}
\label{EqnDefFreeSuper}
\supercom{i}{I}{\lambda_i}{w_i} \: = \:
\left(x\, \mapsto \, \sum_{i \in I} \, \lambda_i\cdot {w_i(x)}\right),
\end{equation}
and
%the midpoint operation, $x \oplus y = \frac{1}{2}x + \frac{1}{2} y$,
%defined pointwise and
the insertion of generators
$\eta \colon X \to FX$ given by $\eta(x)\,  =  \, y \mapsto \delta_{xy}\,$.
Moreover, $FX$ is cancellative.
\end{theorem}
\begin{proof}
First, it is easily verified that $FX$ is indeed a cancellative iterative superconvex set.
To prove that $FX$ is the free iterative superconvex set over $X$,
we verify, more generally, that it is the free
superconvex set.

Let $Y$ be an arbitrary superconvex set, and let
$f \colon X \to Y$ be any function. Define $h \colon FX \to Y$
by
\begin{equation}
\label{EqnDefh}
h(w) \: = \: \dlsupercom{Y}{x}{X}{w(x)}{f(x)} \, .
\end{equation}
Easily, by~(\ref{EqnSuperB}),  $f = h \circ \eta$. To verify that $h$ is
superaffine, calculate
\begin{align*}
h\left(\supercom{i}{I}{\lambda_i}{w_i}\right)
& = \dlsupercom{Y}{x}{X}{\left(\supercom{i}{I}{\lambda_i}{w_i}\right)(x)}{f(x)}
&& \text{by~(\ref{EqnDefh})}
\displaybreak[0] \\
& = \dlsupercom{Y}{x}{X}{\left(\sum_{i \in I}\, \lambda_i \times w_i(x)\right)}{f(x)}
& & \text{by~(\ref{EqnDefFreeSuper})}
\displaybreak[0] \\
& = \dlsupercom{Y}{i}{I}{\lambda_i}
         {\left(\dlsupercom{Y}{x}{X}{w_i(x)}{f(x)}\right)}
& & \text{by~(\ref{EqnSuperB})}
\displaybreak[0] \\
& = \dlsupercom{Y}{i}{I}{\lambda_i}{h(w_i)}
&& \text{by~(\ref{EqnDefh}).}
\end{align*}
For uniqueness, suppose we have a superaffine $g \colon FX \to Y$
satisfying $f = g \circ \eta$. Then
\begin{align*}
g(w) & = g\left(x \, \mapsto \, \sum_{y \in X} w(y) \cdot \delta_{yx}\right) \\
& = g\left(x \, \mapsto \, \sum_{y \in X} w(y) \cdot \eta(y)(x)\right)
& & \text{definition of $\eta$}
\displaybreak[0]  \\
& = g\left(\supercom{y}{X}{w(y)}{\eta(y)}\right)
   & & \text{by~(\ref{EqnDefFreeSuper})}
\displaybreak[0] \\
& = \dlsupercom{Y}{y}{X}{w(y)}{g(\eta(y))}
 & &  \text{$g$ superaffine}
\displaybreak[0] \\
& = \dlsupercom{Y}{y}{X}{w(y)}{f(y)}
 & &  \text{$f = g \circ \eta$}
\displaybreak[0] \\
& = h(w)
 & & \text{by~(\ref{EqnDefh}).}
\end{align*}
\end{proof}
\noindent
The two corollaries below, which, because they are central to the paper, we phrase for midpoint-convex bodies rather than superconvex sets, follow as  immediate consequences.

\begin{corollary} \label{n-simplex}
  The free iterative midpoint-convex body  over the $n+1$ element set,
$\{0, \dots, n\}$, is the
$n$-dimensional simplex.
\end{corollary}

\begin{corollary}
\label{ThmSet}
For $a < b \in \mathbb{R}$, the interval
$[a,b]$, with $\oplus$ as midpoint, is an interval object in
the category of sets.
\end{corollary}

As another application of Theorem~\ref{ThmSuperconvex}, we can, by
Proposition~\ref{prop:css-iterative}, directly import the known
examples of cancellative superconvex sets as examples of cancellative
iterative midpoint sets.  Topological vector spaces provide a rich
source of such examples. The following is taken
from~\cite[Example~1.6]{koenig:superconvex}.  It will be applied in
Section~\ref{spaces} to show that the unit interval with its Euclidean
topology is an interval object in the category of topological spaces.
\begin{prop} {(K\"onig~\cite[Example~1.6]{koenig:superconvex})}
\label{ExKoenig}
Let $E$ be a Hausdorff real topological vector space.
A nonempty subset $X \subseteq E$ is said to be
\emph{$\sigma$-convex} if, for each
superconvex combination $(\SetNat, \lambda, x)$, the
series $\sum_{l = 0}^\infty\, \lambda_l x_l$ converges in the topology
of $E$ to some element $\supercom{i}{\SetNat}{\lambda_i}{x_i} \, \in \, X$.
\begin{enumerate}
\item \label{ExKoenig:1} If $X$ is $\sigma$-convex then it is a cancellative superconvex set.
\item \label{ExKoenig:2} If $X$ is $\sigma$-convex then it is convex and bounded.
\item \label{ExKoenig:3} If $\mathrm{dim} \,E$ is finite and $X$ is convex and bounded
      then it is $\sigma$-convex.
\item If $X$ is convex, bounded and sequentially complete
      then it is $\sigma$-convex.
\item If $X$ is convex, bounded and open, and if
      $\overline{X}$ is sequentially complete
      then $X$ is $\sigma$-convex.
\end{enumerate}
\end{prop}

\noindent
The above list gives us a bountiful supply of $\sigma$-convex sets, hence of
cancellative superconvex sets, hence, by Theorem~\ref{ThmSuperconvex} and Proposition~\ref{prop:css-iterative}, of midpoint-convex bodies in the category of sets.

%\begin{prop}
%A subset $X \subseteq \sreal^n$ is an iterative midpoint set under $x \oplus y := \frac{1}{2}\cdot x + \frac{1}{2} \cdot y$ if and only if $X$ is bounded and convex.
%\end{prop}
%\begin{proof}
%By Theorem~\ref{ThmSuperconvex}, $X$ is an iterative midpoint set under $\oplus$ if and only if
%it there is an iterative superconvex structure extending $\oplus$, and such structure is necessarily
%as described in Example~\ref{ExKoenig}.\todo{THIS IS TRUE BUT I DON'T HAVE A NIUCE CLEAN ARGUMENT FOR IT.} By points~(\ref{ExKoenig:1})--(\ref{ExKoenig:3}) of Example~\ref{ExKoenig}, $X$ carries such
%superconvex structure  if and only if it is convex and bounded. (Note that the convex structure is always cancellative,
%and  hence the superconvex structure is iterative by Proposition~\ref{prop:css-iterative}.)
%\end{proof}

As a final application of Theorem~\ref{ThmSuperconvex}, we
strengthen~\cite[Theorem~3.1]{koenig:superconvex} to
apply to midpoint sets rather than convex sets.
\begin{prop}
\label{prop:one-extension}
If $X$ is a cancellative midpoint set then there is at
most one superconvex structure $\supersym$ on
$X$ satisfying
$\frac{1}{2} \cdot x_0 + \frac{1}{2} \cdot x_1 = m(x_0,x_1)$.
\end{prop}
\begin{proof}
If $X$ has a superconvex structure satisfying
$\frac{1}{2} \cdot x_0 +\frac{1}{2} \cdot x_1 = m(x_0,x_1)$ then,
by Lemma~\ref{LemSuperCan}, this is cancellative hence iterative.
By Theorem~\ref{ThmSuperconvex}(\ref{ThmSuperconvex:1}), $X$ is an
iterative midpoint set. Thus the uniqueness of the superconvex
structure follows from Theorem~\ref{ThmSuperconvex}(\ref{ThmSuperconvex:2}).
\end{proof}
Although the assumption of cancellativity is used crucially
in the proof,
we do not know
of any non-cancellative midpoint set that admits more than
one superconvex structure extending the midpoint operation.

We remark that the analogue of Proposition~\ref{prop:one-extension} fails for convex structure.
That is, a cancellative midpoint
structure may have more than one convex extension.
Note that, in view of Lemma~\ref{LemUniqueConab},
this can only happen when no superconvex extension of the
midpoint structure exists. One example of such a phenomenon can be obtained by using a Hamel base for $\mathbb{R}$ as a vector space over $\mathbb{Q}$ to construct a nonstandard
convex structure extending the standard midpoint structure on $\mathbb{R}$.
%\begin{ex}
%\label{ExNonStand}
%Continuing from Example~\ref{ExFreyd}, we exhibit
%non-standard convex structures on $\sreal$
%extending the standard cancellative midpoint operation.
%Write $\supersym$ for the standard convex structure on
%$\sreal$. Let $f \colon \sreal \to \sreal$ be a bijective
%function.
%Define a convex combination
%operation,  $\supersym^f$, by:
%\[
%\dlsupercom{f}{i}{I}{\lambda_i}{x_i} \: = \:
%f\left(\, \supercom{i}{I}{\lambda_i}{(f^{-1}(x_i))} \, \right)\, .
%\]
%It is easily verified that: the
%operation $\supersym^f$ gives cancellative convex structure on $\sreal$;
%the convex structure $\supersym^f$ is standard if and only if
%$f$ is affine (with respect to the standard convex structure);
%and the convex structure $\supersym^f$
%extends the standard midpoint operation if and only if
%$f$ is a midpoint homomorphism (with respect to the standard
%midpoint structure).
%Thus, by the last paragraph of Example~\ref{ExFreyd},
%we obtain $2^{2^{\aleph_0}}$ nonstandard
%affine structures on $\sreal$ that extend the
%standard midpoint structure.
%\end{ex}

%%% Local Variables:
%%% mode: pdflatex
%%% TeX-master: "main.tex"
%%% End:

%% file: top.tex
\section{Interval objects in topological spaces}
\label{SecIntInTop} \label{spaces}

In this section we
%return to the claims made earlier in
%Examples \ref{ExTopConv} and \ref{ExTopInt},
investigate midpoint-convex bodies and interval objects in the category $\Top$
of topological spaces.  In  $\Top$, the discrete natural numbers  $\mathbb{N}$ are an exponentiable nno, with
the exponential $A^\mathbb{N}$ of a topological space $A$ given by the countable power $A^\omega$ with the product topology.
Proposition~\ref{PropSetIt} thus applies to $\Top$, with the map  $M:A^{\omega} \to A$ being continuous
with respect to the product topology.

\begin{theorem}
Let $A \subseteq \SetReal^n$ be closed under $\oplus$.
Then  $(A, \oplus)$, endowed with the relative Euclidean topology, is a
midpoint-convex body in $\Top$ if and only if $A$ is a bounded convex subset of $\SetReal^n$.
\end{theorem}
\begin{proof} Suppose $A$ is a bounded convex subset. By properties~\ref{ExKoenig:1} and~\ref{ExKoenig:3} of Proposition~\ref{ExKoenig},
$A$ is a cancellative superconvex set. By Proposition~\ref{prop:css-iterative} and Theorem~\ref{ThmSuperconvex},
$(A, \oplus)$ is a midpoint-convex body in the category of sets. Since $\oplus : A \times A \to A$ is trivially continuous,
by Proposition~\ref{PropSetIt}, $(A, \oplus)$ is a midpoint-convex body in $\Top$ if the function
$\bigoplus: A^\omega \to A$ defined below is continuous.
\begin{align}
\label{equation:bigoplus}
\bigoplus (x_0,x_1,x_2, \dots) \; & = \;
\sum_{i \geq 0} 2^{-(i+1)}x_i \, .
\end{align}
%If  $\bigoplus: A^\omega \to A$ is continuous w.r.t.\ the
%product topology on $A^\omega$ then conditions~\ref{ExKoenig:1} and~\ref{ExKoenig:3} of Example~\ref{ExKoenig} hold.
%Hence $(A, \oplus)$ is an m-convex body if
%$\bigoplus$ is continuous.
%Moreover, if $(A, \oplus)$ is an m-convex body in $\Top$ then
%the induced infinitary
%operation of Proposition~\ref{PropNEquiv}.1 yields that
%$\bigoplus: A^\omega \to A$ is continuous.
%Hence $(A, \oplus)$ is an m-convex body if and only if
%$\bigoplus$ is continuous.
To show continuity, as $A$ is bounded, let $r>0$ be such that, for all $\Vec{x} \in A$,
$|\Vec{x}| < r$.  Consider any open
$\epsilon$-ball, $B_\epsilon(\Vec{z}) \cap A$, centred at
$\Vec{z} = \bigoplus_i(\Vec{x}_i)$ where $(\Vec{x}_i) \in A^\omega$.
Take any $n$ such that $2^{-n} < \epsilon/r$.  Then, for all
$(\Vec{y}_i)$ in the neighbourhood
$\{(\Vec{y}_i) \in A^\omega \, \mid\,
\mbox{$\Vec{y}_j \in B_{\epsilon/2}(\Vec{x}_j) \cap A$, for all
  $0 \leq j \leq n$}\}$ of $(\Vec{x}_i)$ in $A^\omega$, we have
$\bigoplus_i \Vec{y_i} \in B_\epsilon(\Vec{z}) \cap A$.  Thus
$\bigoplus$ is indeed continuous.

Conversely, suppose $(A, \oplus)$ is a midpoint-convex body in $\Top$. By Proposition~\ref{PropSetIt}, there is a unique
continuous function $M: A^\omega \to A$ satisfying the unfolding property. Moreover, it is easy to show that continuity together with the unfolding property imply $M = \bigoplus$. Equation~\eqref{equation:define-superconvex}, in the proof of
Theorem~\ref{ThmSuperconvex}, which defines the superconvex set structure on $A$ in terms of the supermidpoint structure $\bigoplus$, expresses each superconvex combination  using a normal form consisting of one outside application of the infinitary $\bigoplus$ operation with finitely many nested applications of the binary $\oplus$ operation in each argument
position. It follows that, for any superconvex weight function $(I, \lambda)$, the associated superconvex combination function $A^I \to A$ is continuous. This means that the $\sigma$-convexity condition of Proposition~\ref{ExKoenig} is satisfied. Therefore, by
property~\ref{ExKoenig:2} of Proposition~\ref{ExKoenig}, $A$ is convex and bounded.
%
% ATTEMPT TO SHOW EUCLIDEAN TOPOLOGY IS FINEST:
%
%Now suppose that
%$\bigoplus: A^\omega \to A$ is continuous with
%respect to another topology, $\Omega$, on $A$.
%Let $U$ be any open in $\Omega$. We must show that $U$ is
%Euclidean open. Consequently, take any $\Vec{x} \in U$.
%Then $\Vec{x} = \bigoplus(\Vec{x}, \Vec{x}, \Vec{x}, \dots)$.
%As $\bigoplus$ is continuous with respect to the product topology on
%$A^\omega$, there exist $\Omega$-open neighbourhoods
%$U_0, \dots, U_{n-1}$ of $\Vec{x}$ such that,
%for all $(\Vec{y}_i) \in
%\{(\Vec{y}_i) \in A^\omega \, \mid\, \mbox{$y_j \in U_j$ for all
%$0 \leq j \leq n$}\}$, we have $\bigoplus_i (\Vec{y}_i) \in U$.
%In particular, $\bigoplus_i (\Vec{y}_i) \in U$
%for all $(\Vec{y}_i) \in
%\{(\Vec{y}_i) \in A^\omega \, \mid\, \mbox{$\Vec{y}_j = \Vec{x}$ for all
%$0 \leq j \leq n$}\}$.
%But then, for all
%
%NOW WANT IMAGE UNDER OF SET UNDER $\bigoplus$ CONTAINS SOME
%OPEN BALL, BUT THIS NOT TRUE IF $N \geq 2$ AND $\Vec{x}$ is
%ON BOUNDARY OF $A$!
\end{proof}
In certain cases, e.g. when $n = 1$ or when $A$ is a
Euclidean-open subset of $\SetReal^n$, one can show that
the Euclidean topology on $A$ is the finest such that
$(A, \oplus)$ is an m-convex body.
The following example shows that
this is not the case for an arbitrary $A$.
\begin{ex}
{\em
Let $A$ (resp. $\overline{A}$) be the open (resp. closed)
unit ball (the circle) in $\mathbb{R}^2$. Let $\CU$ be the
following topology: $U \subseteq \overline{A}$ is open if and only
if $U \cap A$ is Euclidean open and every point of $U$ is
a Euclidean accumulation point of $U \cap A$. Then
$(A, \oplus)$ is an m-convex body under $\CU$.
}
\end{ex}

Certain other basic information about m-convex bodies in $\Top$ can be
inferred using Proposition~\ref{PropAdj}. The forgetful functor
$U: \Top \to \Set$ has both a left adjoint $\Delta$ (giving the
discrete toplogy) and a right adjoint $\nabla$ (giving the indiscrete
topology). Thus, both $U$ and $\nabla$ preserve m-convex bodies. As
$U$ does, we see that, by Proposition~\ref{ExKoenig}, under any topology
whatsoever, for a standard midpoint subalgebra $A$ of $\SetReal^n$ to
be an m-convex body in $\Top$, it must be a bounded convex set. Also,
for any bounded convex set, $(A, \oplus)$ with the indiscrete topology
is an m-convex body in $\Top$. (N.b.\ it is {\em not}, in general, an
m-convex body when given the discrete topology, as then $\bigoplus$ is
not continuous.) Recall that $\SetIval := [-1,1]$.

% Also, by Example~\ref{ExKoenig}, if an interval object exists in
% $\Top$ then $U$ preserves it.  In fact, we have already claimed in
% Example~\ref{ExTopInt} that $(\SetIval, \oplus, -1,1)$ is an
% interval object in $\Top$ when given the Euclidean topology.  In
% fact, as $\Top$ is not cartesian closed, it is appropriate to show
% that this is a parameterized interval object.

\begin{theorem}
\label{ThmTop}
$(\SetIval, \oplus, -1, 1)$ with the Euclidean topology
is a parametrised interval object in $\Top$.
\end{theorem}
\begin{proof}
Let $(A,m)$ be any m-convex body in $\Top$ and $f,g:X \to A$ be continuous
functions from a space $X$. We must show that there is a unique
right-homomorphism $h : X \times \SetIval \to A$ in $\Top$.
By Corollary~\ref{ThmSet} and Proposition~\ref{PropParam}, there is
a unique such right-homomorphism $h$ in $\Set$. It suffices to
show that this function $h$ is continuous.

For each $z \in X$, write $d_z: \Confine(\SetRatD) \to A$,
for the unique bipointed homomorphism from $\Confine(\SetRatD)$ to
$A$, as in the proof of Corollary~\ref{ThmSet}. Then, by the
proof of Corollary~\ref{ThmSet}, we have
\begin{eqnarray*}
h\left(z, \bigoplus_i(q_i)\right) & = & M_i(d_z(q_i)).
\end{eqnarray*}
To see that this is continuous, consider
the function $h': X \times \{-1,1\}^\omega \to A^\omega$ defined
by
\begin{eqnarray*}
h'(z,(q_i)_i) & = & (d_z(q_i))_i.
\end{eqnarray*}
It is easily verified that this is continuous with respect to
the product topologies on its domain and codomain.
Also, consider the function $\bigoplus : \{-1,1\}^\omega \to \SetIval$,
the restriction of $\bigoplus : \SetIval^\omega \to \SetIval$.
This is a topological quotient map, as it is a surjective continuous
function between compact Hausdorff spaces. By Lemma~\ref{stability} below, the function
$\Id \times \bigoplus : X \times \{-1,1\}^\omega \to X \times \SetIval$,
is also a quotient map.

Now, we have the following commuting diagram in $\Set$ (where $M$
is the infinitary operation associated with $(A,m)$).
\[
\begin{tikzcd}
X \times \{-1,1\}^\omega \arrow[r, "{h'}"] \arrow[d, "{\Id \times \bigoplus}"'] & A^\omega \arrow[d, "{M}"] \\
X \times \SetIval \arrow[r, "{h}"] & A.
\end{tikzcd}
\]
We already know that all maps other than $h$ are continuous.  But then
$h$ is continuous too because $\Id \times \bigoplus$ is a quotient
map.
\end{proof}

It follows by Proposition~\ref{PropAdj}(\ref{PropAdj:1}) that
$(\SetIval, \oplus, -1, 1)$ with the Euclidean topology is a
parametrised interval object in any full reflective subcategory
of $\Top$ that contains the closed Euclidean interval.
Thus, for example, it is a parametrised interval object
in the category of compact Hausdorff spaces.

To conclude the proof of Theorem~\ref{ThmTop}, recall that a
topological quotient map $q: A \to B$ is called {\em stable} if, for
every space $B'$ and continuous function $f: B \to B'$, the pullback
$q':A' \to B'$ of $q$ along $f$ in $\Top$ is a topological quotient
map. In particular, stable quotient maps are preserved under products.
\begin{prop}[Day\&Kelly~\cite{day:kelly}]
\label{PropDayKelly}
A topological quotient map $q: A \to B$ is stable if and only if, for
every $y \in B$, and an open covering $\{U_i\}_{i \in I}$ of
$q^{-1}(y)$, there is a finite set $I' \subseteq I$ such that
$\bigcup_{i' \in I'} q(U_{i'})$ is a neighbourhood of $y$.
\end{prop}
\begin{lemma} \label{stability} The quotient map
  $\bigoplus : \{-1,1\}^\omega \to \SetIval$ is stable.
\end{lemma}
\begin{proof}
  We use Proposition~\ref{PropDayKelly}.  Take any $y \in
  \SetIval$. If $y$ is non-dyadic then there exists a unique sequence
  $(d_i) \in \{-1,1\}^\omega$ with $\bigoplus_i d_i = y$.  Take any
  basic neighbourhood
  $U_n = \{ (e_i) \in \{-1,1\}^\omega \mid \mbox{for all $i<n$,
    $e_i = d_i$}\}$ of $(d_i)$.  Then $q(U_n)$ is a neighbourhood of
  $y$ as each of $-1$ and $1$ occurs infinitely often in the sequence.
  If $y$ is dyadic, then there exist exactly two sequences in
  $q^{-1}(y)$. These have the form
  $x = d_0 \dots d_n (-1) 1 1 1 \dots$ and
  $x' = d_0 \dots d_n 1 (-1) (-1) (-1) \dots$.  It is easily seen
  that, for any basic neighbourhood $U$ of $x$, it holds that $q(U)$
  restricts to a neighbourhood of $y$ in the relative topology on
  $[-1,y]$. Similarly, for any basic neighbourhood $U'$ of $x'$, we
  have that $q(U')$ restricts to a neighbourhood of $y$ in the
  relative topology on $[y,1]$. Thus $q(U) \cup q(U')$ is a
  neighbourhood of $y$, as required.
\end{proof}

Similarly, it is easy to show that the $n$-simplex endowed with the
Euclidean topology is freely generated by the $(n+1)$-point discrete
space, adapting Corollary~\ref{n-simplex} from the category of sets to
the category of topological spaces.

%%% Local Variables:
%%% mode: latex
%%% TeX-master: "main.tex"
%%% End:

%% file: topos.tex
 \newcommand{\ACNN}{\mathrm{AC_{\Nat\Nat}}}
\newcommand{\ACN}{\mathrm{(XXXXX)}}
\newcommand{\ACNT}{\mathrm{(XXXXX)}}

\section{Interval objects in an elementary topos}
\label{SecIntInTopos} \label{topos}

Let $\CE$ be an elementary topos with natural numbers
object $\Nat$. Our references for topos theory
are~\cite{johnstone:topostheory,MACMOER}.
Our notation is standard, in particular we write $\Power(X)$ for
the powerobject of $X$.
We shall make substantial use of the (intuitionistic) internal logic of $\CE$,
the Mitchell-B{\'e}nabou language, using it informally, but
taking care to distinuish between
ordinary (external) mathematical statements and
statements that are to
be interpreted internally in $\CE$. We shall use the following (standard)
definition.
\begin{defn}[Separated object]
\label{DefSep}
An
object $X$ of $\CE$ is said to be \emph{(double-negation) separated}
if, internally in $\CE$,
\[
\text{for all $x,x' \in X$, $\lnot\lnot(x = x')$ implies $x = x'$.}
\]
\end{defn}
It is easily seen and standard that any subobject of a separated object is
itself separated.

It is a well-known fact that different constructions
of the real numbers, which are equivalent to each other using
classical logic, give different notions of real number
when interpreted within the intuitionistic internal logic
of a topos. Two notions are
particularly prevalent, the
{\em Dedekind} reals and the {\em Cauchy} (or {\em Cantor}) reals.
Both are defined using the object of rationals $\Rat$ and its
associated ordering. We give brief reviews of the definitions.
For more details see~\cite{johnstone:topostheory}.

One way of defining the object of
Dedekind reals in $\CE$ is as the subobject
$\Real_D \subseteq \mathcal{P}(\Rat) \times \mathcal{P}(\Rat)$ of
pairs $(L, U)$ of sets
of rationals that satisfy: $L$ and $U$ are disjoint
inhabited sets; $L$ is a down-closed set, each element of which
has a strictly greater element in $L$; $U$
is an upper-closed set, each element of which
has a strictly lower element in $U$; and $(L,U)$ satisfies
the {\em locatedness} property that  $x<y$ in $\Rat$ implies
either $x \in L$ or $y \in U$. We write $\Real_D$ for the object
of Dedekind reals, and use standard notation for the
usual arithmetic operations on it.
We also identify
$\Rat$ explicitly as a subobject of $\Real_D$ (via the
embedding $q \mapsto (\{r \mid r < q\} , \{s \mid s > q\})$). It is not hard to
show that $\Real_D$ is a separated object in $\CE$
\cite[Lemma 6.63.iii]{johnstone:topostheory}.

One direct way of defining the Cauchy reals is as a quotient
of the object of all Cauchy sequences of rationals (where the
notion of Cauchy sequence must be phrased in an appropriate
constructive way, see below).
For our purposes, it is more convenient
to adopt an alternative equivalent definition, identifying
the Cauchy reals as particular Dedekind reals. First, we
recall the relevant intuitionistic definition of a Cauchy sequence
of Dedekind reals, which requires an explicit modulus of convergence.
A sequence $\alpha_{(-)} \in {\Real_D}^\Nat$ of Dedekind reals
is {\em Cauchy} if
\[
\exists m\in \Nat^{\Rat^+}.\;
\forall \epsilon \in \Rat^+.\;
\forall i,j \geq m(\epsilon),\;
|\alpha_i - \alpha_j| \leq \epsilon \,,
\]
where $\Rat^+$ is the object of strictly positive rationals.
We say that $x \in \Real_D$ is the {\em limit} of the Cauchy sequence $\alpha$
if
\[
\exists m\in \Nat^{\Rat^+}.\;
\forall \epsilon \in \Rat^+ .\;
\forall i \geq m(\epsilon),\;
|\alpha_i - x| \leq \epsilon.
\]
Here it is not strictly necessary to require the
modulus function to exist, as the assumed modulus of $\alpha$
can be used instead.
The Cauchy reals are defined explicitly by
\begin{eqnarray*}
\Real_C & = & \{x \in \Real_D \mid
\mbox{$\exists \alpha \in \Rat^\Nat$ s.t.
$\alpha$ is a Cauchy sequence and $x = \mathit{\:\lim\:} \alpha$}\},
\end{eqnarray*}
where we are viewing $\Rat$ as a subobject of $\Real_D$ as above.
By the remark after Definition~\ref{DefSep}, $\Real_C$ is
a separated object.

%The reason for defining Cauchy sequences of Dedekind reals
%(rather than rationals) as basic, and for considering rationals and
%Cauchy reals as special Dedekind reals, is that it is now
% easy to consider the question of whether the various
% objects are Cauchy complete.
We say that a subobject $X \subseteq \Real_D$
is {\em Cauchy complete} if every Cauchy sequence in $X^\Nat$ has
a (necessarily unique) limit in $X$. It is not hard to show that
$\Real_D$ is Cauchy complete. Obviously $\Rat$ is not Cauchy complete.
$\Real_C$ partially rectifies the non-completeness of $\Rat$ by
adding all limits of Cauchy sequences of rationals.
If $\CE$ satisfies the choice principle,
\begin{align*}
& (\forall n \in \Nat.\,  \exists m \in \Nat.\:
\phi(n,m)) \text{\:implies\:}
 (\exists f \in \Nat^\Nat.\, \forall n \in \Nat. \:
\phi(n,f(n)))
& & (\ACNN)
%\\
%& (\forall n \in \Nat.\,  \exists b \in \mathbf{2}.\:
%\phi(n,b)) \mathit{\:implies\:}
% (\exists f \in \Nat^\mathbf{2}.\, \forall n \in \Nat. \:
%\phi(n,f(n)))
%& & \ACNT
\end{align*}
%\newcommand{\ACN}{\mathrm{(AC_{\Nat\text{\rm{-bounded}}})}}
%\begin{align*}
%(\forall n \in \Nat.\, & \exists m \leq n.\:
%\phi(n,m)) \mathit{\:implies\:} \\
%& (\exists f \in \Nat^\Nat.\, \forall n \in \Nat. \:
%f(n) \leq n \mathit{\:and\:} \phi(n,f(n))),
%& & \ACN
%\end{align*}
which holds in any boolean topos, then it follows that
$\Real_C$ is itself Cauchy complete. However, in general,
$\Real_C$ is {\em not} Cauchy complete, see \cite{Lubarsky2007}.

The possible failure of Cauchy completeness for $\Real_C$ makes it
natural to introduce another
object of reals, namely the Cauchy completion of $\Rat$
within $\Real_D$,
\begin{eqnarray*}
\Real_E & = & \bigcap \{ X \subseteq \Real_D \mid
\mbox{$\Rat \subseteq X$ and $X$ is Cauchy complete}\}\, .
\end{eqnarray*}
For reasons explained below,
we shall refer to $\Real_E$ as the object of {\em Euclidean} reals.

% \todo[inline]{We may wish to change the terminology ``Euclidean'' to something else.}
% We agreed that we like "Euclidean" after all.

By definition, $\Real_E$ is Cauchy complete. Also, by the
remark after Definition~\ref{DefSep},
$\Real_E$ is separated.
In contrast to the Cauchy reals, we do not have a direct method
of constructing the Euclidean reals without first going via
the Dedekind reals.

So far, we have identified three objects of reals
\[
\begin{array}{ccccc}
\Real_C & \subseteq & \Real_E & \subseteq & \Real_D \, .
\end{array}
\]
In the case that $\CE$ satisfies $\ACNN$, both
inclusions are equalities. The Grothendieck
topos $\mathit{Sh}(\mathbb{R})$
(sheaves over Euclidean space $\mathbb{R}$) is a simple example
in which the second inclusion is strict.
%To our embarrassment, we do not know an example topos
%in which the first inclusion is strict.
%(Thus we do not know if the envisaged
%failure of the Cauchy completeness of
%$\Real_C$ is actually possible --- although we are sure that it must be.)
An example showing that the first inclusion need not be an isomorphism
has been given by Lubarsky \cite{Lubarsky2007}.

Each kind of real number object determines a corresponding
interval; for example:
\[
\begin{array}{ccccc}
\Ival_D & = & \{x \in \Real_D \mid -1 \leq x \leq 1 \} & & \\
\Ival_E & = & \{x \in \Real_E \mid -1 \leq x \leq 1 \} & = &
 \Real_E \, \cap \, \Ival_D
\end{array}\]
Again, by the remark after Definition~\ref{DefSep},
each of these intervals is a  separated object of $\CE$.
Our motivation for introducing the Euclidean reals
is that, as we shall see, it
is the Euclidean interval $\Ival_E$ that gives
an interval object in $\CE$.
Thus closed Euclidean intervals in $\CE$  enjoy the same universal property as closed intervals with the
Euclidean topology in $\Top$. Moreover, this universal property is motivated by constructions from Euclidean geometry,
as discussed in Section~\ref{SecIntro}. It is for these reasons that we refer to
$\Real_E$ as the Euclidean reals in  $\CE$.

For $\Ival_E$ to be an interval object it must
\emph{a fortiori} carry a midpoint algebra structure.
That it does so is not immediately obvious.
Indeed, as part of the proof that
$\Ival_E$ is an interval object, we shall prove
that $\Ival_E$ is  closed under several important operations
on $\Ival_D$, including  $\oplus$.
We next define the closure properties to be established.

We say that a subobject $X \subseteq \Ival_D$ is an
{\em (endpointed) subalgebra} of $\Ival_D$ if $-1,1 \in X$ and $X$
is closed under $\oplus: \Ival_D \times \Ival_D \to \Ival_D$.
We say that a subalgebra
$X \subseteq \Ival_D$ is
{\em symmetric} if $X$ is closed under
$-: \Ival_D \to \Ival_D$.

Next we observe that there is a
unique morphism $\Confine : \Real_D \to \Ival_D$ in $\CE$ satisfying,
internally in $\CE$,
\begin{equation}
\label{EqnConfine}
\Confine(x) \: =  \:
\left \{ \begin{array}{ll}
          1 & \mbox{if $1 \leq x$} \\
          x & \mbox{if $-1 \leq x \leq 1$} \\
          -1 & \mbox{if $x \leq -1$}.
\end{array} \right.
\end{equation}
There is some subtlety here.
The above does not, in itself, constitute a definition
of a morphism as it is not necessarily true
in $\CE$ that
\[
\text{for all $x \in \Real_D$, $1 \leq x$ or
$-1 \leq x \leq 1$ or $-1 \leq x$.}
\]
Instead,~(\ref{EqnConfine}) should be read as a constraint
on $\Confine$. A morphism satisfying
the constraint is defined explicitly by
$\Confine(x)  =  \Min(1, \Max(x, -1))$.
The uniqueness of the morphism
is a routine consequence
of $\Real_D$ being a separated object.
The following is easy to show.
\begin{lemma}
\label{LemDouble}
For a symmetric subalgebra $X \subseteq \Ival_D$,
the following are equivalent.
\begin{enumerate}
\item For all $x \in X$, $\Confine(2x) \in X$.

\item For all $x,y \in X$, $\Confine(x+y) \in X$.

\item For all $x,y \in X$, $\Confine(x-y) \in X$.
\end{enumerate}
\end{lemma}
We say that a symmetric subalgebra
$X$ of $\Ival_D$ is {\em magnifiable}
if any of the  equivalent
conditions of Lemma~\ref{LemDouble} hold.

\begin{theorem}
\label{ThmTopos}
\leavevmode
\begin{enumerate}
\item $\Ival_E$ is a magnifiable symmetric subalgebra of $\Ival_D$.
\item \label{ThmTopos:ii} $(\Ival_E, \oplus, -1, 1)$ is an interval object in $\CE$.
\end{enumerate}
\end{theorem}
%
%
%
%In the proof
%we shall use  the fact that
%Propositions~\ref{PropSetIt},~\ref{PropSetEqns}
%and~\ref{PropSetApprox} all hold internally in $\CE$.
%We shall also reuse
%Lemmas~\ref{LemTAA}, \ref{LemTAZ} and~\ref{LemTAB}
%from the proof of Theorem~\ref{ThmSet}.
%Indeed, the proofs of all the above results go through, as written,
%in the internal logic of $\CE$.
%It is worth emphasising that
%the only result used in the proof of Theorem~\ref{ThmSet} that
%does not carry over to the internal logic of $\CE$ is
%Lemma~\ref{LemTAC}. The failure of this
%(It is only necessary to be slightly
%careful with the ellipses, and note that the
%$z_n$ and $w_n$ in the statement of Proposition~\ref{PropSetApprox}
%must be given
%as internal sequences.)

Crucial to the proof of Theorem~\ref{ThmTopos} is an
alternative description of the Euclidean interval, which
is better adapted to establishing the universal property
of an interval object. To motivate its definition,
we first observe that $(\Ival_D, \oplus)$ is an m-convex body.
This follows from Proposition~\ref{PropSetIt} (interpreted
in  $\CE$)
on account of the morphism
$\bigoplus: {\Ival_D}^\Nat \to \Ival_D$ defined by:
\[
\bigoplus_i x_i \:  = \:
\sum_{i \geq 0} 2^{-(i+1)} x_i \, ,
\]
using the Cauchy completeness of $\Real_D$.
The alternative description of the interval object
is as the smallest subalgebra of $(\Ival_D, \oplus)$
that is closed under $\bigoplus$ and contains $1$ and $-1$.
Explicitly, define
\begin{eqnarray*}
\Ival & = &
  \bigcap \{ X \subseteq \Ival_D \mid
               \text{$-1,1 \in X$ and $\bigoplus_i x_i \in X$
                     for all $x_{(-)} \in X^\Nat$}
                     \}.
\end{eqnarray*}
Then $\Ival$ is itself closed under $\bigoplus$, hence also
under binary $\oplus$,
because $x \oplus y = \bigoplus(x,y,y,y, \dots)$. Also, $\Ival \subseteq \Ival_E$
because, by the Cauchy completeness of $\Real_E$, it holds
that $\Ival_E$ is closed under $\bigoplus$.

Theorem~\ref{ThmTopos} is an immediate consequence of
the proposition below.
\begin{prop}
\label{PropThreeLines}
\leavevmode
\begin{enumerate}
\item $\Ival$ is a magnifiable symmetric subalgebra of $\Ival_D$.

\item
$(\Ival, \oplus,-1,1)$ is an interval object in $\CE$.

\item $\Ival = \Ival_E$.
\end{enumerate}
\end{prop}
We now embark on the proof of the above proposition.
% The bulk of the work will go into
% (simultaneously) establishing statements 1 and 2.
The whole proof is structured around its crucial use of Pataraia's
fixed-point theorem \cite{pataraia:pssl}, which we now
recall.

Let $(X , \leq)$ be a poset in $\CE$
(i.e. $\leq$ is a subobject
of $X \times X$ satisfying the usual axioms for a non-strict
partial order, expressed in the internal logic of $\CE$). Internally in $\CE$,
we say that a subobject $D \subseteq X$
(i.e. an element $D \in \CP X$) is
{\em directed} if it is inhabited (i.e. there exists some $x \in D$)
and, for any $x, y \in D$ there exists $z \in D$
with $x \leq z \geq y$.
(Because we are working intuitionistically,
the condition of being inhabited is stronger than the condition
of being nonempty.)
We say that $X$ is a {\em directed-complete
partial order (dcpo)}
if, internally in $\CE$,  every directed subobject $D \subseteq X$
has a least upper bound (lub) in $X$.
We say that an endofunction $f$ on a poset $X$ is
\emph{monotone} if $x \leq y$ implies
$f(x) \leq f(y)$ and {\em inflationary} if $x \leq f(x)$.
\begin{theorem}[Pataraia \cite{pataraia:pssl}]
\label{PropInfFix}
If $X$ is an inhabited dcpo then, internally in $\CE$,
every monotone inflationary function on $X$
has a fixed point.
\end{theorem}
\begin{proof}
Let $I$ be the object of monotone inflationary endofunctions on $X$
ordered pointwise. It is easily checked that $I$ is a dcpo:
lubs can be constructed pointwise because the image of a directed set
(of functions) under a monotone function (application to
an argument) is directed, and the lub so constructed is indeed
inflationary. Moreover, $I$ is itself directed: the identity is in $I$;
and, given $f,g \in I$,
it holds that $f \leq f \circ g \geq g$. Therefore $I$ considered as
a subobject of itself has a lub, i.e. I has a maximum element, $t$.
Then, for any $f \in I$ we have that $f \circ t = t$. Thus, given
any element $x \in X$, it holds that $t(x)$ is a common fixed point for
all monotone inflationary functions on $X$.
\end{proof}
Pataraia's result has  many applications. For example,
the authors have previously applied it to
the construction of initial algebras of endofunctors
\cite[Theorem~5]{simpson:classes}, and to a localic version
of the Hofmann--Mislove Theorem \cite{escardo:hofmann}.
We now proceed with our third application, a proof
of Proposition~\ref{PropThreeLines}.

\subsection{Proof of Proposition~\ref{PropThreeLines}}
\label{SubSecPropThreeLines}

Let $(A,m)$ be an m-convex body in $\CE$ with
global points $a,b \in A$.
Internally in $\CE$, we say that a subalgebra
$X \subseteq \Ival_D$ is \emph{$(A,m,a,b)$-initial}
if there exists a unique
homomorphism $f_X \colon X \to A$ from
$(X, \oplus)$ to $(A,m)$ such that $f_X(-1) = a$ and
$f_X(1) = b$.
Consider the subobject $\CF_{(A,m,a,b)} \subseteq \Power(\Ival)$
(henceforth just $\CF$) consisting (internally)
of those  subobjects $X$ of $\Ival$ that satisfy:
\begin{enumerate}
\item $X$ is a magnifiable symmetric subalgebra of $\Ival_D$; and
\item $X$ is $(A,m,a,b)$-initial.
% ; and \item $X \subseteq \Ival_E$.
\end{enumerate}

\begin{lemma}
Internally in $\CE$,
we have that $(\CF, \subseteq)$ is a dcpo with least element.
\end{lemma}
\begin{proof}
Let $\Rat_d$ be the subobject of $\Rat$ of dyadic rationals.
We write $\Confine(\Rat_d)$ for $\Rat_d \cap \Ival_D$.
Then $\Confine(\Rat_d)$ is the least element of $\CF$. In
particular, it
is $(A,m,a,b)$-initial, using Proposition~\ref{PropDyadic}
internalized in $\CE$, which is possible as the proof is constructive.

Suppose $\CD$ is a directed subset of $\CF$. We show that
$X = \bigcup \CD$ is in $\CF$.
Easily, $X$ is a magnifiable symmetric subalgebra
of $\Ival_D$ (it is closed under $\oplus$ because $\CD$ is directed).
It remains to show that it is $(A,m,a,b)$-initial.

Define $f_X: X \to A$ by mapping $x \in X$ to $f_Y(x)$, where
$Y$ is any element of $\CD$ containing $x$
and $f_Y:Y \to A$ is the unique homomorphism
from $(Y,\oplus,-1,1)$ to $(A,m,a,b)$. This is uniquely
determined because, if $x \in Y \in \CD$ and $x \in Y' \in \CD$,
then, as $\CD$ is directed, there exists $Z \in \CD$ with
$Y \subseteq Z \supseteq Y'$. But then $f_Z: Z \to A$ restricts to
homomorphisms from $Y$ to $A$ and from $Y'$ to $A$.
But $f_Y$ and $f_{Y'}$ are the unique such, so
$f_Y(x) = f_Z(x) = f_{Y'}(x)$.

Similarly, $f_X$ is unique, because given any
homomorphism $g: X \to A$ and $x \in X$, we have
$x \in Y$ for some $Y \in \CD$. Then $g$ restricts to a
homomorphism from $Y$ to $A$, and $f_Y$ is the unique such,
so indeed $g(x) = f_Y(x) = f_X(x)$.
\end{proof}

Define $\Phun : \Power(\Ival_D) \to \Power(\Ival_D)$ by
\begin{equation}
\label{EqnPhun}
\Phun(X) \: =  \: \left\{\bigoplus_i  x_i \mid
   x_{(-)} \in X^\Nat\right\}\, .
\end{equation}

\begin{lemma}
\label{LemPIF}
Internally in $\CE$, the map
$\Phun$ restricts to a monotone inflationary function
on $(\CF, \subseteq)$.
\end{lemma}
The somewhat technical proof of this lemma is
delayed to {\S}\ref{SubSecPIF} below.

\begin{lemma}
\label{LemApplyFix} \label{LemIsMag}
$\Ival \in \CF$, hence $\Ival$ is magnifiable.
\end{lemma}
\begin{proof}
By Pataraia's Theorem~\ref{PropInfFix}, the function
$\Phun: \CF \to \CF$ has a fixed point $X \in \CF$.
Thus $X \subseteq \Ival$, by the definition of $\CF$.
Also $X = \Phun(X)$. Therefore $X$ is a subobject of $\Ival_D$ that
contains $1,-1$ and is closed under $\bigoplus$. But
$\Ival$ was defined as the smallest such subset of $\Ival_D$.
Therefore $\Ival \subseteq X$. Thus $\Ival = X$ and hence
$\Ival \in \CF$.
\end{proof}

\begin{lemma} \label{lemma:Ical:cc}
  $\Ival$ is Cauchy complete.
\end{lemma}
\begin{proof}
For any Cauchy sequence in $\Ival$ we can extract a subsequence $\alpha_{(-)} \in \Ival^\Nat$ with
\(
\text{$|\alpha_{i+1} -  \alpha_i | \leq 2^{-(i+1)}$ for all $i$}.
\)
Then the limit of $\alpha$ can be calculated as
\[
\mathit{\:\lim\:} \alpha \: = \:
2 \bigoplus (\alpha_0,\, 2(\alpha_1 - \alpha_0),\, 4(\alpha_2 - \alpha_1),\, \dots ,\, 2^n(\alpha_{n}-\alpha_{n-1}) ,\, \dots).
\]
By magnifiability, $\Ival$ is closed under (truncated) subtraction
and multiplication by any $2^n$. It is also closed under
$\bigoplus$. Thus indeed $\mathit{\:\lim\:} \alpha  \in \Ival$.
\end{proof}

\begin{lemma}
$\Ival = \Ival_E$.
\end{lemma}
\begin{proof}
We know that $\Ival \subseteq \Ival_E$, so we prove
the converse inclusion.

For that purpose, we first show that $\Real_E \subseteq \Real_\Ival$, where
\[ \Real_\Ival =
\{2^n x \in \Real_D \mid \text{$x \in \Ival$ and $n \in \Nat$}\}.\]
As every rational in $\Ival_D$
can be defined by a binary sequence $\bigoplus_i x_i$ where
each $x_i$ is $-1$ or $1$, it follows that
$\Real_\Ival$ contains all rationals.
% We next show that
% $\Real_\Ival$ is Cauchy complete.
Also, for any Cauchy
sequence in $\Real_\Ival$,
we can obtain every element of the sequence as $2^n x_i$ for
a fixed~$n$, yielding an associated Cauchy sequence $(x_i)$ in $\Ival$.
By the Cauchy completeness of $\Ival$ (Lemma~\ref{lemma:Ical:cc}), we have that
$\lim_i x_i \in \Ival$, hence $\lim_i 2^n x_i \in \Real_\Ival$. That is, $\Real_\Ival$ is Cauchy complete.
Therefore, by the definition of $\Real_E$, it holds
that $\Real_E \subseteq \Real_\Ival$.

Finally, to show that $\Ival_E \subseteq \Ival$, take any
$y \in \Ival_E$. By the above, $y \in \Real_\Ival$, so
there exists $x \in \Ival$ such that
$y = 2^n x$ for some $n$. But then $y \in \Ival$ because
$\Ival$ is magnifiable.
\end{proof}

\begin{lemma}
\label{LemIsIO}
$(\Ival, \oplus,-1,1)$ is an interval object in $\CE$.
\end{lemma}
\begin{proof}
Let $(A,m)$ be any m-convex body in $\CE$ with points
$a,b \in A$. Then, by Lemma~\ref{LemApplyFix}, we have
that $\Ival \in \CF_{(A,m,a,b)}$. Thus, by the definition
of $\CF_{(A,m,a,b)}$ there indeed exists a unique
homomorphism $f \colon \Ival \to A$ from
$(\Ival, \oplus)$ to $(A,m)$ such that $f(-1) = a$ and
$f(1) = b$.
\end{proof}
Together, Lemmas~\ref{LemIsMag}--\ref{LemIsIO}
prove Proposition~\ref{PropThreeLines}, and hence
Theorem~\ref{ThmTopos}, that is, $(\Ival_E, \oplus, -1, 1)$ is an interval object in $\CE$, and, moreover, is magnifiable.

\subsection{Proof of Lemma~\ref{LemPIF}}
\label{SubSecPIF}

We now give the postponed proof of Lemma~\ref{LemPIF}, namely that internally in $\CE$, the map
$\Phun$ restricts to a monotone inflationary function
on $(\CF, \subseteq)$.
All lemmas and proofs in this subsection are
to be interpreted in the internal logic of $\CE$.

\begin{lemma}
\label{LemInflate}
$\Phun$ is a monotone inflationary
endofunction on
$(\Power(\Ival_D), \subseteq)$, and restricts to a monotone inflationary
endofunction on $(\Power(\Ival), \subseteq)$.
\end{lemma}

\begin{proof} Immediate from the definitions of $\Phun$
and $\Ival$.
\end{proof}

\begin{lemma}
\label{LemTAZ}
If $X \subseteq \Ival_D$ is a symmetric subalgebra then
so is $\Phun(X) \subseteq \Ival_D$.
\end{lemma}
\begin{proof}
Trivially, $1,-1 \in \Phun(X)$. For closure under $\oplus$,
suppose $x = \bigoplus_i x_i$ and $y = \bigoplus_i y_i$
(here $x,y \in \Phun(X)$ and $x_i,y_i \in X$). Then
$x \oplus y = \bigoplus_i (x_i \oplus y_i)$, so indeed
$x \oplus y \in \Phun(X)$ because each $x_i \oplus y_i \in X$.
Similarly, ${-x} = \bigoplus_i {-x_i}$, so $\Phun(X)$ is
a symmetric subalgebra of $\Ival_D$.
\end{proof}

\begin{lemma}
\label{LemTAB}
If $X \subseteq \Ival_D$ is an
$(A,m,a,b)$-initial magnifiable symmetric subalgebra
then the subalgebra
$\Phun(X) \subseteq \Ival_D$ is $(A,m,a,b)$-initial.
\end{lemma}
\begin{proof}
Let $f\colon X \to A$ be the unique homomorphism
from $X$.
We obtain the required
$f' \colon \Phun(X) \to A$ by defining,
for $x_0,x_1, \dots \in X$,
\begin{eqnarray*}
f'\left(\bigoplus_i (x_i)\right) & = & M_i(f(x_i)).
\end{eqnarray*}
To show that $f'$ is well-defined,
we must show that
$\bigoplus_i (x_i) = \bigoplus_i (x'_i)$ implies
$M_i(f(x_i)) = M_i(f(x'_i))$. Suppose that
$\bigoplus_i (x_i) = \bigoplus_i (x'_i) = x$.
Then, for any $n \geq 0$,
$|\oplus_n(x_0, \dots, x_{n-1}, 0) - x|  \leq  2^{-n}$ and
$|\oplus_n(x'_0, \dots, x'_{n-1}, 0) - x|   \leq  2^{-n}$.
Thus, defining $d_n = \oplus_n(x_0, \dots, x_{n-1}, 0) -
                     \oplus_n(x'_0, \dots, x'_{n-1}, 0)$, it holds that
$|d_n| \leq 2^{-(n-1)}$. Because $X$ is magnifiable,
$d_n \in X$ and hence
$2^{-(n-1)} d_n \in X$.
Let $m_n$ and $\oplus_n$ be the $n$-ary operations derived from the binary $m$ and $\oplus$, as in Section~\ref{SecItCan}.
Then
\[
\oplus_n(x_0,\dots, x_{n-1}, -2^{-(n-1)} d_n) \: = \:
\oplus_n(x'_0,\dots,x'_{n-1}, 2^{-(n-1)} d_n),
\]
where all values are in $X$. As $f$ is a homomorphism
from $X$ to $A$,
\[
m_n(f(x_0),\dots, f(x_{n-1}), f(-2^{-(n-1)} d_n)) \: = \:
m_n(f(x'_0),\dots, f(x'_{n-1}), f(2^{-(n-1)} d_n)),
\]
and this holds for any $n$. Thus, by  the approximation
property of $A$ (Proposition~\ref{PropSetApprox}),
$M_i(f(x_i)) = M_i(f(x'_i))$ as required.

Trivially $f'(-1) = f(-1) = a$ and $f'(1) = f(1) = b$.
To show that $f'$ is a homomorphism,
take any $x = \bigoplus_i(x_i)$
and $y = \bigoplus_i(y_i)$ in $\Phun(X)$ (with each $x_i, x'_i \in
X$). Then indeed
\begin{align*}
f'(x \oplus y) & =  f'\left(\left(\bigoplus_i x_i\right) \oplus \left(\bigoplus_i y_i\right)\right)
\\
& =  f'\left(\bigoplus_i \left(x_i \oplus y_i\right)\right)  \\
& =  M_i(f(x_i \oplus y_i)) & &  \text{definition of $f'$} \\
& =  M_i(m(f(x_i),\, f(y_i)))
   & &  \text{homomorphism property of  $f$} \\
& =  m(M_i(f(x_i)),\,M_i(f(y_i)))
   & &  \text{Proposition~\ref{PropSetEqns}} \\
& =  m\left(f'\left(\bigoplus_i x_i\right),\,f'\left(\bigoplus_i y_i\right)\right)
   & & \text{definition of $f'$} \\
& =  m(f'(x),f'(y)).
\end{align*}

It remains to show that $f'$ is unique.
Suppose that $g: \Phun(X) \to A$ is another homomorphism
with $g(-1) = a$ and $g(1)=b$.
Then the restriction of $g$ to $X$ is
also such a homomorphism, so $g(x) = f(x)$ for any
$x \in X$.
Thus, for $x_0, x_1, \dots \in X$,
\begin{align*}
g\left(\bigoplus_{i} x_{i+n}\right) & =
g\left(x_n \oplus \left(\bigoplus_{i} x_{i+n+1}\right)\right)  \\
& =  m\left(g(x_n), g\left( \bigoplus_{i} x_{i+n+1}\right)\right)
& & \text{homomorphism property of $g$} \\
& =  m\left(f(x_n), g\left(\bigoplus_{i} x_{i+n+1}\right)\right),
\end{align*}
and this holds for every $n \geq 0$. So
\begin{align*}
g\left(\bigoplus_{i} x_i\right) & =
M_i \, f(x_i) & & \text{canonicity property of Proposition~\ref{unfolding:canonicity}} \\
& =  f'\left(\bigoplus_{i} x_i\right)
   & &  \text{definition of $f'$.}
\end{align*}
Thus indeed $g=f'$.
\end{proof}

\begin{lemma}
\label{LemMag}
If $X \subseteq \Ival_D$ is a magnifiable symmetric subalgebra then
so is $\Phun(X) \subseteq \Ival_D$.
\end{lemma}
\begin{proof}
By Lemma~\ref{LemTAZ}, we just have to show that
$\Phun(X)$ is magnifiable. The necessity of truncating all operations
to the interval makes the proof somewhat technical.

In this proof, the symbol $\oplus$ denotes the average operation on Dedekind reals, to begin with in the definition of $y_{n+1}$ below.

We first establish:
\begin{equation}
\label{EqnEnRoute}
\mbox{for all $u \in \Phun(X)$ and $y \in X$,
it holds that $\Confine(u + 2y) \in \Phun(X)$}
\end{equation}
For this, take any $u = \bigoplus_i x_i$ in $\Phun(X)$, with
each $x_i \in X$,  and $y \in X$.
Define
\begin{align*}
y_0 & = y \\
y_{n+1} & = \Confine((x_n + 4y_n) \oplus - \Confine(x_n + 4y_n)).
\end{align*}
Then each $y_n$ above is in $X$.
Indeed, for $z,w \in X$, we have
\[ \Confine(z + 4w) = \Confine (8 (z/4 \oplus w)),\]
so $\Confine(z + 4w) \in X$ by magnifiability, whence, by
a similar argument using magnifiability,
$\Confine((z + 4w) \oplus {-\Confine(z + 4w)}) \in X$.
% because
% $\Confine((z + w) \oplus {-\Confine(z + 4w)}) =
% \Confine(8 \times ((z/4 \oplus w) \oplus \Confine(z + 4w)/8))$,
% which is again in $X$ by magnifiability.

To prove (\ref{EqnEnRoute}), it suffices to verify
that
$\Confine((\bigoplus_{i} x_i) + 2y) =
\bigoplus_i \Confine(x_i + 4y_i)$, as the right-hand side is
clearly in $\Phun(X)$.
Below we show that the equation
\begin{eqnarray}
\label{EqnFurtherEnRoute}
\Confine((x \oplus z) + 2y) & = &
\Confine(x + 4y) \: \oplus \:
\Confine(z + 2\Confine((x + 4y) \oplus {-\Confine(x + 4y)})),
\end{eqnarray}
holds for all $x,y,z \in \Ival_D$. By this equation, we have,
for any $j \geq 0$,
\begin{align*}
\Confine & \left( \left(\bigoplus_{i} x_{i+j}\right) + 2y_j\right) \\
& = \Confine\left(\left(x_j \oplus \bigoplus_{i} x_{i+j+1}\right) + 2y_j\right) \\
& = \Confine(x_j + 4y_j) \: \oplus \:
\Confine\left(\left(\bigoplus_{i} x_{i+j+1}\right)+
2\Confine((x_j + 4y_j) \oplus {-\Confine(x_j + 4y_j)})\right)
& & \text{by~(\ref{EqnFurtherEnRoute})} \\
& = \Confine(x_j + 4y_j) \: \oplus \:
\Confine\left(\left(\bigoplus_{i} x_{i+j+1}\right) + 2y_{j+1}\right).
\end{align*}
Thus, by iterativity,
\[ \Confine\left(\left(\bigoplus_{i} x_{i}\right) + 2y_0\right) =
\bigoplus_j \Confine(x_j + 4y_j),\]
i.e.\
$ \Confine\left(\left(\bigoplus_{i} x_i\right) + 2y\right) =
\bigoplus_i \Confine(x_i + 4y_i)$ as required.

To verify (\ref{EqnFurtherEnRoute}), we use the separatedness
of $\Ival_D$,  which allows us to check the equation by
splitting into cases: (i) $3 \leq x + 4y \leq 5$,
(ii) $1 \leq x + 4y \leq 3$,
(iii) $-1 \leq x + 4y \leq 1$,
(iv) $-3 \leq x + 4y \leq -1$,
(v) $-5 \leq x + 4y \leq -3$.
In each case, the verification is routine. We check case (ii).
Referring to (\ref{EqnFurtherEnRoute}), we have
\begin{align*}
\text{r.h.s.} & =
1 \: \oplus \: \Confine(z + 2 \Confine((x + 4y) \oplus -1))
& & \text{ because $1 \leq x+4y$} \\
& =  1 \: \oplus \: \Confine(z + 2 ((x + 4y) \oplus -1))
& & \text{because $x+4y \leq 3$} \\
& =  1 \: \oplus \: \Confine(x + z + 4y -1)  \\
& = 1 \: \oplus \: \Confine(2((x \oplus z) + 2y) - 1).
\end{align*}
There are now two subcases.
If $(x \oplus z) + 2y \geq 1$ then
\[
1  \oplus  \Confine(2((x \oplus z) + 2y) - 1)   = 1 =
\Confine((x \oplus z) + 2y).
\]
If instead
$0 \leq (x \oplus z) + 2y \leq 1$ then
\begin{align*}
1  \oplus  \Confine(2((x \oplus z) + 2y) - 1) & =
1  \oplus  (2((x \oplus z) + 2y) - 1) \\
& = (x \oplus z) + 2y \\
& = \Confine((x \oplus z) + 2y).
\end{align*}
Together with cases (i) and (iii)--(v),
this completes the verification of (\ref{EqnFurtherEnRoute})
and hence of (\ref{EqnEnRoute}).

Finally, we show that $\Phun(X)$ is indeed magnifiable.
Suppose that $u \in \Phun(X)$. Then
$u = \bigoplus_{i} x_i$ where $x_i \in X$.
We must show that $\Confine(2u) \in \Phun(X)$.

Below we verify that the equation
\begin{eqnarray}
\label{EqnAlmostThere}
\Confine( x + (y \oplus z)) & = &
\Confine(2x + y) \: \oplus \:
\Confine(z + 2\Confine((2x + y) \oplus {-\Confine(2x + y)}))
\end{eqnarray}
holds for all $x,y,z \in \Ival_D$. By this equation, we have:
\begin{align*}
\Confine(2u) & =  \Confine(2 \bigoplus_{i} x_i)  \\
& = \Confine(x_0 + (x_1 \oplus \bigoplus_{i} x_{i+2}))  \\
& =  \Confine(2x_0 + x_1) \: \oplus \:
\Confine((\bigoplus_{i} x_{i+2}) +
2\Confine((2x_0 + x_1) \oplus {-\Confine(2x_0 + x_1)})) & &
\text{by~(\ref{EqnAlmostThere}).}
\end{align*}
Using the magnifiability of $X$, it is straightforward to show
that $\Confine(2x_0 + x_1) \in X$
and then that
$\Confine((2x_0 + x_1) \oplus {-\Confine(2x_0 + x_1)}) \in X$.
Hence, by (\ref{EqnEnRoute}), it holds that
$\Confine((\bigoplus_{i} x_{i+2}) +
2\Confine((2x_0 + x_1) \oplus {-\Confine(2x_0 + x_1)})) \in \Phun(X)$.
So, by the above expansion of  $\Confine(2u)$,
indeed $\Confine(2u) \in \Phun(X)$.

It remains to verify (\ref{EqnAlmostThere}). This is done by
splitting into
three cases: (i) $1 \leq 2x + y \leq 3$;
(ii) $-1 \leq 2x + y \leq 1$; (iii) $-3 \leq 2x + y \leq -1$.
Again, the verification is routine. This time we check case (i).
Referring to (\ref{EqnAlmostThere}), we have
\begin{align*}
\text{r.h.s.} & =
1 \: \oplus \: \Confine(z + 2 \Confine((2x + y) \oplus -1))
& & \text{because $1 \leq 2x + y$} \\
& = 1 \: \oplus \: \Confine(2x + y + z -1)
& & \text{ because $0 \leq (2x + y) \oplus -1 \leq 1$} \\
& = 1 \: \oplus \: \Confine(2(x + (y \oplus z)) - 1).
\end{align*}
There are now two subcases. When $x + (y \oplus z) \geq 1$
then $1 \oplus  \Confine(2(x + (y \oplus z)) -1)  = 1 =
\Confine(x + (y \oplus z))$. If instead
$0 \leq x + (y \oplus z) \leq 1$ then
$1 \oplus  \Confine(2(x + (y \oplus z)) -1) =
1 \oplus  (2(x + (y \oplus z)) -1) = x + (y \oplus z)
= \Confine(x + (y \oplus z))$.
\end{proof}

Taken together, Lemmas~\ref{LemInflate}--\ref{LemMag}
provide a proof of Lemma~\ref{LemPIF}.

%% file: open.tex
\section{Discussion, questions and open problems} \label{conclusion}

Our notion of interval object is motivated by Lawvere's concept of
\emph{natural numbers object (nno)}, see Definition~\ref{DefNno}. The
notion of nno gives a way of identifying the correct natural numbers
structure in an arbitrary category.  Our definitions perform the same
task for the interval.  Such an enterprise has been attempted before
in many different ways.  In categories with rich \emph{internal
  logics}, such as \emph{elementary toposes}, various (non-equivalent)
logical definitions of the real line are available, see, for example,
\cite[{\S}D4.7]{johnstone:elephant:2}.  The idea of using algebraic
structure instead was first investigated by
Higgs~\cite{higgs:universal}, who defined \emph{magnitude modules}
(with a nullary operation $0$, a unary operation $h$, and an
$\omega$-ary operation $\sum$ subject to suitable axioms) and
characterised the interval $[0,\infty]$ in the category of sets as the
free magnitude module on one generator, with $h(x)=x/2$ and sums as
ordinary sums. Other authors have proposed \emph{coalgebraic}
definitions; for example,~\cite{MR2000d:03174}.  Peter Freyd, in
particular, obtained a characterisation, valid in any elementary
topos, of a closed interval of reals as a final
coalgebra~\cite[Theorem D4.7.17]{johnstone:elephant:2}.  Although his
theorem succeeds in characterising intervals of Dedekind reals, which
are often considered the real numbers of choice in a topos, its very
formulation relies on strong properties of the internal logic of a
topos.  In contrast, as we have emphasised throughout the paper, the
formulation of our universal property requires only a category with
finite products. Also, our approach has the distinctive property of
being based on direct geometric conceptions of the interval and real
line. Furthermore, our universal property provides a direct means for
defining the arithmetic operations and other functions between
intervals of real numbers, as described in
Section~\ref{interval:objects:section}.  This line was pursued further
in the original conference version of the
paper~\cite{escardo:simpson:interval:lics}, where we proposed a notion
of \emph{primitive interval function}, a sort of primitive recursive
function on the interval~$[-1,1]$, based on our universal
property. In~\cite{escardo:simpson:tlca2014}, we developed this
further, in the context of an extension of G\"odel's system T with a
primitive interval type, showing that the resulting type theory provides an extensional characterisation of the
class of all functions on the interval that are definable in system T
using an intensional representation of reals (the signed digit
representation).

One can consider Higgs' magnitude modules~\cite{higgs:universal} in a
topological setting, in which the unary and $\omega$-ary operations
are required to be continuous, and realising that the continuity
requirement on the countable-sum operation rules out the Euclidean
topology.  The free topological magnitude algebra on one generator is
presumably $[0,\infty]$ with the topology of lower semicontinuity, but
we have not checked the details.  The failure of continuity with
respect to the Euclidean topology is repaired, in this paper, by the
use of the iterated midpoint operation, leading to the results of
Section~\ref{spaces}.

We remark that useful topologies, different from
the Euclidean topology, can be obtained in our setting by considering
free midpoint-convex bodies over more exotic generating spaces. For
example, the midpoint-convex body freely generated by the Sierpi\'nski
space (with one isolated point and one limit point) is the unit
interval with the topology of lower semicontinuity (which is the same
as the Scott topology), and the midpoint-convex body freely generated
by the poset with one minimal and two maximal elements, with the Scott
topology, is the unit interval domain (closed subintervals of
the unit interval ordered by reverse inclusion) with the Scott
topology.  %We have not included these results in the paper for space reasons.

%Higgs~\cite{higgs:universal} defined \emph{magnitude modules} (with a
%nullary operation $0$, a unary operation $h$, and an $\omega$-ary
%operation $\sum$ subject to suitable axioms) and characterized the
%interval $[0,\infty]$ in the category of sets as the free magnitude
%module on one generator, with $h(x)=x/2$ and sums as ordinary
%sums. One can consider adapt to a topological
%
% \emph{topological} magnitude modules, Higgs~\cite{higgs:universal}where the
%unary and $\omega$-ary operations are required to be continuous, and
%ask whether there is a free topological magnitude algebra on one
%generator. If it does, it is certainly not $[0,\infty]$ with the
%Euclidean topology, because countable summation is not a continuous
%operation.

Section~\ref{sets} of this paper characterises the interval in the
category of sets (under classical logic), not just as the free
midpoint-convex body on two generators, but more generally as the free
iterative midpoint object on two generators. That is, the
cancellation property of m-convex bodies does not need to be
assumed. This is achieved via Theorem~\ref{ThmSuperconvex}, which
shows an equivalence between iterative midpoint sets and iterative
superconvex sets. This result implies that supermidpoint sets are
equivalent to superconvex sets if either iterativity or cancellativity
is assumed on both sides.  It would be interesting to know whether, in
fact, supermidpoint sets are equivalent to superconvex sets in full
generality, without either of the two non-equational assumptions.  This
can only be true if equation~\eqref{equation:flatten}
from~Proposition~\ref{PropFlatten} holds for arbitrary supermidpoint
objects. Independently of this, we conjecture that
if~\eqref{equation:flatten} is added as an additional equation to the
supermidpoint axioms (Definition~\ref{def:supermidpoint}) then one obtains a complete
equational axiomatisation of superconvex sets in terms of the iterated midpoint operation alone. We remark that the paper~\cite{Marra_2017}  presents a finite equational axiomatisation of an algebraic theory containing the iterated midpoint operation, where the algebraic theory is the theory of continuous maps between infinite dimensional cubes $[0,1]^\kappa$, studied by Isbell~\cite{Isbell1982}, who was the first to realise the importance of the iterated midpoint operation in that context. The finite axiomatisation in~\cite{Marra_2017} does not answer our conjecture, because it crucially relies on the presence of additional non-affine operations in the theory of~\cite{Isbell1982}.

As mentioned
in Section~\ref{SecIntro}, our characterisation of interval objects in
elementary toposes, Theorem~\ref{ThmTopos}(\ref{ThmTopos:ii}), can be
used to identify interval objects in several categories of interest
that come with natural product-preserving embeddings in toposes.
Several motivating such examples come from computability theory.  For
example, in the category of \emph{modest sets}
% and the more general category of \emph{assemblies},
over Kleene's first partial combinatory algebra $K_1$~\cite{vanOosten}, which is a
category modelling computability, the interval object is given by an
interval $[a,b]$ of computable real numbers, equipped with its natural
modest set structure.  This follows as a consequence of
Theorem~\ref{ThmTopos}, because closed intervals $[a,b]$ of Cauchy
reals in Hyland's effective topos~\cite{HylandEff}, which validates countable choice,
have exactly this form, and their universal property as an interval
object in the topos is necessarily retained in the subcategory of
modest sets.
% which is embedded via a product-preserving inclusion functor.
Similarly, the interval object in the category of continuous maps
between \emph{Baire space representations}, considered in Weihrauch's
theory of \emph{Type Two Effectivity}~\cite{weihrauch:2000}, is a closed interval $[a,b]$ of
real numbers equipped with an \emph{admissible representation}. Once
again, this is a consequence of Theorem~\ref{ThmTopos}, this time via
an embedding in the realizability topos over Kleene's second partial
combinatory algebra $K_2$~\cite{vanOosten}. Furthermore, if one restricts to the
subcategory of computable maps between Baire space representations,
then the interval object is a closed interval $[a,b]$ of real numbers
with a \emph{computably admissible representation}~\cite{weihrauch:2000}; as shown, this
time, by an embedding in the \emph{relative realizability topos} of
computable maps over $K_2$~\cite{bauer}.

%% We don't need to say this:
% Looking further afield to results that do not fall out directly from
% the present work,
We conjectured in~\cite{escardo:simpson:interval:lics} that the
localic unit interval should be an interval object in the category of
locales. This was proved to hold for locales internal to any
elementary topos (with natural numbers object) by Vickers~\cite{vickers:localic:es:reals}.
This result provides
a way of reconciling our notion of interval object with the Dedekind reals in any elementary topos; namely,
the Dedekind interval arises as the points of the locale characterised as an interval object.

The HoTT Book~\cite{hottbook} gives a higher-inductive-recursive
definition of real numbers, and conjectures that their unit interval
is an interval object in our sense, which was proved by
Booij~\cite{Booij17}, who also showed in this setting that the
Euclidean reals are the smallest Cauchy complete subset of the
Dedekind reals containing the rationals. By Shulman's result that Book
HoTT can be interpreted in any
$\infty$-topos~\cite{shulman2019infty1toposesstrictunivalentuniverses},
it follows that this holds in any $\infty$-topos, generalizing our
results of Section~\ref{topos}.

%\todo[inline]{Rename Euclidean reals.} % we changed our mind.

%% Done:
% \todo[inline]{Rename \emph{convex body}.}

\section*{Acknowledgements}

We thank Phil Scott for providing connections with this work, by
drawing our attention to Higgs' paper~\cite{higgs:universal}, which
motivated our definition~\ref{def:supermidpoint} of supermidpoint
object. We would also like to thank Klaus Keimel and Thomas Streicher
for numerous discussions from which this work greatly benefited. In
particular, Klaus drew our attention to the work of
K\"onig~\cite{koenig:superconvex} on superconvex spaces, after he saw
our definition~\ref{def:supermidpoint}, of which we made profitable
use in this paper. We also thank the anonymous reviewer for their helpful suggestions.